\documentclass[12pt, reqno]{amsart}
\usepackage{amssymb, mathrsfs}
\usepackage{latexsym}
\usepackage{amsmath}
\usepackage{amsthm}
\usepackage{amsfonts}
\usepackage{dsfont, upgreek}
\usepackage{mathtools}
\usepackage{epsfig}
\usepackage{amscd}
\usepackage{graphicx}
\usepackage{tikz-cd}
\usepackage{enumitem}
\setlist[enumerate]{itemsep=0.5ex}

\usepackage{comment}

\usepackage{color}
\usepackage{tikz}
\usetikzlibrary{patterns}

\usepackage[left=1.4in,right=1.4in,top=1.2in,bottom=1.2in]{geometry}

\usepackage{tikz}
\usetikzlibrary{decorations.markings}
\usetikzlibrary{arrows.meta}

\usepackage{extarrows}

\usepackage{adjustbox}
\usepackage{pgfplots}
\pgfplotsset{compat=1.18}
\usepgfplotslibrary{colormaps}
\usepackage{slashed}

\usepackage[colorlinks=true,linkcolor=blue,citecolor=blue]{hyperref}

\usepackage[colorinlistoftodos,prependcaption,textsize=tiny]{todonotes}  

\usepackage[all,cmtip]{xy}
\newenvironment{plainproof}
  {\par\noindent\ignorespaces}
  {\hfill\(\square\)\par}
  
\theoremstyle{plain}
\newtheorem{theorem}{Theorem}[section]
\newtheorem{proposition}[theorem]{Proposition}
\newtheorem{lemma}[theorem]{Lemma}
\newtheorem{corollary}[theorem]{Corollary}

\newtheorem{remark}[theorem]{Remark}
\newtheorem{definition}[theorem]{Definition}

\theoremstyle{definition}
\newtheorem{example}{Example}[section]

\numberwithin{equation}{section}
\newcommand{\E}{\mathcal{E}}

\newcommand{\supp}{\mathrm{supp}}

\newcommand{\interior}[1]{%
	{\kern0pt#1}^{\mathrm{\,o}}%
}

\makeatletter
\let\save@mathaccent\mathaccent
\newcommand*\if@single[3]{%
	\setbox0\hbox{${\mathaccent"0362{#1}}^H$}%
	\setbox2\hbox{${\mathaccent"0362{\kern0pt#1}}^H$}%
	\ifdim\ht0=\ht2 #3\else #2\fi
}
\newcommand*\rel@kern[1]{\kern#1\dimexpr\macc@kerna}
\newcommand*\overbar[1]{\@ifnextchar^{{\wide@bar{#1}{0}}}{\wide@bar{#1}{1}}}
\newcommand*\wide@bar[2]{\if@single{#1}{\wide@bar@{#1}{#2}{1}}{\wide@bar@{#1}{#2}{2}}}
\newcommand*\wide@bar@[3]{%
	\begingroup
	\def\mathaccent##1##2{%
		\let\mathaccent\save@mathaccent
		\if#32 \let\macc@nucleus\first@char \fi
		\setbox\z@\hbox{$\macc@style{\macc@nucleus}_{}$}%
		\setbox\tw@\hbox{$\macc@style{\macc@nucleus}{}_{}$}%
		\dimen@\wd\tw@
		\advance\dimen@-\wd\z@
		\divide\dimen@ 3
		\@tempdima\wd\tw@
		\advance\@tempdima-\scriptspace
		\divide\@tempdima 10
		\advance\dimen@-\@tempdima
		\ifdim\dimen@>\z@ \dimen@0pt\fi
		\rel@kern{0.6}\kern-\dimen@
		\if#31
		\overline{\rel@kern{-0.6}\kern\dimen@\macc@nucleus\rel@kern{0.4}\kern\dimen@}%
		\advance\dimen@0.4\dimexpr\macc@kerna
		\let\final@kern#2%
		\ifdim\dimen@<\z@ \let\final@kern1\fi
		\if\final@kern1 \kern-\dimen@\fi
		\else
		\overline{\rel@kern{-0.6}\kern\dimen@#1}%
		\fi
	}%
	\macc@depth\@ne
	\let\math@bgroup\@empty \let\math@egroup\macc@set@skewchar
	\mathsurround\z@ \frozen@everymath{\mathgroup\macc@group\relax}%
	\macc@set@skewchar\relax
	\let\mathaccentV\macc@nested@a
	\if#31
	\macc@nested@a\relax111{#1}%
	\else
	\def\gobble@till@marker##1\endmarker{}%
	\futurelet\first@char\gobble@till@marker#1\endmarker
	\ifcat\noexpand\first@char A\else
	\def\first@char{}%
	\fi
	\macc@nested@a\relax111{\first@char}%
	\fi
	\endgroup
}
\makeatother
\begin{document}

\title[Riemannian PMT With Low-Codimension Singularities]{Riemannian Positive Mass Theorem in All Dimensions in the Presence of Low-Codimension Singularities}
\author{Marcus Khuri}
\address[Marcus Khuri]{Department of Mathematics, Stony Brook University}
\email{marcus.khuri@stonybrook.edu}
\thanks{}
\author{Jian Wang}
\address[Jian Wang]{State Key Laboratory of Mathematical Sciences\\
Academy of Mathematics and Systems Science, Chinese Academy of Sciences \\
Beijing 100190, China}
\email{jian.wang.4@amss.ac.cn}
\thanks{}
\author{Jinmin Wang}
\address[Jinmin Wang]{State Key Laboratory of Mathematical Sciences\\
Academy of Mathematics and Systems Science, Chinese Academy of Sciences \\
Beijing 100190, China}
\email{jinmin@amss.ac.cn}

\date{}

\thanks{Marcus Khuri acknowledges the support of NSF Grant DMS-2405045. Jinmin Wang acknowledges the support of NSFC 12501169.}

\begin{abstract}
We prove the Riemannian positive mass theorem in all dimensions for
asymptotically flat \(L^\infty\)-metrics with subcritical singular sets.
More precisely, we consider complete asymptotically flat manifolds whose
metrics are smooth away from a compact singular set of Minkowski dimension
less than \(n-3+\frac{2}{n}\), and whose scalar curvature is nonnegative on the regular set.
We show that the ADM mass of each asymptotically flat end is nonnegative, and that
the mass vanishes in some end only in the Euclidean case. For the rigidity statement,
we require additionally that the Minkowski dimension of the singular set is not larger 
than $n-3+\frac{1}{n-1}$. This gives an asymptotically flat analogue
of Schoen's codimension-three conjecture for positive scalar curvature.  The
proof combines a density theorem for singular asymptotically flat metrics,
capacity estimates across the singular set, conformal blow-up inspired by 
Bi--Hao--He--Shi--Zhu \cite{BiHaoHeShiZhu26}, and a \(\mu\)-bubble dimension-descent 
argument adapted from Brendle--Wang \cite{BrendleWang2026}.
\end{abstract}
\maketitle

\section{Introduction}
\label{sec1}

The positive mass theorem is one of the central results in scalar curvature
geometry and mathematical relativity.  In its Riemannian form it asserts that a
complete asymptotically flat manifold with nonnegative scalar curvature has
nonnegative ADM mass, with equality only for Euclidean space.  The theorem was
proved by Schoen and Yau in dimensions \(3\leq n\leq 7\) using minimal
hypersurfaces \cite{SchoenYau1979PMT,SchoenYau1981PMTII}, and by Witten in
the spin case in all dimensions using the Dirac operator \cite{Witten1981}.
In the past few years there has been renewed interest in extending
minimal-hypersurface methods beyond the classical regularity range. In particular,
we mention the approaches of Schoen--Yau \cite{SchoenYau2019MinimalSingularities} and Lohkamp
\cite{Lohkamp2006,Lohkamp2016HigherDimensionalPMTII} for arbitrary dimensions, as well as the work of Smale \cite{Smale1993} and
Chodosh--Mantoulidis--Schulze \cite{ChodoshMantoulidisSchulze2023} who proved generic regularity results in
dimensions 8, respectively \(9\) and \(10\), and the extension of this by
Chodosh--Mantoulidis--Schulze--Wang \cite{ChodoshMantoulidisSchulzeWang2025} to dimension \(11\). 
Very recently, Bi--Hao--He--Shi--Zhu \cite{BiHaoHeShiZhu26} blew up the singular set by a harmonic conformal factor to extend the Riemannian positive mass theorem up to dimension $19$, and in a breakthrough Brendle--Wang \cite{BrendleWang2026} improved this to all dimensions by using a different conformal factor based on the distance function to the singular set and a new inductive dimensional-descent scheme. These latter two results make use of Lesourd--Unger--Yau's shielding principle \cite{LesourdUngerYau2024} to control the additionally created complete ends. Versions of the spacetime positive mass theorem were also recently established in all dimensions in both the asymptotically flat and asymptotically hyperboloidal contexts, see
\cite{BrendleWang2026Spacetime,HirschKhuriLesourdZhang2026,Tsang2026InitialDataArbitraryEnds}.

The purpose of this paper is to prove a singular version of the Riemannian
positive mass theorem in all dimensions for asymptotically flat metrics which are
only \(L^\infty\) across a codimension greater than $3-\frac{2}{n}$ singular set.  Thus, an
\(L^\infty\) singular set of subcritical dimension cannot be used to hide negative
mass.

\begin{theorem}\label{A}
Let \((M^n,g)\), \(n\geq 3\), be a complete asymptotically flat
\(L^\infty\)-manifold. Suppose that the singular set \(\mathcal S\) is
compact.

\medskip
\noindent\textbf{Inequality.}
Assume that
\begin{equation}
        \dim_{\mathcal M}\mathcal S < n-3+\frac{2}{n}.
\end{equation}
If the scalar curvature of \(g\) is nonnegative on
\(M^n\setminus\mathcal S\), then the ADM mass of each asymptotically flat end
is nonnegative.

\medskip
\noindent\textbf{Rigidity.}
Assume in addition that
\begin{equation}
        \dim_{\mathcal M}\mathcal S\leq n-3+\frac{1}{n-1}.
\end{equation}
If the ADM mass of some asymptotically flat end vanishes, then \(M^n\) has
only one end and there exists a homeomorphism $\Psi:M^n\to\mathbb R^n$
such that
\begin{equation}
        \Psi|_{M^n\setminus\mathcal S}:
        M^n\setminus\mathcal S
        \longrightarrow
        \mathbb R^n\setminus\Psi(\mathcal S)
\end{equation}
is a smooth Riemannian isometry. Furthermore, the metric completion of the length space $(M^n\setminus\mathcal S,d_g)$
is isometric, as a metric space, to Euclidean space \((\mathbb R^n,d_{\delta})\).
\end{theorem}

\begin{remark}
The homeomorphism \(\Psi\) constructed in the rigidity statement is actually locally bi-Lipschitz with respect to the background metric $g_0$. However, under the present \(L^\infty\) hypotheses, \(\Psi\) need not be \(C^1\) across the singular set \(\mathcal S\), as is shown by Example~\ref{example:flat-Linfty-nonsmooth} below.
\end{remark}

\begin{remark} 
In the proof of the \textbf{Inequality} part of Theorem~\ref{A}, the Minkowski dimension assumption is used only in two places: in the density theorem, Theorem~\ref{density}, and in the blow-up argument for the singular set of an $L^\infty$-metric, Proposition~\ref{blow-up1}. Both ingredients are available under the corresponding Hausdorff dimension hypotheses. More
precisely, Theorem~\ref{density} remains valid under the assumption $\dim_{\mathcal H}\mathcal S<n-2$, as explained in Remark~\ref{density-hausdorff}. In addition,
 Proposition~\ref{blow-up1} is stated under the Hausdorff dimension assumption $\dim_{\mathcal H}\mathcal S<n-3+\frac{2}{n}.$

Consequently, the \textbf{Inequality} statement also holds if the Minkowski dimension hypothesis is replaced by the Hausdorff dimension bound
\begin{equation}
    \dim_{\mathcal H}\mathcal S < n-3+\frac{2}{n}.
\end{equation}
\end{remark}

The assumptions in Theorem~\ref{A} are essential.  In particular, the
\(L^\infty\)-hypothesis cannot be omitted; see Definition \ref{linfinity} below
for the notion of an $L^{\infty}$-manifold.  The negative mass Schwarzschild
manifold has zero scalar curvature, a singular set of Minkowski dimension zero, and negative ADM mass, but its singularity is not \(L^\infty\): the metric is not uniformly equivalent to a
smooth background metric near the singular set.  Thus, the two-sided
\(L^\infty\) control rules out precisely this type of degeneracy.  The
Minkowski dimension assumption is also sharp in spirit.  Codimension-two singularities
may carry conical angle defects, and such cone-type singularities can contribute
a negative distributional scalar curvature concentrated along the singular set;
in asymptotically flat examples this can lead to negative ADM mass despite
nonnegative scalar curvature on the regular set.  Therefore, one cannot expect a
positive mass theorem under the sole assumption that the scalar curvature is nonnegative away from a
singular set of codimension two or lower. The theorem shows that any sharp Minkowski-codimension threshold for the validity of the positive mass theorem must lie between codimensions two and 
$3-\frac{2}{n}$.  Determining the precise threshold remains an interesting open problem.

This result may be viewed as an asymptotically flat analogue of Schoen's
codimension-three conjecture for positive scalar curvature.  In the compact
setting, this conjecture predicts that singular sets of codimension at least
three should be invisible from the point of view of positive scalar curvature:
if an $L^{\infty}$-metric is smooth and has nonnegative scalar curvature
away from such a set, then the singularities should not evade the usual
obstructions to positive scalar curvature.  A formulation of this conjecture
appears, for example, in the work of Li--Mantoulidis
\cite[Conjecture 1.5]{LiMantoulidis2019}.  Theorem \ref{A} establishes a
corresponding phenomenon in the asymptotically flat setting: the presence of a
codimension-three \(L^\infty\)-singular set does not invalidate the positive mass
theorem. In other words, codimension-three \(L^\infty\)-singularities cannot hide 
negative mass in an asymptotically flat manifold.

There is a substantial literature on positive mass theorems and scalar
curvature problems for nonsmooth or singular metrics.  Miao \cite{Miao2002} proved a positive
mass theorem for metrics with corners along a hypersurface under a mean-curvature
jump condition, see also Shi--Tam \cite{ShiTam2002}; related smoothing and Ricci-flow approaches were
developed by McFeron--Szekelyhidi \cite{McFeronSzekelyhidi2012}.  Lee \cite{Lee2013} proved a
positive mass theorem for Lipschitz metrics with small singular sets, while Lee--LeFloch \cite{LeeLeFloch2015} established a spin positive mass theorem for
metrics with distributional scalar curvature in the class
\(C^0\cap W^{1,n}_{\mathrm{loc}}\).  Other singular
positive mass results include work on boundary and fill-in phenomena
\cite{HirschMiao2020,MantoulidisMiaoTam2020}, incomplete manifolds
\cite{LeeLesourdUnger2023}, manifolds with structured skeleton-type singularities \cite{LiMantoulidis2019}, isolated conical singularities
\cite{DaiSunWang2025}, and $C^0$-metrics \cite{JSZ,LeeTam,MazurowskiYao2026ContinuousMetrics}.
We also mention the related work of Burkhardt-Guim
\cite{BurkhardtGuim2024Smoothing}, who used Ricci flow to smooth
\(L^\infty\)-metrics on \(\mathbb R^n\) with nonnegative scalar curvature away
from a singular set of finite \((n-\alpha)\)-dimensional Minkowski content,
\(\alpha>2\), under a small \(L^\infty\)-closeness assumption to the Euclidean
metric.  

The compact positive scalar curvature problem with \(L^\infty\)-singularities
has also seen important recent developments. Li--Mantoulidis \cite{LiMantoulidis2019} 
treated the case of dimension three, and
Kazaras \cite{Kazaras2024} proved desingularization results in dimension
four. Dai--Sun--Wang \cite{DaiSunWang2024ConicalPSC} and Dai--Wang--Wang--Wei \cite{DaiWangWangWei2024RCD}
established the case of isolated singularities on restricted topologies, while Wang--Xie \cite{WangXie2024SphereSubsetsLinfty}
and Shi--Wang--Wu--Zhu \cite{ShiWangWeiZhu2021FillIn} obtained further positive results
for singular sets of codimension greater than or equal to $\frac{n}{2}+1$. 
Very recently and concurrently with the present paper, Wang–Wang–Xie \cite{Wang-Wang-Xie} treated the torus case when the singular set has codimension greater than $n-3+\frac{2}{n}$, while Bi–Zhu \cite{Bi-Zhu} obtained an analogous result for enlargeable manifolds under a similar hypothesis involving the Assouad dimension. The latter assumption is more restrictive than the corresponding Hausdorff or Minkowski dimension bounds, as it involves a uniform local packing condition across all centers and scales. In contrast, Cecchini--Frenck--Zeidler \cite{CecchiniFrenckZeidler2024} showed that, in high
dimensions, there are counterexamples to Schoen's conjecture with point
singularities.  These compact results show that
the codimension-three question is quite subtle. On the other hand, as Theorem \ref{A} shows,
these subtleties disappear when the conjecture is considered in the noncompact context of the positive mass theorem.

The proof of Theorem \ref{A} involves three primary components. The first is a density and deformation
argument, that perturbs the metric to one of the same type of singularities but with strictly positive 
scalar curvature away from the singular set, harmonic asymptotics, and an ADM mass that is nearly unchanged. 
The codimension-$(3-\frac{2}{n})$
assumption enters analytically partially through capacity, namely, compact sets of Minkowski
dimension less than \(n-2\) have vanishing \(W^{1,2}\)-capacity.  This permits
integration by parts across \(\mathcal S\), yields the global Sobolev inequality,
and allows one to solve the elliptic equations needed for the deformation.
The second component consists of establishing a generalization of the main positive mass-type theorem
proven by Brendle--Wang \cite[Theorem 1.5]{BrendleWang2026}, so that it applies to what we call
weak $n$-data sets. Here, the notion of weak $n$-data set primarily differs from the Brendle--Wang $n$-data set concept
by a weakening of the weighted scalar curvature integral inequality. Lastly, the third component of the proof
is inspired by the work of Bi--Hao--He--Shi--Zhu \cite{BiHaoHeShiZhu26} and Brendle--Wang \cite{BrendleWang2026}.
In particular, using the codimension threshold $3-\frac{2}{n}$, we conformally blow up the singular set with Green's functions, producing a complete smooth manifold \(M^n\setminus\mathcal S\) with the chosen asymptotically
flat end preserved and satisfying a pointwise weighted scalar curvature inequality.
A stable \(\mu\)-bubble is then constructed, which serves as a lower-dimensional asymptotically flat space carrying
an integral weighted scalar curvature inequality, placing it within the context of weak $(n-1)$-data sets
and their positive mass-type theorem.  The rigidity statement is obtained by extending the standard splitting 
argument to show that the regular part is flat, while \(L^\infty\)-control across the singular set
together with the capacity property rule out nontrivial singular behavior. However, difficulties arise due to a lack of uniform Hessian control for harmonic functions near the singular set, and to overcome this, a stronger codimension threshold of
$3-\frac{1}{n-1}$ is needed for our approach to rigidity.

The paper is organized as follows.  In Section~\ref{sec2} we develop the elliptic theory
for \(L^\infty\)-metrics with codimension-greater-than-two singular sets, including
integration by parts across \(\mathcal S\), a global Sobolev inequality, and the
density theorem producing harmonically asymptotically flat metrics.  In
Section~\ref{sec3} we blow up the singular set and construct the \(\mu\)-bubbles used in
the dimension-reduction argument.  Section~\ref{sec4} formulates and establishes the generalized 
positive mass-type inequality for weak $n$-data sets, following Brendle--Wang \cite{BrendleWang2026Spacetime}, and this is applied to obtain the inequality portion of Theorem \ref{A}.
Finally, in Section~\ref{sec5} we prove rigidity and complete the proof of the singular positive mass
theorem.

\medskip

\noindent\textbf{Acknowledgements:} MK would like to thank Sven Hirsch, Demetre Kazaras, and Yiyue Zhang for discussions.

\section{A Density Theorem for \texorpdfstring{$L^\infty$}{Linfty}-Metrics}
\label{sec2}

In this section we develop the theory of elliptic PDE on $L^\infty$-asymptotically flat manifolds with low-codimension
singularities. This is then used to achieve perturbations to positive scalar curvature and harmonic asymptotics. We begin with several definitions.

\begin{definition}\label{linfinity}
Let \(M^n\) be a smooth manifold equipped with a smooth Riemannian
metric \(g_0\).  A measurable symmetric \((0,2)\)-tensor \(g\) is called an \textup{\(L^\infty\)-metric} with respect to \(g_0\) if every point \(x\in M^n\) admits a neighborhood \(V_x\) and
a constant \(\Lambda_x\geq 1\) such that
\begin{equation}\label{Lambda}
        \Lambda_x^{-1} g_0 \leq g \leq \Lambda_x g_0
        \qquad \text{a.e. on } V_x
\end{equation}
as quadratic forms. We say that $(M^n,g)$ is a \textup{complete $L^{\infty}$-Riemannian manifold} if $g_0$ is complete and 
\eqref{Lambda} holds with a uniform constant $\Lambda\geq 1$. The \textup{regular set} of an \(L^\infty\)-metric \(g\) consists of those points having a neighborhood on which \(g\) is smooth (admits a smooth representative), and the \textup{singular set}
is the complement of the regular set.
\end{definition}


\begin{definition}\label{AFdef} 
A complete $L^\infty$-Riemannian manifold $(M^n, g)$, with $n\geq 3$, is called \textup{asymptotically flat} if there exists a compact set $\mathcal{K}\subset M^n$ such that the following properties hold.
\begin{enumerate}
\item The complement $M^n\setminus \mathcal{K}$ has finitely many components $\{\mathcal{E}_i\}_{i=1}^k$, called ends, each of which is diffeomorphic to $\mathbb{R}^n\setminus \overline{B}_1$.

\item For each $i$ the metric $g|_{\E_i}$ is smooth, and after pulling back by the diffeomorphism $\mathbb{R}^n\setminus \overline{B}_1 \rightarrow \E_i$, it asymptotes to the Euclidean metric $\delta$ in Cartesian coordinates with decay 
\begin{equation}
|\mathring{\nabla}^l( g|_{\E_i}-\delta)|_{\delta}=O(r^{-q-l}),\quad \quad l=0,1,2,
\end{equation} 
for some $q>\frac{n-2}{2}$, where $\mathring{\nabla}$ denotes covariant differentiation with respect to $\delta$.
    
\item For each $i$ the scalar curvature of the end is integrable, $R_g \in L^1(\mathcal{E}_i)$.

\end{enumerate}
\end{definition}

\begin{remark}\label{unif-bounded} 
By smoothly deforming the background metric $g_0$, it may be assumed 
that $g_0$ agrees with the Euclidean metric in each asymptotically flat end. Thus, we may take the background manifold $(M^n,g_0)$ itself to be asymptotically flat in Definition \ref{AFdef}.
\end{remark}

\begin{remark}
By definition, the singular set of an $L^\infty$-Riemannian manifold is closed. Furthermore, in the
asymptotically flat context it is contained within a compact set $\mathcal{K}$, and hence the singular
set for these type of manifolds is necessarily compact.
\end{remark}

For each asymptotically flat end, the \textit{ADM mass} is defined by
\begin{equation}
    m
    =
    \frac{1}{2(n-1)\omega_{n-1}}
    \lim_{r\to\infty}
    \int_{S_r}\sum_{i,j=1}^n
    \left(\partial_i g_{ij}-\partial_j g_{ii}\right)\nu^j\,dA_\delta ,
\end{equation}
where $\nu$ and $dA_{\delta}$ are respectively the Euclidean outward unit normal and hypersurface measure
of the coordinate sphere $S_r$. Under the above decay assumptions, and with the corresponding integrability of
the scalar curvature, this quantity is independent of the choice of
asymptotically flat coordinates and hence is a well-defined geometric invariant
of the end, as shown by Bartnik and Chru\'{s}ciel \cite{Bartnik1986,Chrusciel}.

We next recall the class of asymptotically flat metrics with harmonic
asymptotics.  This special form is useful because it simplifies the asymptotics, and
the leading coefficient in the asymptotic expansion directly determines the ADM mass.

\begin{definition}\label{HAF}
An end $(\mathcal{E}, g)$ of an asymptotically flat $(L^\infty)$-Riemannian manifold is called \textup{harmonically asymptotically flat} if the following decay\footnote{See \cite[Appendix A]{Lee} for the definition of weighted 
H\"{o}lder spaces.} holds in the asymptotic coordinate system
\begin{equation}\label{fonaoinoih}
g-\left(1+\frac{2m}{(n-2)r^{n-2}}\right)\delta\in C^{2,\alpha}_{-\mathring{q}}(\E),
\end{equation}
for some $m\in\mathbb{R}$, $\mathring{q}>n-2$, and $\alpha\in(0,1)$. In this case, $m$ is the ADM mass of $\mathcal{E}$.
\end{definition}

\begin{remark}\label{2.6}
It will be convenient in some places to extend the definition of harmonically asymptotically flat to accommodate 2-dimensional ends, by removing the mass term from \eqref{fonaoinoih}. In this situation the mass is defined to be zero.\end{remark}

The main result of this section is the following density theorem, which generalizes a well-known
result in the setting of smooth asymptotically flat manifolds \cite[Lemma 1]{Bray2001}.

\begin{theorem}\label{density}
Let \((M^n,g)\) be a complete asymptotically flat
\(L^\infty\)-manifold. Assume that the singular set \(\mathcal{S}\) satisfies the Minkowski dimension upper bound $\dim_{\mathcal M}\mathcal{S}< n-2$.
Suppose that \(R_g\geq 0\) on \(M^n\setminus \mathcal{S}\). Then for each \(\varepsilon>0\)
there exists a complete asymptotically flat \(L^\infty\)-metric
\(g'\) on \(M^n\) 
with the following properties.
\begin{enumerate}
\item It holds that $\|g'-g\|_{L^{\infty}(M^n,g_0)}<\varepsilon$.
\item \(g'\) is smooth on \(M^n\setminus \mathcal{S}\) and \(R_{g'}> 0\) on
\(M^n\setminus \mathcal{S}\).
\item On each asymptotically flat end, the new metric \(g'\) is harmonically asymptotically flat.
\item On each asymptotically flat end, if \(m'\) and \(m\) denote the respective ADM masses of \(g'\) and \(g\), then
$|m'-m|<\varepsilon$.
\end{enumerate}
\end{theorem}

\subsection{Laplace-Beltrami operator}
When the singular set $\mathcal{S}$ is of Minkowski codimension greater than two, it has vanishing $W^{1,2}$-capacity. In particular,
it is invisible to $W^{1,2}$-analysis in the sense that $W^{1,2}(M^n\setminus\mathcal{S},g)$ agrees with $W^{1,2}(M^n,g)$, and integration by parts is not disrupted.

\begin{proposition}[Integration by parts across a codimension-two${}^+$ singular set]\label{prop2.3}
Let $(M^n,g)$ be an $L^\infty$-manifold.
If the singular set $\mathcal{S}$ is compact and satisfies the Minkowski dimension upper bound
$\dim_{\mathcal M} \mathcal S < n-2$,
then the singular set has vanishing \(W^{1,2}\)-capacity with respect to \(g\).
In particular consider \(u\in W^{1,2}(M^n,g)\), and let
$v\in C^\infty(M^n\setminus \mathcal S)$
have compact support in \(M^n\) in addition to
\begin{equation}
v\in W^{1,2}(M^n,g),\qquad
|\nabla v|_g\in L^\infty(M^n\setminus \mathcal S),
\qquad
\Delta_g v\in L^2(M^n\setminus \mathcal S),
\end{equation}
then
\begin{equation}
\int_{M^n\setminus \mathcal S} u\,\Delta_g v\,dV_g
=
-\int_{M^n\setminus \mathcal S}
\langle \nabla u,\nabla v\rangle_g\,dV_g .
\end{equation}
\end{proposition}

\begin{proof}
Let $g_0$ be the background metric. Since the Minkowski dimension of $\mathcal S$ exists and is of codimension greater than two,
for every $\delta$ satisfying 
\begin{equation}
\max\{0,\dim_{\mathcal{M}}\mathcal{S}+3-n\}<\delta <1
\end{equation}
there exists a constant \(C_\delta\) such that, for all
sufficiently small \(r>0\),
\begin{equation}
\operatorname{Vol}_{g_0}\bigl(N_r(\mathcal S)\bigr)
\leq C_\delta r^{3-\delta}.
\end{equation}
Here \(N_r(\mathcal{S})\) denotes the \(r\)-neighborhood of \(\mathcal{S}\) with respect to \(g_0\).

Choose smooth cut-off functions \(\chi_r\in C_c^\infty(M^n)\) satisfying
\begin{equation}
0\leq \chi_r\leq 1,\qquad
\chi_r=1 \quad\text{on } N_r(\mathcal{S}),
\qquad
\supp\, \chi_r\subset N_{2r}(\mathcal{S}),
\end{equation}
and
\begin{equation}
|\nabla \chi_r|_{g_0}\leq \frac{C}{r}.
\end{equation}
Since \(g\) and \(g_0\) are uniformly equivalent on a fixed compact neighborhood
of \(\mathcal{S}\), there is a constant \(C\), independent of \(r\), such that
\begin{equation}
|\nabla \chi_r|_g^2\,dV_g
\leq
C |\nabla \chi_r|_{g_0}^2\,dV_{g_0}.
\end{equation}
Therefore
\begin{equation}
\int_{M^n} |\nabla \chi_r|_g^2\,dV_g
\leq
\frac{C}{r^2}
\operatorname{Vol}_{g_0}\bigl(N_{2r}(\mathcal{S})\bigr)
\leq
C_\delta r^{1-\delta}.
\end{equation}
Since \(\delta<1\), it follows that
\begin{equation}
\|\nabla \chi_r\|_{L^2(M^n,g)}\to 0
\qquad\text{as } r\to0.
\end{equation}
Thus \(\mathcal{S}\) has vanishing \(W^{1,2}\)-capacity.

Set $\eta_r=1-\chi_r$.
Then \(\eta_r u\) vanishes in a neighborhood of \(\mathcal{S}\). Hence, by integration by
parts on the smooth manifold \(M^n\setminus \mathcal{S}\), and using the compact support of
\(v\), we obtain
\begin{align}
\int_{M^n\setminus \mathcal{S}} \eta_r u\,\Delta_g v\,dV_g
&=
-\int_{M^n\setminus \mathcal{S}}
\langle \nabla(\eta_r u),\nabla v\rangle_g\,dV_g \\
&=
-\int_{M^n\setminus \mathcal{S}}
\eta_r\langle\nabla u,\nabla v\rangle_g\,dV_g
+
\int_{M^n\setminus \mathcal{S}}
u\langle\nabla \chi_r,\nabla v\rangle_g\,dV_g .\nonumber
\end{align}
We now let \(r\to0\). Since \(\eta_r\to1\) almost everywhere and
$u\,\Delta_g v\in L^1(M^n\setminus \mathcal{S})$,
we have
\begin{equation}
\int_{M^n\setminus \mathcal{S}} \eta_r u\,\Delta_g v\,dV_g
\to
\int_{M^n\setminus \mathcal{S}} u\,\Delta_g v\,dV_g .
\end{equation}
Similarly,
\begin{equation}
\int_{M^n\setminus \mathcal{S}}
\eta_r\langle\nabla u,\nabla v\rangle_g\,dV_g
\to
\int_{M^n\setminus \mathcal{S}}
\langle\nabla u,\nabla v\rangle_g\,dV_g .
\end{equation}
The remaining cut-off error satisfies
\begin{equation}
\left|
\int_{M^n\setminus \mathcal{S}}
u\langle\nabla \chi_r,\nabla v\rangle_g\,dV_g
\right|
\leq
\|u\|_{L^2(M^n,g)}
\|\nabla v\|_{L^\infty(M^n\setminus \mathcal{S},g)}
\|\nabla\chi_r\|_{L^2(M^n,g)} .
\end{equation}
Since
\begin{equation}
\|\nabla\chi_r\|_{L^2(M^n,g)}\to0,
\end{equation}
this error term vanishes. Therefore
\begin{equation}
\int_{M^n\setminus \mathcal{S}} u\,\Delta_g v\,dV_g
=
-\int_{M^n\setminus \mathcal{S}}
\langle \nabla u,\nabla v\rangle_g\,dV_g .
\end{equation}
\end{proof}

\begin{remark}\label{density-hausdorff}
Proposition~\ref{prop2.3} remains valid if the Minkowski dimension assumption is replaced by the Hausdorff dimension assumption $
 \dim_{\mathcal H}\mathcal S<n-2 $.
Indeed,  the Minkowski dimension bound is used only to
construct cut-off functions $\chi_r$ supported in shrinking neighborhoods of $\mathcal S$ and satisfying
\[\|\nabla \chi_r\|_{L^2(M^n,g)}\to 0 .\]
Under the Hausdorff dimension assumption, one may instead use the cut-off
functions constructed by Donaldson--Sun~\cite[pp.~75--76]{Donaldson-Sun},
which have the same vanishing \(W^{1,2}\)-energy property. With these cut-off
functions in place, the rest of the proof is unchanged.

Consequently, all subsequent arguments in this section that rely on the dimension assumption only through Proposition~\ref{prop2.3} remain valid under the Hausdorff dimension hypothesis. In particular, Theorem~\ref{density} also
holds for $\dim_{\mathcal H}\mathcal S<n-2 .$  
\end{remark}

\subsection{Existence of positive solutions} 
Throughout the remainder of this section, it will be assumed for convenience that $M^n$ has a single end $\mathcal{E}$. The case of additional ends may be treated with minor modifications. Consider the equation with prescribed asymptotics
\begin{equation}\label{afh09g}
    \Delta_g u-\mathbf{c}u=0 \text{ on } M^n , \qquad u\rightarrow 1 \text{ as } r\rightarrow \infty,
\end{equation}
with $\mathbf{c}\in C^{0, \alpha}_{-n-q_0}(\mathcal{E})\cap C^\infty(M^n\setminus \mathcal S)\cap L^\infty(M^n)$ for some $\alpha\in(0,1)$ and $q_0 >0$. 
To study this problem, we record the following Sobolev inequality.

\begin{lemma}\label{sobolev}
Let \((M^n,g)\) be a complete asymptotically flat
\(L^\infty\)-manifold. Assume that the singular set \(\mathcal{S}\) satisfies the Minkowski dimension upper bound  $\dim_{\mathcal M}\mathcal{S}< n-2$.
Then there exists a constant \(C_*>0\), depending only on the geometry of
\((M^n,g)\), such that for every compactly supported function
$u\in W^{1,2}(M^n,g)$,
one has
\begin{equation}
    \left(\int_{M^n} |u|^{\frac{2n}{n-2}}\,dV_g\right)^{\frac{n-2}{n}}
    \leq
    C_*\int_{M^n} |\nabla u|_g^2\,dV_g .
\end{equation}
\end{lemma}

\begin{proof}
Since the background metric \(g_0\) is complete, smooth, and asymptotically flat, the usual
Sobolev inequality holds on \((M^n,g_0)\); see \cite[Lemma 3.1]{SchoenYau1979PMT} for $n=3$ and note that the
same proof holds in all dimensions. In particular, by the uniform equivalence between $g$ and $g_0$,
the inequality holds with respect to $g$ for all $u\in C^{\infty}_c(M^n \setminus\mathcal{S})$.
Since $\mathcal{S}$ is (Minkowski) codimension greater than two, it has vanishing \(W^{1,2}\)-capacity by Proposition \ref{prop2.3}. Thus, $C^{\infty}_c(M^n \setminus\mathcal{S})$ is dense inside $W^{1,2}(M^n,g)$, yielding the desired result.
\end{proof}

\begin{theorem}\label{positive-solu}
Let \((M^n,g)\) be a complete asymptotically flat
\(L^\infty\)-manifold with asymptotic end $\mathcal{E}$. Assume that the singular set \(\mathcal{S}\) satisfies the Minkowski dimension upper bound $\dim_{\mathcal M}\mathcal{S}< n-2$. Suppose that $\mathbf{c}\in C^{0, \alpha}_{-n-q_0}(\mathcal{E})\cap C^\infty(M^n\setminus \mathcal S)\cap L^\infty(M^n)$ and satisfies  
\begin{equation}
\left(\int_{M^n} |\mathbf{c}_-|^{\frac{n}{2}}dV_g\right)^{\frac{2}{n}}\leq \frac{1}{2C_*}  ,\quad\quad  \mathbf{c}_-:=\min\{0, \mathbf c\},
\end{equation}
where $C_*$ is a constant from Lemma \ref{sobolev}.
Then there is a positive solution $u\in C^{\infty}(M^n\setminus \mathcal{S})\cap W^{1,2}(M^n,g)\cap L^{\infty}(M^n)$ of \eqref{afh09g} such that $u-(1+\frac{\mathcal{C}}{r^{n-2}})\in C^{2, \alpha}_{2-n-q'}(\mathcal{E})$, where $q'=\min\{1, q_0\}$ and
\begin{equation}
\mathcal{C}=-\frac{1}{(n-2)\omega_{n-1}}\int_{M^n}\mathbf{c}u \,dV_g.
\end{equation}
Moreover, \(u\) is bounded below by a positive constant in the
essential-infimum sense, that is
\begin{equation}
        \operatorname*{ess\,inf}_{M^n} u>0 .
\end{equation}
\end{theorem}

\begin{proof}
Let $\{\Omega_i\}_{i=1}^{\infty}$ be an exhaustion of $M^n$ by precompact connected open sets, each of which has a boundary $\partial \Omega_i\subset \mathcal{E}$ that is a coordinate sphere.  For each $i$, the kernel of $\Delta_g -\mathbf{c}$ in $W^{1,2}_0(\Omega_i,g)$ is trivial. To see this, observe that if a function $w\in\ker(\Delta_g-\mathbf{c})\cap W^{1,2}_0(\Omega_i,g)$ then utilizing Lemma \ref{sobolev}, H\"{o}lder's inequality, and the notion of weak solution produces
\begin{align}\label{gradient-end}
\int_{\Omega_i} |\nabla w|_g^2 dV_g&=-\int_{\Omega_i}\mathbf{c}w^2 dV_g \leq \left(\int_{M^n} |\mathbf{c}_-|^{\frac{n}{2}} dV_g\right)^{\frac{2}{n}} \left(\int_{M^n} w^\frac{2n}{n-2} dV_g\right)^{\frac{n-2}{n}}\\
&\leq \frac{1}{2C_*} \left(\int_{M^n} w^\frac{2n}{n-2}dV_g \right)^{\frac{n}{n-2}} 
\leq \frac{1}{2} \int_{M^n} |\nabla w|_g^2dV_g = \frac{1}{2} \int_{\Omega_i} |\nabla w|_g^2 dV_g,\nonumber
\end{align}
and hence $w=0$ in $\Omega_i$. It follows by elliptic theory that there is then a unique weak solution 
$w_i\in C^{\infty}(\Omega_i\setminus \mathcal{S})\cap W_0^{1,2}(\Omega_i,g)$ to the Dirichlet problem
\begin{equation}\label{Prob1}
\Delta_g  w_i-\mathbf{c} w_i=\mathbf{c} \quad \text{ in } \Omega_i , \quad\quad\quad w_i =0 \quad \text{ on }\partial\Omega_i .
\end{equation}


We will now make uniform estimates for the sequence $\{w_i\}_{i=1}^{\infty}$. The same manipulations as in \eqref{gradient-end} yield
\begin{align}
\int_{M^n} |\nabla w_i|_g^2 dV_g &=-\int_{M^n} \mathbf{c}w_i^2 dV_g- \int_{M^n} \mathbf{c}w_i dV_g\\
&\leq \frac{1}{2C_*}\left(\int_{M^n}\! w_i^\frac{2n}{n-2} dV_g\right)^{\frac{n-2}{n}}\!\!\!\!\!\!+  \left(\int_{M^n} \!\mathbf{c}^{\frac{2n}{n+2}}dV_g \right)^{\frac{n+2}{2n}}\!\! \!\left(\int_{M^n} \!w_i^\frac{2n}{n-2}dV_g \right)^{\frac{n-2}{2n}}\!\!\!.\nonumber
\end{align}
This together with Lemma \ref{sobolev} implies that  
\begin{equation}\label{initial-iteration}
\left(\int_{M^n}w^{\frac{2n}{n-2}}_i dV_g \right)^{\frac{n-2}{2n}}\leq 2C_* \left(\int_{M^n} \mathbf{c}^{\frac{2n}{n+2}} dV_g\right)^{\frac{n+2}{2n}}. 
\end{equation}
Note that the right-hand side is finite since $\mathbf{c}\in C^{0,\alpha}_{-n-q_0}(\E)\cap L^{\infty}(M^n)$. 

Moser iteration then provides a uniform $L^{\infty}$ bound on any fixed compact subset of $M^n$. Note that this holds even across the singular set, since the coefficients of $\Delta_g$, when expressed in divergence form, are controlled in $L^{\infty}$ near $\mathcal{S}$. We may then apply Schauder estimates to obtain uniform $C^{2,\alpha}$ bounds on compact subsets of $M^n \setminus\mathcal{S}$. 
Therefore, a diagonal argument may be used to extract a $C^{2,\alpha}$-convergent subsequence (with notation unchanged for convenience) $w_{i} \rightarrow w$. This limit solves the equation of \eqref{Prob1} on $M^n$, and by elliptic regularity as well as the estimates above $w\in C^{\infty}(M^n \setminus\mathcal{S})\cap W^{1,2}(M^n,g)\cap L^{\infty}(M^n)$.  Furthermore, the basic asymptotic analysis of \cite[Corollary A.38]{Lee} shows that $w-\frac{\mathcal{C}}{r^{n-2}}\in C^{2,\alpha}_{2-n-q'}(\E)$, for some constant $\mathcal{C}$ where $q'=\min\{1,q_0\}$. By setting $u=1+w$, we find that this function satisfies equation \eqref{afh09g} along with all other desired properties, except perhaps positivity.

To show positivity, let $u_{-}=\min\{0,u\}$ and observe that $u_- \in W^{1,2}_0(\Omega_i,g)$. Using the notion of weak solution and $u^-$ as a test function produces
\begin{equation}
\int_{\Omega_i}\left(\langle \nabla u,\nabla u_-\rangle_g +\mathbf{c}uu_-\right)dV_g =0.
\end{equation}
We may then replace $u$ with $u_-$ in this equation, and ignore the positive part of $\mathbf{c}$ to find
\begin{equation}
\int_{\Omega_i}|\nabla u_-|_g^2 dV_g \leq \int_{\Omega_i}|\mathbf{c}_-| u_-^2 dV_g.
\end{equation}
The same strategy as in \eqref{gradient-end} then implies that $u_-$ vanishes in $\Omega_i \setminus\mathcal{S}$. Since this sequence of domains is an exhaustion, we obtain $u\geq 0$ on $M^n \setminus \mathcal{S}$. The weak Harnack inequality \cite[Theorem 8.18]{GT} then shows that $u$ must be strictly positive away from the singular set. In fact, since the
equation holds weakly across \(\mathcal S\) and \(g\) is uniformly equivalent
to a smooth background metric, the same weak Harnack inequality applies
on balls intersecting \(\mathcal S\) as well. Hence \(u\) has a representative
which is strictly positive on \(M^n\), and
\begin{equation}
        \operatorname*{ess\,inf}_{B}u>0
\end{equation}
for every relatively compact ball \(B\subset M^n\). Together with the
asymptotic condition \(u\to1\), this gives a global positive lower bound for
\(u\) in the essential-infimum sense.

Lastly, to obtain the monopole expression let $\phi_r\in C_c^{\infty}(M^n)$ be smooth cut-off functions
such that $\phi_r=1$ on $M^n_r$ and $\phi_r=0$ on $M^n \setminus M^n_{2r}$ with $|\nabla \phi_r|_g \leq 2/r$, in which $M_r^n$ denotes the bounded component of $M^n \setminus S_r$ where $S_r\subset\mathcal{E}$ is the $r$-coordinate sphere. Then using the definition of weak solution and integrating by parts produces
\begin{align}
\int_{M^n}\!\mathbf{c}u\phi_r dV_g =-\int_{M^n}\!\langle\nabla u,\nabla\phi_r \rangle_g dV_g
&=\int_{M^n_{2r}\setminus M^n_r}\!\!\!\phi_r \Delta_g u dV_g +\int_{S_{r}}\nu(u) dA_g\\
&=\int_{M^n_{2r}\setminus M^n_r}\!\!\!\phi_r\mathbf{c}u dV_g -(n-2)\omega_{n-1}\mathcal{C} +o(1),\nonumber
\end{align}
where $\nu$ is the unit normal pointing towards infinity. Since $\mathbf{c}=O(r^{-n-q_0})$, the bulk integral on the right-hand side tends to zero as $r\rightarrow \infty$, and the desired result follows.
\end{proof}

\subsection{Proof of Theorem \ref{density}}

\begin{proposition}\label{deform-positive-scalar} 
Assume the setting and notation of Theorem \ref{density}. 
Then for each \(\varepsilon>0\)
there exists a complete asymptotically flat \(L^\infty\)-metric
\(\hat g\) on \(M^n\)
with the following properties.
\begin{enumerate}
\item There exists a uniform constant $c>0$ such that $\|\hat{g}-g\|_{L^{\infty}(M^n,g_0)}\leq c\varepsilon$.
\item \(\hat g\) is smooth on \(M^n\setminus \mathcal{S}\) and \(R_{\hat g}> 0\) on
\(M^n\setminus \mathcal{S}\).
\item If \(m\) and \(\hat m\) denote the ADM masses of \(g\) and \(\hat g\) on $\mathcal{E}$, then
$\hat m=m+2\varepsilon$.
\end{enumerate}
\end{proposition}

\begin{proof}
Choose a negative function $\hat{\mathbf{c}}\in C^\infty(M^n)\cap C^{0, \alpha}_{-n-q_0}(\E)$ for some $\alpha\in(0,1)$ and $q_0 >0$, with 
$\|\hat{\mathbf{c}}\|_{L^{n/2}(M^n)}\leq \frac{1}{2C_*}$, and consider the following equation with prescribed asymptotics
\begin{equation}\label{big1}
    \Delta_g u-\hat{\mathbf{c}} u=0\quad \text{ on }M^n,\qquad\quad u(x)\rightarrow 1 \quad\text{ as } |x|\rightarrow \infty.
\end{equation}
We now apply Theorem \ref{positive-solu} to find a positive solution $u\in C^{\infty}(M^n\setminus \mathcal{S})\cap W^{1,2}(M^n,g)\cap L^{\infty}(M^n)$ of \eqref{big1} with $u-(1+\frac{\hat{\mathcal{C}}}{r^{n-2}})\in C^{2, \alpha}_{2-n-q'}(\mathcal{E})$, where
\begin{equation}
\hat{\mathcal{C}}=-\frac{1}{(n-2)\omega_{n-1}}\int_{M^n} \hat{\mathbf{c}} udV_g>0. 
\end{equation}
For each $\varepsilon\in (0,\hat{\mathcal{C}})$ set
\begin{equation}
    f=\hat{\mathcal{C}}^{-1}\varepsilon u+(1-\hat{\mathcal{C}}^{-1}\varepsilon), \qquad\qquad \hat{g}=f^{\frac{4}{n-2}}g.
\end{equation}
This defines a complete asymptotically flat $L^{\infty}$-metric on $M^n$, which clearly satisfies property $(1)$. Moreover, its scalar curvature is given by
\begin{equation}
f^{\frac{4}{n-2}}R_{\hat{g}}=R_g-\frac{4(n-1)}{n-2}\frac{\Delta_g f}{f}=R_g -\frac{4(n-1)}{n-2}\frac{\hat{\mathcal{C}}^{-1}\varepsilon\hat{\mathbf{c}}u}{f}>0 
\end{equation}
on $M^n \setminus\mathcal{S}$. Lastly, a computation shows that the new mass takes the form
\begin{equation}
    \hat m=m+2\hat{\mathcal{C}}^{-1}\varepsilon\cdot \hat{\mathcal{C}}=m+2\varepsilon.
\end{equation}
\end{proof}

We next establish a slightly weaker version of Theorem \ref{density}, which differs from the desired result
by its nonnegative (instead of strictly positive) scalar curvature conclusion.

\begin{proposition}\label{weak-density} 
Assume the setting and notation of Theorem \ref{density}. 
Then for each \(\varepsilon>0\)
there exists a complete asymptotically flat \(L^\infty\)-metric
\(\tilde{g}\) on \(M^n\) 
with the following properties.
\begin{enumerate}
\item It holds that $\|\tilde{g}-g\|_{L^{\infty}(M^n,g_0)}<\varepsilon$.
\item \(\tilde{g}\) is smooth on \(M^n\setminus \mathcal{S}\) and \(R_{\tilde{g}}\geq 0\) on
\(M^n\setminus \mathcal{S}\).
\item On the end $\mathcal{E}$ the new metric \(\tilde{g}\) is harmonically asymptotically flat.
\item If \(\tilde{m}\) and \(m\) denote the ADM masses of \(\tilde{g}\) and \(g\) on $\mathcal{E}$, then
$|\tilde{m}-m|<\varepsilon$.
\end{enumerate}
\end{proposition}

\begin{proof}
In the asymptotic end, choose a smooth one-parameter family of radial cut-off functions \(0\leq \phi_s\leq 1\),
\(s\gg 1\), such that
\begin{equation}
\phi_s(r)=1 \quad\text{for } r\leq 2s,\qquad
\phi_s(r)=0 \quad\text{for } r\geq 4s,
\end{equation}
and
\begin{equation}
        |\nabla^k\phi_s|_{\delta}\leq C_k s^{-k}
        \qquad\text{on } \{2s\leq r\leq 4s\}.
\end{equation}
Extend \(\phi_s\) trivially to the rest of \(M^n\). Let \(g_{\mathrm{mod}}\)
be the following harmonically asymptotically flat model metric on \(\mathcal E\) with the same
mass as \(g\):
\begin{equation}
        g_{\mathrm{mod}}
        =
        \left(1+\frac{m}{2r^{n-2}}\right)^{\frac{4}{n-2}}\delta,
\end{equation}
and define
\begin{equation}
        g_s=\phi_s g+(1-\phi_s)g_{\mathrm{mod}} .
\end{equation}
Then \(g_s=g\) on \(\{r\leq 2s\}\), while
\(g_s=g_{\mathrm{mod}}\) on \(\{r\geq 4s\}\). In particular, \(g_s\) is a
complete asymptotically flat \(L^\infty\)-metric with singular set \(\mathcal S\),
and its chosen end is harmonically asymptotically flat with mass $m$ for each $s$.
Moreover \(g_s\to g\) uniformly on $M^n \setminus \mathcal{S}$ as $s\rightarrow\infty$.

Now choose a second family of radial cut-off functions \(0\leq \psi_s\leq 1\) with
\begin{equation}
        \psi_s(r)=0 \quad\text{for } r\leq s,\qquad
        \psi_s(r)=1 \quad\text{for } r\geq 2s,
\end{equation}
and
\begin{equation}
        |\nabla^k\psi_s|_{\delta}\leq C_k s^{-k}
        \qquad\text{on } \{s\leq r\leq 2s\}.
\end{equation}
Set
\begin{equation}
       \mathbf{c}_s=\frac{n-2}{4(n-1)}\psi_s R_{g_s}.
\end{equation}
Since \(g_s=g\) on \(\{r\leq 2s\}\) and \(R_g\geq 0\) there, while
\(g_s=g_{\mathrm{mod}}\) for \(r\geq 4s\), the negative part of \(\mathbf{c}_s\) is
supported in the transition annulus \(\{2s\leq r\leq 4s\}\). Note that in this annulus
$|R_{g_s}|\leq C r^{-q-2}$, and hence
\begin{equation}
        \|\mathbf{c}_{s-}\|_{L^{n/2}(M^n,g_s)}
        \leq C\left(
        \int_{2s\leq r\leq 4s} r^{-\frac n2(q+2)}r^{n-1}\,dr
        \right)^{2/n}
        \leq C s^{-q}
        \longrightarrow 0 .
\end{equation}
Thus, for \(s\) sufficiently large, the smallness hypothesis in Theorem \ref{positive-solu} is
satisfied. We may therefore solve
\begin{equation}\label{oah0fh08a0gh}
        \Delta_{g_s}u_s-\mathbf{c}_su_s=0,
        \qquad
        u_s\to 1 \quad\text{as } r\to\infty .
\end{equation}
Theorem \ref{positive-solu} gives a solution \(u_s\) bounded below by a positive constant in the
essential-infimum sense, which is smooth on \(M^n\setminus \mathcal{S}\), and with
asymptotic expansion
\begin{equation}\label{fonq0h-0q9h09gnh}
        u_s=1+\frac{\mathcal C_s}{r^{n-2}}
        +O_{2,\alpha}(r^{2-n-q'}),\quad\quad
        \mathcal C_s
        =
        -\frac{1}{(n-2)\omega_{n-1}}
        \int_{M^n} \mathbf{c}_s u_s\,dV_{g_s}.
\end{equation}
Note that we may take $q'=1$ and any $\alpha\in(0,1)$ since $\mathbf{c}_s$ vanishes identically for $r\geq 4s$.

We claim that $\mathcal C_s\longrightarrow 0$ as $s\rightarrow\infty$. Write
\begin{equation}\label{foafih0i9hh}
        \int_{M^n} \mathbf{c}_su_s\,dV_{g_s}
        =
        \int_{M^n} \mathbf{c}_s\,dV_{g_s}
        +
        \int_{M^n} \mathbf{c}_s(u_s-1)\,dV_{g_s},
\end{equation}
and observe that the energy estimate \eqref{initial-iteration} in the proof of Theorem \ref{positive-solu} gives 
\begin{equation}\label{afoh0wa9h09}
\|u_s-1\|_{L^{2n/(n-2)}(M^n,g_s)}\leq 2C_* \|\mathbf{c}_s\|_{L^{2n/(n+2)}(M^n,g_s)}, 
\end{equation}
while
$\|\mathbf{c}_s\|_{L^{2n/(n+2)}(M^n,g_s)}\rightarrow 0$
because \(|R_{g_s}|\leq Cr^{-q-2}\) on the transition region and
\(q>(n-2)/2\). Hence H\"{o}lder's inequality produces
\begin{equation}
        \left|
        \int_{M^n} \mathbf{c}_s(u_s-1)\,dV_{g_s}
        \right|\
        \leq
        \|\mathbf{c}_s\|_{L^{2n/(n+2)}}
        \|u_s-1\|_{L^{2n/(n-2)}}
        \longrightarrow 0.
\end{equation}
To see that the first integral on the right-hand side of \eqref{foafih0i9hh} also converges to zero,
write $g_s=\delta+h_s$ on the transition annulus \(\{2s\leq r\leq 4s\}\).
Note that the asymptotic flatness assumption implies
\begin{equation}
        h_s=O(r^{-q}),\qquad
        \partial h_s=O(r^{-q-1}),\qquad
        \partial^2 h_s=O(r^{-q-2}),
\end{equation}
uniformly in \(s\). Furthermore, the scalar curvature expansion in
Cartesian coordinates has the form
\begin{equation}\label{f0qjt-09jhq-}
        R_{g_s}
        =
        \partial_i\partial_j(h_s)_{ij}
        -
        \partial_i\partial_i(h_s)_{jj}
        +
        Q_s,
\end{equation}
where the quadratic error is estimated by
\begin{equation}
        |Q_s|\leq C\bigl(|\partial h_s|^2
        +|h_s|\,|\partial^2 h_s|\bigr)\leq Cr^{-2q-2},
\end{equation}
and thus
\begin{equation}
        \int_{\{2s\leq r\leq 4s\}} |Q_s|\,dV_\delta
        \leq C s^{n-2q-2}\longrightarrow 0.
\end{equation}
It remains to treat the linear part of \eqref{f0qjt-09jhq-}.
By integration by parts on the annuli and by the ADM flux formula, its integral
over the exterior region equals, up to an error tending to zero, the difference
between the ADM flux of \(g_s\) through a large coordinate sphere and the ADM
flux of \(g\) through the inner sphere. The outer flux is the ADM mass of
\(g_{\mathrm{mod}}\), while the inner flux converges to the ADM mass of \(g\).
Since \(g_{\mathrm{mod}}\) was chosen to have the same ADM mass as \(g\), these
two fluxes cancel. Consequently
\begin{equation}
\int_{M^n} \psi_s R_{g_s}\,dV_{g_s}=o(1),
\end{equation}
and it follows that $\mathcal{C}_s \rightarrow 0$.

Finally define $\tilde{g}_s=u_s^{\frac{4}{n-2}}g_s$ and set $\tilde{g}=\tilde{g}_{s_{\varepsilon}}$ for sufficiently large $s_{\varepsilon}$ depending on $\varepsilon$. Then $\tilde{g}$ is a complete asymptotically flat $L^{\infty}$-metric on $M^n$, which is smooth away from $\mathcal{S}$.
Observe that by the energy estimate \eqref{afoh0wa9h09} from Theorem \ref{positive-solu} and Moser iteration $\| u_s -1\|_{L^{\infty}(M^n,g_0)}<\varepsilon$ for large $s$, and thus the same holds for the deviation of $\tilde{g}_s$ from $g$, establishing property $(1)$. 
By equation \eqref{oah0fh08a0gh} and the conformal scalar curvature formula 
\begin{equation}
        R_{\tilde g_s}
        =
        u_s^{-\frac{n+2}{n-2}}
        \left(
        -\frac{4(n-1)}{n-2}\Delta_{g_s}u_s
        +R_{g_s}u_s
        \right)= u_s^{-\frac{4}{n-2}}(1-\psi_s)R_{g_s}.
\end{equation}
On \(\{r\leq 2s\}\) one has \(\psi_s \leq 1\) and \(g_s=g\), so
$R_{\tilde g_s}\geq0$ in this region away from $\mathcal{S}$.
On \(\{r\geq 2s\}\) one has \(\psi_s=1\) so $R_{\tilde g_s}=0$ in this region, establishing property $(2)$.
Property $(3)$ concerning the harmonic asymptotics of $\tilde{g}_s$, follows directly from the expansion \eqref{fonq0h-0q9h09gnh}
and structure of $g_{\text{mod}}$. Furthermore, the relation between the ADM masses is
\begin{equation}
        m(\tilde g_s)
        =
        m(g_s)+2\mathcal C_s
        =
        m(g)+2\mathcal C_s .
\end{equation}
Property $(4)$ then follows since \(\mathcal C_s\to0\).
\end{proof}


\begin{proof}[Proof of Theorem \ref{density}]
Consider the case of a single asymptotically flat end. First, apply Proposition
\ref{weak-density} to obtain a new metric $\tilde{g}$ on $M^n$ satisfying all desired 
properties, except that its scalar curvature is only known to be nonnegative instead of
strictly positive. Next, apply Proposition \ref{deform-positive-scalar} to $(M^n, \tilde{g})$
to obtain $(M^n , \hat{g})$, which satisfies all desired properties including that of positive scalar
curvature; note that the conformal change in the proof of this proposition preserves the harmonic asymptotics
of the end. Thus, setting $g'=\hat{g}$ yields the desired conclusion. The case of multiple asymptotically
flat ends may be treated in an analogous way with straightforward modifications.
\end{proof}

\section{Blow-up of Singular Sets}
\label{sec3}

In this section the singular set will be removed by a two-stage blow-up procedure. First, using a Green's function with pole on the codimension greater than $3-\frac{2}{n}$ singular set, we conformally deform the original $L^\infty$-metric to obtain a complete metric on the regular set, and a positive scalar-curvature-type inequality. We then construct suitable $\mu$-bubbles and derive a stability inequality strong enough to pass to the next dimension. Finally, the GMT singular set arising on the minimizing hypersurface is blown up by a second conformal deformation, allowing the dimension-reduction argument to continue. 

\subsection{Blow-up of the singular set for \texorpdfstring{$L^{\infty}$}{Linfty}-metrics}

Let \((M^n,g)\) be a complete harmonically asymptotically Euclidean
\(L^\infty\)-manifold with background metric $g_0$, and let \(\mathcal S\) denote its singular set.  In order to build
the desired conformal factor, we will employ Green's functions with pole at points of \(\mathcal S\). Choose precompact open neighborhoods \(U_1\) and \(U_2\) of $\mathcal{S}$, such that $\partial U_2$ is smooth and 
\begin{equation}
        \mathcal S\subset U_1,\qquad
        \overline{U}_1\subset U_2 .
\end{equation}
The existence of Green's functions for uniformly elliptic operators with $L^\infty$ coefficients, and their basic properties, has been achieved in \cite[Theorem~1.1]{MR657523}, see also \cite{MR161019} and \cite[Lemma~A.1]{ChengLeeTam2022}. The relevant form of this result, needed for our purposes, is recorded here. 

\begin{lemma}[Existence and estimates for the Green's function]
\label{exist-Green}
Let \((U_2,g|_{U_2})\) and \(\mathcal S\) be as above.  Then there exists a
Dirichlet Green's function
\begin{equation}
        G:U_2\times U_2\to \mathbb R\cup\{\infty\}
\end{equation}
associated with the the Laplace-Beltrami operator on \(U_2\).  Moreover, for each
\(x\in \mathcal S\), the function
\begin{equation}
        y\mapsto G(x,y)
\end{equation}
is positive and \(g\)-harmonic on \(U_2\setminus \mathcal S\).  In addition,
there exists a constant \(C>0\), independent of \(x\in\mathcal S\), such that
for all \(x\in\mathcal S\) and \(y\in U_1\setminus\mathcal S\),
\begin{equation}\label{f0a-9fjh-9qjhh}
        C^{-1} d_{g_0}(x,y)^{2-n}
        \leq G(x,y)
        \leq C d_{g_0}(x,y)^{2-n}.
\end{equation}
On the regular set \(U_2\setminus\mathcal S\), the function \(G(x,\cdot)\) is
smooth.
\end{lemma}

These Green's functions will now be used for the first blow-up, making $M^n\setminus\mathcal{S}$ into a complete manifold with
scalar curvature satisfying a positivity property.  For the remainder of this section, it will be assumed for convenience that the asymptotically flat manifold $(M^n, g)$ has only one end.

\begin{proposition}[Blow-up of low-codimension singular set]\label{blow-up1}
Let $(M^n, g)$ be a complete harmonically asymptotically flat $L^\infty$-manifold with end $\E$, and assume that its singular set satisfies the Hausdorff dimension upper bound, $\dim_{\mathcal{H}}\mathcal{S}< n-3 +\frac{2}{n}$. If the scalar curvature $R_g$ is positive on $M^n\setminus \mathcal{S}$, then there exists a complete harmonically asymptotically flat smooth metric $g'$ on $M^n \setminus\mathcal{S}$ with arbitrary ends, and a constant $\varepsilon_n>0$ such that the following properties hold.
\begin{enumerate}
\item The metric $g'$ agrees with $g$ on the end $\E$.
\item For each $\varepsilon\in (0, \varepsilon_n
)$, there exists a positive function $\rho\in C^{\infty}(M^n \setminus\mathcal{S})$ such that $\rho= 1$ on $\E$ and 
\begin{equation}\label{Q-term1}
 Q:=R_{g'}-2\Delta_{g'}\log\rho-\left(\frac{n-1}{n}+\varepsilon\right)|\nabla\log\rho|_{g'}^2>0\quad \text{ on }M^n\setminus\mathcal{S}.
\end{equation}
\end{enumerate}
\end{proposition}

\begin{proof}
\emph{Step 1: Construction of a Green's function having pole $\mathcal S$.} 
Let $\sigma\in(\dim_{\mathcal{H}}\mathcal{S}, n-3+\frac{2}{n})$, and choose a value $a$
satisfying
\begin{equation}
   \max\left\{1,\frac{2}{n-2},\frac{1}{\,n-2-\sigma\,}\right\} <a< \frac{n}{n-2}.
\end{equation}
Since \(\dim_{\mathcal H}\mathcal{S}< \sigma\) and $\mathcal{S}$ is compact, there exists
\(\ell_0\) such that, for every integer \(\ell\ge \ell_0\), one can find a
finite cover \(\{B^{g_0}_{r_{\ell,j}}(x_{\ell,j})\}_{j=1}^{N_\ell}\) of \(\mathcal S\) by $g_0$-geodesic balls,
with \(x_{\ell,j}\in \mathcal S\), satisfying
\begin{equation}\label{cover-choice}
        r_{\ell,j}\le 2^{-\ell},
        \qquad\quad
        \sum_{j=1}^{N_\ell} r_{\ell,j}^{\sigma}\le 1 .
\end{equation}
After reindexing \(\ell\), we may assume this holds for all integers \(\ell\ge 1\).
Define the atomic measure
\begin{equation}
\nu_{\mathcal S}\coloneqq
\sum_{\ell=1}^{\infty}\sum_{j=1}^{N_\ell}
r_{\ell,j}^{\,n-2-a^{-1}}\delta_{x_{\ell,j}} .
\end{equation}
Set \(p:=n-2-a^{-1}\), and observe that since \(p-\sigma>0\) it follows from \eqref{cover-choice} that for any measurable set \(X\subset U_2\) we have
\begin{align}
\begin{split}
\nu_S(X)
\leq
\sum_{\ell=1}^{\infty}\sum_{j=1}^{N_\ell}
r_{\ell,j}^{p} &=
\sum_{\ell=1}^{\infty}\sum_{j=1}^{N_\ell}
r_{\ell,j}^{p-\sigma}r_{\ell,j}^{\sigma}  \\
&\leq
\sum_{\ell=1}^{\infty}
2^{-\ell(p-\sigma)}
\sum_{j=1}^{N_\ell}r_{\ell,j}^{\sigma}  
\leq
\sum_{\ell=1}^{\infty}2^{-\ell(p-\sigma)}
<\infty .
\end{split}
\end{align}
Thus \(\nu_{\mathcal S}\) is a finite measure. Moreover, since each atom is centered at a point
\(x_{\ell,j}\in \mathcal S\), and \(\mathcal S\) is closed, \(\nu_{\mathcal S}\) is supported on \(\mathcal S\).
Next set
\begin{equation}
        G_{\mathcal{S}}(y):=\int_{\mathcal{S}} G(x,y)\,d\nu_{\mathcal S}(x)
        =
        \sum_{\ell=1}^{\infty}\sum_{j=1}^{N_\ell}
        r_{\ell,j}^{\,n-2-a^{-1}}G(x_{\ell,j},y).
\end{equation}
Since \(\nu_{\mathcal S}\) is finite and the Green's functions are locally uniformly
bounded for \(y\) in compact subsets of \(U_2\setminus \mathcal{S}\) by Lemma \ref{exist-Green}, this series
converges locally uniformly on \(U_2\setminus \mathcal S\). Hence \(G_{\mathcal S}\) is a
positive harmonic function on \(U_2\setminus \mathcal S\). Furthermore, by interior elliptic
regularity, \(G_{\mathcal S}\) is smooth on \(U_2\setminus \mathcal S\).

We note that \(G_{\mathcal S}\) has pole set \(\mathcal S\). Indeed, fix
\(x\in \mathcal S\). For each \(\ell\geq 1\), choose \(j_\ell\) such that
$x\in B^{g_0}_{r_{\ell,j_\ell}}(x_{\ell,j_\ell})$,
and set \(s_\ell:=r_{\ell,j_\ell}\). Then \(s_\ell\leq 2^{-\ell}\), and hence
\(s_\ell\to0\). If \(y\in U_2\setminus \mathcal S\) satisfies
\(d_{g_0}(y,x)<s_\ell\), then
\begin{equation}
        d_{g_0}(y,x_{\ell,j_\ell})
        \leq d_{g_0}(y,x)+d_{g_0}(x,x_{\ell,j_\ell})
        <2s_\ell .
\end{equation}
Using the lower Green's function estimate \eqref{f0a-9fjh-9qjhh} and the definition of \(\nu_{\mathcal S}\), we obtain
\begin{align}
\begin{split}
        G_{\mathcal S}(y)
        \geq
        s_\ell^{\,n-2-a^{-1}}G(x_{\ell,j_\ell},y)  
        &\geq
        C^{-1}s_\ell^{\,n-2-a^{-1}}
        d_{g_0}(x_{\ell,j_\ell},y)^{2-n}  \\
        &\geq
        C^{-1}2^{2-n}
        s_\ell^{\,n-2-a^{-1}}s_\ell^{\,2-n}  
        =
        c\,s_\ell^{-a^{-1}} .
\end{split}
\end{align}
It follows that
\begin{equation}
        G_{\mathcal S}(y)\to+\infty
        \qquad\text{as } y\to x,\quad y\in U_2\setminus\mathcal S .
\end{equation}
For the completeness of the conformally blown-up metric, however, this
pointwise blow-up is not sufficient by itself. What is needed is the stronger
integral statement that \(G_{\mathcal S}^a\) is non-integrable along every curve
approaching \(\mathcal S\). The Wolff-type potential introduced below encodes
the scale-by-scale lower bounds for the Green's potential and will be used to
prove precisely this nonintegrability.

Consider the Wolff-type potential on $U_2$,
\begin{equation}
        W(x):=\int_0^1\left(\nu_{\mathcal{S}}(B^{g_0}_r(x))r^{2-n}\right)^a\,dr .
\end{equation}
Fix \(x\in \mathcal S\), and as above for each \(\ell\ge1\) choose \(j_\ell\) such that
$x\in B^{g_0}_{r_{\ell,j_\ell}}(x_{\ell,j_\ell})$.
Note that for every \(r\in [s_\ell,2s_\ell]\), one has
\(x_{\ell,j_\ell}\in B^{g_0}_r(x)\). Therefore
\begin{equation}
        \nu_{\mathcal S}(B^{g_0}_r(x))\ge s_\ell^{\,n-2-a^{-1}} .
\end{equation}
It follows that, for \(r\in [s_\ell,2s_\ell]\),
\begin{equation}
\left(\nu_{\mathcal S}(B^{g_0}_r(x))r^{2-n}\right)^a
\ge
c\,\left(s_\ell^{\,n-2-a^{-1}}s_\ell^{\,2-n}\right)^a
=
c\,s_\ell^{-1}.
\end{equation}
Since \(s_\ell\to0\), we may choose a subsequence \(\ell_k\) such that $2s_{\ell_{k+1}}<s_{\ell_k}$,
then the intervals \([s_{\ell_k},2s_{\ell_k}]\) are pairwise disjoint. Consequently,
\begin{equation}\label{blow-up-wolff}
W(x)
\ge
\sum_{k=1}^\infty
\int_{s_{\ell_k}}^{2s_{\ell_k}}
\left(\nu_{\mathcal S}(B^{g_0}_r(x))r^{2-n}\right)^a\,dr  
\ge
c\sum_{k=1}^\infty
\int_{s_{\ell_k}}^{2s_{\ell_k}}s_{\ell_k}^{-1}\,dr
= +\infty .
\end{equation}
Thus $W$ 
takes the value $+\infty$ at all points of \(\mathcal S\).

\medskip
\noindent\emph{Step 2: Construction of a complete metric on $M^n\setminus \mathcal S$.} 
Let \(\varphi\) be a smooth function on $M^n$ satisfying
\begin{equation}
        0\leq \varphi\leq 1,\qquad
        \varphi=1 \text{ on } U_1,\qquad
        \overline{\operatorname{supp}\varphi}\subset U_2 .
\end{equation}
For a constant \(\tau>0\), to be chosen later, define
\begin{equation}
        w:=1+\tau\varphi G_{\mathcal S},
        \qquad
        g':=w^{2a}g ,
\end{equation}
where \(\varphi G_{\mathcal S}\) is extended by zero outside \(U_2\).
Since \(\varphi=0\) on the asymptotically flat end \(\mathcal E\), we have
\(g'=g\) on \(\mathcal E\). Thus, the harmonically asymptotically flat
asymptotics are preserved.

We now prove completeness near \(\mathcal S\). Let
\(\gamma:[0,L)\to U_1\setminus \mathcal S\) be a \(g\)-unit-speed curve such
that $\lim_{t\to L}\gamma(t)=x\in \mathcal S$.
Since \(g\) and \(g_0\) are uniformly equivalent on \(\overline{U}_2\), there
is a constant \(C_0>0\) such that, for \(t\in[L-\eta,L)\) with $\eta>0$ sufficiently small, we have
$d_{g_0}(\gamma(t),x)\leq C_0(L-t)$.
Set \(r_t:=2C_0(L-t)\). If \(z\in B_{r_t}^{g_0}(x)\), then
\begin{equation}
        d_{g_0}(z,\gamma(t))
        \leq d_{g_0}(z,x)+d_{g_0}(x,\gamma(t))
        \leq 2r_t .
\end{equation}
Using the lower Green's function estimate \eqref{f0a-9fjh-9qjhh}, and the fact that $\nu_{\mathcal S}$ is supported on $\mathcal{S}$, we obtain
\begin{align}
\begin{split}
        G_{\mathcal S}(\gamma(t))
        &=\int_{\mathcal S}G(z,\gamma(t))\,d\nu_{\mathcal S}(z) \\
        &\geq
        \int_{B_{r_t}^{g_0}(x)}\!\!
        G(z,\gamma(t))\,d\nu_{\mathcal S}(z) 
        \geq
        c\,\nu_{\mathcal S}(B_{r_t}^{g_0}(x))\, r_t^{2-n}.
\end{split}
\end{align}
Therefore, as in \cite[Lemma 2.20]{BiHaoHeShiZhu26},
\begin{align}
\begin{split}
        \int_0^L G_{\mathcal S}(\gamma(t))^a\,dt
        &\geq
        c\int^{L}_{L-\eta}
        \left(\nu_{\mathcal S}(B_{r_t}^{g_0}(x))r_t^{2-n}\right)^a\,dt \\
        &=
        c(2C_0)^{-1}\int_0^{2C_0\eta}
        \left(\nu_{\mathcal S}(B_r^{g_0}(x))r^{2-n}\right)^a\,dr
        =
        +\infty,
\end{split}
\end{align}
where the last equality follows from \(W(x)=+\infty\).

Since \(\varphi=1\) on \(U_1\), we have \(w\geq \tau G_{\mathcal S}\) along
\(\gamma\) for \(t\) sufficiently close to \(L\). Hence
\[
        L_{g'}(\gamma)
        =
        \int_0^L w(\gamma(t))^a\,dt
        \geq
        \tau^a\int_{L-\eta}^L G_{\mathcal S}(\gamma(t))^a\,dt
        =
        +\infty .
\]
Thus every curve approaching \(\mathcal S\) has infinite \(g'\)-length.
Since \(g'=g\) on the asymptotically flat end and \(g'\) is smooth on
\(M^n\setminus\mathcal S\), it follows that \(g'\) is complete on
\(M^n\setminus\mathcal S\).

\medskip
\noindent\emph{Step 3: The positivity of $Q$.} 
Fix \(\varepsilon>0\), to be chosen sufficiently small, and set
\begin{equation}
        \mu:=\frac{n-1}{n}+\varepsilon,
        \qquad
        \rho:=w^b,
\end{equation}
where \(b\in\mathbb R\) will be chosen below. 
Since \(g'=w^{2a}g\), the conformal transformation formulas give
\begin{align}
\begin{split}
        w^{2a}R_{g'}
        &=
        R_g-2a(n-1)\Delta_g\log w
        -(n-1)(n-2)a^2|\nabla\log w|_g^2, \\
        w^{2a}\Delta_{g'}\log w
        &=
        \Delta_g\log w
        +a(n-2)|\nabla\log w|_g^2, \\
        w^{2a}|\nabla\log w|_{g'}^2
        &=
        |\nabla\log w|_g^2 .
\end{split}
\end{align}
Using \(\log\rho=b\log w\), we obtain
\begin{align}
\begin{split}
    w^{2a}Q
    &=
    R_g-2\bigl(b+a(n-1)\bigr)\Delta_g\log w          \\
    &\quad
    -
    \left(
        2ab(n-2)
        +(n-1)(n-2)a^2
        +\mu b^2
    \right)
    |\nabla\log w|_g^2 .
\end{split}
\end{align}
Since
\begin{equation}
        \Delta_g\log w
        =
        \frac{\Delta_g w}{w}
        -
        |\nabla\log w|_g^2,
\end{equation}
this becomes
\begin{equation}\label{Q-computation-2'}
    w^{2a}Q
    =
    R_g
    -2\bigl(b+a(n-1)\bigr)\frac{\Delta_g w}{w}       
    +
    c_*|\nabla\log w|_g^2 ,
\end{equation}
where
\begin{equation}\label{b-equation'}
  c_*=  -\bigl(2a(n-2)-2\bigr)b
    -\bigl((n-1)(n-2)a^2-2a(n-1)\bigr)
    -\mu b^2.
\end{equation}
We choose \(b\) so that the coefficient of the gradient term in
\eqref{Q-computation-2'} vanishes, namely $c_*=0$.
Because $a>\frac{2}{n-2}$
this quadratic equation has a real root provided
\begin{equation}\label{discriminant-condition'}
    \mu
    <
    \frac{(2a(n-2)-2)^2}
    {4\bigl((n-1)(n-2)a^2-2a(n-1)\bigr)} .
\end{equation}
Moreover, the choice \(a<\frac{n}{n-2}\) implies
\begin{equation}
    \frac{(2a(n-2)-2)^2}
    {4\bigl((n-1)(n-2)a^2-2a(n-1)\bigr)}
    >
    \frac{n-1}{n}.
\end{equation}
Hence there exists \(\varepsilon_n>0\) such that
\eqref{discriminant-condition'} holds whenever
\begin{equation}
        0<\varepsilon<\varepsilon_n,
        \qquad
        \mu=\frac{n-1}{n}+\varepsilon .
\end{equation}
For each such \(\varepsilon\), choose a real root
\(b=b(\varepsilon)\) of \eqref{b-equation'}. The branch may be chosen so that
\(b(\varepsilon)\) remains uniformly bounded for
\(\varepsilon\in(0,\varepsilon_n)\).

With this choice of \(b\), the gradient term in \eqref{Q-computation-2'}
vanishes, and it follows that
\begin{equation}
    w^{2a}Q
    =
    R_g
    -2\bigl(b+a(n-1)\bigr)
    \frac{\tau\Delta_g(\varphi G_{\mathcal S})}{w}.
\end{equation}
Because \(w\geq1\), and because \(b=b(\varepsilon)\) is uniformly bounded
for \(\varepsilon\in(0,\varepsilon_n)\), there exists a constant
\(C_1>0\), independent of \(\varepsilon\), such that
\begin{equation}\label{Q-lower'}
    w^{2a}Q
    \geq
    R_g
    -
    C_1\tau|\Delta_g(\varphi G_{\mathcal S})| .
\end{equation}
Observe that \(\Delta_g(\varphi G_{\mathcal S})=0\) on
\(U_1\setminus\mathcal S\), since \(\varphi=1\) there and
\(G_{\mathcal S}\) is \(g\)-harmonic. It also vanishes on
\(M^n\setminus U_2\), since \(\varphi G_{\mathcal S}\) is supported in
\(U_2\). On the compact set
\(\overline{U_2\setminus U_1}\), define
\begin{equation}
    C_2:=
    \sup_{\overline{U_2\setminus U_1}}
    |\Delta_g(\varphi G_{\mathcal S})|<\infty,
    \qquad
    C_3:=
    \inf_{\overline{U_2\setminus U_1}}R_g>0 .
\end{equation}
Choose \(\tau>0\) sufficiently small so that
$C_1C_2\tau<C_3$.
Then \eqref{Q-lower'} gives
\begin{equation}
        w^{2a}Q>0
        \qquad\text{on } U_2\setminus U_1 .
\end{equation}
On \(U_1\setminus\mathcal S\) and on \(M^n\setminus U_2\), we have
\(\Delta_g(\varphi G_{\mathcal S})=0\), and hence on those sets $w^{2a}Q=R_g>0$.
Therefore, $Q>0$ on $M^n\setminus\mathcal S$.
\end{proof}

\subsection{\texorpdfstring{$\mu$}{mu}-bubble existence and stability}
Although Proposition \ref{blow-up1} shows that \(M^n\setminus\mathcal S\)
carries a complete smooth asymptotically flat metric satisfying the
scalar-curvature inequality \eqref{Q-term1}, this is not by itself sufficient
to carry out the dimension-reduction argument in all dimensions. The
obstruction is that the coefficient of the gradient term in \eqref{Q-term1}
is not preserved under dimension reduction. In this subsection we construct
suitable \(\mu\)-bubbles and derive a stability inequality in which the
coefficient of the gradient term is improved enough to allow the
dimension-reduction argument to be iterated.

\subsubsection{Potential function}

The metric produced in Proposition \ref{blow-up1} is complete and
harmonically asymptotically flat, but it may have additional ends. In
order to construct the required \(\mu\)-bubbles, we first construct a barrier
potential which vanishes on the prescribed asymptotically flat end and
tends to \(-\infty\) along the auxiliary ends.

\begin{lemma}[Existence of a barrier potential]
\label{potential-existence}
Let \((M^n,g)\) be a complete Riemannian manifold with a harmonically asymptotically
flat end $\mathcal{E}$, possibly with additional arbitrary ends, and let
\(Q\in C^\infty(M^n)\) be positive. There exists a constant \(\delta_0>0\)
such that, for every \(\delta\in(0,\delta_0)\), there is a connected open
set \(\Omega_\delta\subset M^n\) 
that contains the prescribed end $\mathcal{E}$ with $\Omega_{\delta}\setminus\E$ precompact,
and a smooth function \(\Phi\in C^\infty(\Omega_\delta)\) satisfying
\begin{equation}\label{potential-ineq}
    \Phi=0 \quad \text{on } \mathcal E,
    \qquad
    \Phi(x)\to -\infty
    \quad \text{as } x\to \partial\Omega_\delta ,
\end{equation}
and
\begin{equation}\label{potential-main-ineq}
    Q+2\delta^2\Phi^2-4|\nabla\Phi|_g
    \geq 0
    \quad \text{on } \Omega_\delta .
\end{equation}
\end{lemma}

\begin{proof}
Choose \(U\subset M^n\), a connected open neighborhood of the prescribed end $\E$, with smooth compact boundary. Since \(M^n\) is complete, the
distance function \(d_g(\cdot,\partial U)\) is a proper \(1\)-Lipschitz
function on \(M^n\setminus U\). After smoothing this function and modifying
it in a collar of \(\partial U\), we may choose a smooth proper function
$\phi:M^n\setminus U\to [0,\infty)$
such that $\phi=0$ on $\partial U$ and $|\nabla\phi|_g\leq 1$.
Choose a regular value \(d_0>0\) of \(\phi\), and set
\begin{equation}
        U':=U\cup\{x\in M^n\setminus U:\phi(x)<d_0\}.
\end{equation}
Then \(U'\) has smooth boundary, \(\overline U\subset U'\), and
\(V:=U'\setminus\overline U\) has compact closure. Moreover, with
\begin{equation}
        \partial_0V:=\partial U,\qquad
        \partial_1V:=\partial U'=\{\phi=d_0\},
\end{equation}
we have
\begin{equation}
        \phi=0 \quad \text{on } \partial_0V,
        \qquad
        \phi=d_0 \quad \text{on } \partial_1V,
        \qquad
        |\nabla\phi|_g\leq 1 .
\end{equation}
Since \(Q>0\) and \(\overline{V}\) is compact, there exists \(\delta_0>0\) appropriately
small such that
\begin{equation}\label{delta-small}
        Q>8\delta_0^2
        \quad \text{on } V,
        \qquad
        d_0\delta_0^2<\frac{\pi}{2}.
\end{equation}
For \(\delta\in(0,\delta_0)\), set
\begin{equation}\label{L-delta}
        L_\delta
        :=
        d_0+\delta^{-2}\cot(\delta^2d_0),
\end{equation}
and consider the open set
\begin{equation}\label{Omega-delta}
        \tilde{\Omega}_\delta
        :=
        U\cup\{x\in M^n\setminus U:\phi(x)<L_\delta\}.
\end{equation}
We define \(\Omega_\delta\) to be the connected component of $\tilde{\Omega}_{\delta}$ containing the prescribed end \(\mathcal E\), and note that
\(\partial\Omega_\delta\) is a subset of the level set \(\{\phi=L_\delta\}\).
Next, define a barrier function \(\Phi\) on \(\Omega_\delta\) by
\begin{equation}\label{Phi-piecewise}
    \Phi:=
    \begin{cases}
        0,
        & \text{on } U,\\[1mm]
        -2\tan(\delta^2\phi),
        & \text{on } V,\\[2mm]
        -\dfrac{2}{
        \cot(\delta^2d_0)-\delta^2(\phi-d_0)},
        & \text{on } \Omega_\delta\setminus U' .
    \end{cases}
\end{equation}
The two nonzero pieces agree in value along \(\partial_1V\), and both agree
with \(0\) along \(\partial_0V\). After smoothing in arbitrarily small
collars of \(\partial_0V\) and \(\partial_1V\), and after decreasing
\(\delta_0\) if necessary, we obtain a smooth function, still denoted
\(\Phi\), satisfying the properties below. By \eqref{L-delta}, the denominator
in the last line of \eqref{Phi-piecewise} tends to \(0^+\) as
\(\phi\to L_\delta\). Hence $\Phi(x)\to -\infty$ as $x\to\partial\Omega_\delta$.

It remains to verify \eqref{potential-main-ineq}. On \(U\), we have
\(\Phi=0\), and hence the desired inequality holds on this set due to the positivity of $Q$.
On \(V\), using \(|\nabla\phi|_g\leq 1\), we obtain
\begin{equation}
        |\nabla\Phi|_g
        \leq
        2\delta^2\sec^2(\delta^2\phi)  
        =
        2\delta^2\left(1+\frac{\Phi^2}{4}\right).
\end{equation}
Therefore, on \(V\),
\begin{equation}
        Q+2\delta^2\Phi^2-4|\nabla\Phi|_g
        \geq
        Q+2\delta^2\Phi^2
        -8\delta^2\left(1+\frac{\Phi^2}{4}\right) 
    =Q-8\delta^2>0,
\end{equation}
where the last inequality follows from \eqref{delta-small}.
On \(\Omega_\delta\setminus U'\), a direct computation gives
$|\nabla\Phi|_g\leq\frac{\delta^2}{2}\Phi^2$, and
consequently on this region
\begin{equation}
Q+2\delta^2\Phi^2-4|\nabla\Phi|_g\geq
        Q>0.
\end{equation}
The smoothing near \(\partial_0V\) and \(\partial_1V\) is performed inside a
fixed compact subset of \(\overline{V}\), where \(Q\) has a positive lower bound; by
choosing the smoothing collars sufficiently small and then decreasing
\(\delta_0\) if necessary, the same desired inequality remains valid there. 
\end{proof}

\subsubsection{Existence of \(\mu\)-bubbles} 
Let \((M^n,g)\) be a complete Riemannian manifold with a harmonically asymptotically
flat end $\mathcal{E}$ having negative mass $m(\E,g)<0$, and possibly with additional arbitrary ends.
We seek a complete asymptotically flat hypersurface (the $\mu$-bubble), that plays a role analogous to the stable minimal surfaces in classical Schoen-Yau dimensional reduction, but now in the context of arbitrary ends. The existence of such surfaces, which were originally introduced by Gromov \cite[Section 9]{Gromov1}, is well-known \cite{chodosh2024generalized,LesourdUngerYau2024,Zhu2021Width}, and thus here we only outline the argument with some detail pertinent to the current setting.

Let \(\rho\in C^{\infty}(M^n)\) be a positive weight function with $\rho=1$ on $\E$, and let \(\Phi\) and $\Omega_{\delta}$ be the barrier potential and domain constructed in Lemma \ref{potential-existence} with respect
to a positive $Q\in C^\infty(M^n)$, for some $\delta\in(0,\delta_0)$. Choose asymptotic Cartesian coordinates $x=(x',x^n)$
on the prescribed end \(\mathcal E\), where \(x'\in\mathbb R^{n-1}\). For \(r,\sigma>0\) large, define
\begin{align}
\begin{split}
\mathcal Z_{r,\sigma}
&:=
\bigl(\Omega_\delta\setminus\mathcal E\bigr)
\cup
\{x\in\mathcal E: |x'|\le r,\ -\sigma\le x^n\le\sigma\}, \\
\mathcal P_r^{\pm\sigma}
&:=
\{x\in\mathcal E: |x'|\le r,\ x^n=\pm\sigma\}, \\
\mathcal D_{r,\sigma}
&:=
\{x\in\mathcal E: |x'|=r,\ -\sigma\le x^n\le\sigma\}.
\end{split}
\end{align}
Since \(\rho=1\) and $\Phi=0$ on \(\mathcal E\), and
\(m(M^n,g,\mathcal E)<0\), the standard asymptotic barrier estimates imply
that, for all $r$ and \(\sigma\) sufficiently large,
\begin{equation}\label{boundary-barrier}
H+\nu(\log\rho)-\Phi>0 \quad\quad \text{ at } \mathcal P_r^{\pm\sigma}\cup\mathcal D_{r,\sigma},
\end{equation}
where $\nu$ is the unit outer normal and $H$ is the mean curvature of these boundary components of $\mathcal Z_{r,\sigma}$.

\medskip

\noindent\emph{Step 1: The truncated \(\mu\)-bubble problem with free boundary.}
Decompose the boundary of \(\mathcal Z_{r,\sigma}\) as
\begin{equation}
        \partial_+
        :=
        \mathcal P_r^{+\sigma},
        \qquad
        \partial_-
        :=
        \mathcal P_r^{-\sigma}\cup\partial\Omega_\delta,
        \qquad
        \partial_0
        :=
        \mathcal D_{r,\sigma}.
\end{equation}
Let \(\Omega_0\subset \mathcal Z_{R,\sigma}\) be a fixed Caccioppoli set such that
\(\partial_+\subset \Omega_0\) and \(\partial_-\cap \Omega_0=\emptyset\). We minimize
the $\mu$-bubble functional with free boundary
\begin{equation}
        \mathcal F_{r,\sigma}(\Omega)
        :=
        \int_{\partial^*\Omega}\rho\,d\mathcal H^{n-1}
        -
        \int_{\mathcal Z_{r,\sigma}}
        (\chi_\Omega-\chi_{\Omega_0})\rho\Phi\,d\mathcal{H}^n
\end{equation}
among Caccioppoli sets \(\Omega\subset\mathcal Z_{r,\sigma}\) satisfying
\begin{equation}
        \Omega\Delta \Omega_0
        \Subset
        \mathring{\mathcal Z}_{r,\sigma}\cup\partial_0 .
\end{equation}
Here $\chi_{\Omega}$ is the characteristic function of $\Omega$, $d\mathcal{H}^n$
is the $n$-dimensional Hausdorff measure, and \(\partial^*\Omega\) denotes the reduced boundary relative 
to the region in which variations are allowed; in particular, the fixed boundary faces
\(\partial_+\) and \(\partial_-\) are not included in the first variation.

The inequalities \eqref{boundary-barrier},
together with the fact that \(\Phi\to-\infty\) as
\(x\to\partial\Omega_\delta\), provide the required barriers at
\(\partial_+\) and \(\partial_-\). 
Therefore the standard compactness and barrier
argument for \(\mu\)-bubbles \cite[Sections 3 \& 4]{chodosh2024generalized} gives a minimizer \(\Omega_{r,\sigma}\) of
\(\mathcal F_{r,\sigma}\). The desired free boundary $\mu$-bubble is defined to be the closure of the reduced boundary
\begin{equation}
        \Sigma_{r,\sigma}:=\overline{\partial^*\Omega}_{r,\sigma}.
\end{equation}
The side-barrier inequality along \(\partial_0\) implies that
\(\Sigma_{r,\sigma}\) meets \(\partial_0\) orthogonally.

\medskip

\noindent\emph{Step 2: Passing to the limit and asymptotics of the end.}
We now let \(r\to\infty\), keeping \(\sigma\) fixed. By the compactness
argument for \(\mu\)-bubbles \cite[Theorem 4.5]{LesourdUngerYau2024} and almost-minimizing boundaries \cite[Lemma A.3]{Eichmair2009PlateauMOTS}, after passing to a subsequence
\(r_j\to\infty\), the hypersurfaces \(\Sigma_{r_j,\sigma}\) converge in the sense of currents
to a limiting $(n-1)$-integral current, whose support we denote by \(\Sigma^{n-1}\); we will refer to this surface as a \textit{$\mu$-bubble associated with $(\rho,\Phi)$}. The regularity theory
for almost-minimizing hypersurfaces with prescribed mean curvature \cite[Theorem A.1]{Eichmair2009PlateauMOTS} implies
that \(\Sigma^{n-1}\) is smooth outside a closed singular set \(\mathcal S_{\Sigma}\) of Hausdorff codimension seven or greater. Away from this singular set, the first variation formula implies that
\begin{equation}\label{mean-curvature}
H_\Sigma=\Phi-\nu(\log\rho),
\end{equation}
where $H_{\Sigma}$ denotes mean curvature with respect to the unit outer normal $\nu$; here the `outer' direction is determined by the approximating Caccioppoli sets. Furthermore, the corresponding Minkowski dimension bound for the singular set of \(\mu\)-bubbles follows from the
Cheeger--Naber \cite[Theorem 5.8]{CheegerNaber2013QuantitativeStratification}/Naber--Valtorta \cite[Theorem 1.6]{NaberValtorta2019Varifolds} stratification argument, as
adapted to \(\mu\)-bubbles by Brendle--Wang \cite[Theorem~3.34]{BrendleWang2026}, thus
\begin{equation}
        \dim_{\mathcal M}(\mathcal S_{\Sigma})\le n-8 .
\end{equation}
Moreover, \(\mathcal S_{\Sigma}\) is contained in a compact subset of \(M^n\). In fact, stronger statements about the asymptotics of the limiting $\mu$-bubble $\Sigma^{n-1}$ are valid. Indeed, in \cite[Theorem 4.5]{LesourdUngerYau2024} asymptotic estimates combined with Allard's regularity theorem show that there exists $r_0>0$ large enough, so that
$\Sigma^{n-1}\cap \{|x'|>r_0\}$ may be expressed as a graph $x^n=z(x')$ for a smooth function $z$ satisfying
\begin{align}
\begin{split}
        z-a_0 &\in C^{3,\alpha}_{-\epsilon}(\mathbb{R}^{n-1}\setminus B_{r_0})
        \qquad \text{if } n=3,\\
        z-\bigl(a_0+b_0|x'|^{3-n}\bigr)
        &\in C^{3,\alpha}_{3-n-\epsilon}(\mathbb{R}^{n-1}\setminus B_{r_0})
        \qquad \text{if } n\ge4,
\end{split}
\end{align}
for some $a_0 ,b_0 \in \mathbb{R}$, $\epsilon>0$, and $\alpha\in (0,1)$.
Since the ambient end $(\E,g)$ is harmonically asymptotically flat, it follows
that the corresponding asymptotic end of the $\mu$-bubble $(\E_{\Sigma},g_{\Sigma})$
is also harmonically asymptotically flat, and has zero mass $m(\E_{\Sigma},g_{\Sigma})=0$
if $n\geq 4$. We now record these results.

\begin{proposition}[\(\mu\)-bubble existence]
\label{existence-bubble}
Consider a complete Riemannian manifold $(M^n,g)$ with a harmonically asymptotically
flat end $\mathcal{E}$ having negative mass $m(\E,g)<0$, and possibly with additional arbitrary ends.
Let \(\rho\in C^{\infty}(M^n)\) be a positive weight function with $\rho=1$ on $\E$, and let \(\Phi\) and $\Omega_{\delta}$ be the barrier potential and domain constructed in Lemma \ref{potential-existence} with respect
to a positive $Q\in C^\infty(M^n)$, for some $\delta\in(0,\delta_0)$. Then there exists a $\mu$-bubble $\Sigma^{n-1}\subset\Omega_{\delta}$ associated with $(\rho,\Phi)$, which is smooth outside a closed singular set $\mathcal{S}_{\Sigma}$ residing within a compact subset of $M^n$, such that if it is nonempty
\begin{equation}
\dim_{\mathcal M}\mathcal S_{\Sigma}\le n-8 .
\end{equation}
Moreover, $(\Sigma^{n-1},g_{\Sigma})$ satisfies equation \eqref{mean-curvature} away from the singular set, and has a single harmonically asymptotically flat end\footnote{When $n=3$, the 2-dimensional end $\E_{\Sigma}$ is still referred to as harmonically asymptotically flat, where Definition \ref{fonaoinoih} is extended to accommodate this case by removing the mass term from \eqref{fonaoinoih} as in Remark \ref{2.6}. In this situation, the mass is defined to be zero.} $\mathcal{E}_\Sigma$, with the property that
if $n\geq 4$ then its mass vanishes, $m(\E_{\Sigma},g_{\Sigma})=0$.
\end{proposition}

\begin{remark}\label{boundness-mean-cuvrature} 
Note that since $\Sigma^{n-1}$ is contained in $\Omega_\delta$ and $\Omega_\delta \setminus\E$ is a precompact subset of $M^n$, the functions $\rho$ and $|\nabla\rho|_g$ are uniformly controlled, namely there is a positive constant $C_0$ such that
\begin{equation}\label{rho-bound}
   0< C_0^{-1}\leq \rho\leq C_0, \quad\quad   |\nabla \log \rho|_g\leq C_0,\quad\quad \text{ on }\Sigma^{n-1}.
\end{equation}
Moreover, since $\Sigma^{n-1}$ stays a positive distance away from $\partial \Omega_\delta$, the function $|\Phi|$ is uniformly bounded on the $\mu$-bubble. Thus, equation \eqref{mean-curvature} implies that 
\begin{equation}\label{mean-curvature-bound}
    |H_{\Sigma}|\leq C_1\quad\quad \text{ on }\Sigma^{n-1}\setminus\mathcal{S}_{\Sigma},
\end{equation}
for some positive constant $C_1$.
\end{remark}

\subsubsection{The stability inequality} 
In this subsection we derive the stability inequality satisfied by the
limiting \(\mu\)-bubble. The stability obtained here is stronger than the
usual compactly supported stability: following the strong-stability argument
of Lesourd--Unger--Yau \cite{LesourdUngerYau2024}, with the height-picking step replaced by the
free-boundary approximation, one may pass to the limit in stability
inequalities for variations which do not vanish at infinity. This yields a
strong stability inequality for test functions which differ from a constant
by a weighted Sobolev function. We then combine this strong stability with
the scalar-curvature positivity property of the ambient metric to obtain a
refined inequality as in \cite[Definition 1.3]{BrendleWang2026} although with a modified coefficient, 
suitable for the dimension-reduction argument.


\begin{proposition}[Stability inequality]
\label{weighted-stability-inequality}
Consider the hypotheses and setting of Proposition \ref{existence-bubble}, and let $\Sigma^{n-1}\subset\Omega_{\delta}$ be the constructed
$\mu$-bubble associated with $(\rho,\Phi)$. For any 
$c\in \mathbb{R}$ let $f\in C^{\infty}(\Sigma^{n-1}\setminus\mathcal{S}_{\Sigma})$ be a test function such that $f$ vanishes in a neighborhood of $\mathcal{S}_{\Sigma}$, and $f-c\in W^{1,2}_{-\frac{n-3}{2}}(\E_{\Sigma})$ when $n\geq 4$
or $\supp(f-c)$ is compact in $\overline{\E}_{\Sigma}$ when $n=3$. Then
\begin{equation}\label{mu-stability}
	\int_{\Sigma^{n-1}}\!\!\!\! \left( \rho|\nabla_\Sigma f|_{g_{\Sigma}}^2 \!-\! \rho\left(Ric(\nu,\nu)\!+\!|A_\Sigma|^2 \!-\! \nabla^2\log\rho(\nu,\nu)\!+\! \nu(\Phi)\right) f^2 \right)dV_{g_{\Sigma}}\geq 0.
\end{equation}
Furthermore, if additionally
\begin{equation}\label{qequation}
Q=R_{g}-2\Delta_{g}\log\rho-\left(\frac{n-1}{n}+\varepsilon\right)|\nabla\log\rho|_{g}^2>0\quad\text{ on }\Omega_{\delta}
\end{equation}
for some $\varepsilon>0$, then there exists a constant $\delta_{n,\varepsilon} <\delta_0$ depending only on $n$ and $\varepsilon$, and a positive function $Q_1\in C^{\infty}(\Sigma^{n-1}\setminus \mathcal{S}_{\Sigma})$, such that for any $\delta\in (0, \delta_{n,\varepsilon})$ it holds that
\begin{align}\label{inductive1}
 &  \int_{\Sigma^{n-1}} \!\!\left(\rho|\nabla_\Sigma f|_{g_{\Sigma}}^2 \!+\!\frac{1}{2}\rho\!\left(\! R_{g_{\Sigma}}\!-2\Delta_{g_\Sigma}\log\rho-\left(\!\frac{n-1}{n}+\varepsilon\right)|\nabla_{\Sigma}\log\rho|_{g_{\Sigma}}^2\right)f^2 \right)dV_{g_{\Sigma}}\nonumber\\
 \geq &\int_{\Sigma^{n-1}} \rho Q_1f^2 dV_{g_{\Sigma}}.
\end{align}
\end{proposition}


%

\begin{proof} The inequality \eqref{mu-stability} follows from the second variation formula for $\mu$-bubbles with free boundaries.  
Indeed, ordinary stability yields the inequality for compactly supported variations. In the present
free-boundary construction, the approximating bubbles meet the side cylinder
\(\partial_0=\mathcal D_{r,\sigma}\) freely. Since the vertical vector field
\(\partial_{x^n}\) is tangent to \(\partial_0\), the free-boundary stability
inequality may be applied to variations which do not vanish on
\(\partial_0\). This plays the same role as the height-picking argument in
\cite[Section~4.3]{LesourdUngerYau2024}.
Passing to the limit as \(r\to\infty\) gives the strong
stability inequality for the limiting noncompact \(\mu\)-bubble, namely for
test functions \(f\) with \(f-c\) in the appropriate weighted Sobolev space,
as in \cite[Theorem 4.21]{LesourdUngerYau2024}.

It remains to verify the stability inequality \eqref{inductive1}. All calculations to follow will be
computed on the regular set of $\Sigma^{n-1}$. Recall that two traces of the Gauss equations combined with
the Cauchy-Schwarz inequality produces 
\begin{align}
\begin{split}
		Ric_g(\nu,\nu)+|A_\Sigma|^2 &=\frac 1 2 R_{g}-\frac 1 2R_{g_\Sigma}
        +\frac 1 2 H_\Sigma^2+\frac 1 2 |A_\Sigma|^2\\
        &\geq \frac 1 2 R_{g}-\frac 1 2R_{g_\Sigma}+\frac{n}{2(n-1)} H_\Sigma^2
\end{split}
\end{align}
where $A_{\Sigma}$ denotes the second fundamental form, and the standard decomposition of ambient Laplacian with respect to a hypersurface gives
\begin{equation}
		\nabla^2\log \rho(\nu, \nu)=\Delta_g \log \rho-\Delta_{g_\Sigma}\log \rho-H_{\Sigma}\,\nu(\log\rho).	
\end{equation}
It follows that
\begin{align}
&-Ric(\nu,\nu)-|A_\Sigma|^2+ \nabla^2\log\rho(\nu,\nu)-\nu(\Phi) \nonumber\\
			&\leq~\frac {1} {2} R_{g_\Sigma}-\Delta_{g_\Sigma}\log\rho-\left(\frac{1}{2} R_g-\Delta_g\log\rho\right)-\frac{n}{2(n-1)} H_\Sigma^2-H_\Sigma\, \nu(\log\rho)+|\nu(\Phi)|\nonumber\\
			&\leq ~\frac{1}{2}R_{g_\Sigma}-\Delta_{g_\Sigma}\log\rho-\frac{1}{2}Q-\frac{1}{2}\left(\frac{n-1}{n}+\varepsilon\right)|\nabla\log\rho|_g^2+|\nu(\Phi)|\\
			&-\frac{n}{2(n-1)}\Phi^2+\frac{1}{n-1}\Phi\,\nu(\log\rho)+\frac{n-2}{2(n-1)}|\nu(\log\rho)|^2,\nonumber
\end{align}
where we have used \eqref{mean-curvature} and \eqref{qequation}.  The right-hand side of the last inequality can be decomposed into the following three parts.
\begin{align*}
&\textit{Part 1.}\quad \frac{1}{2} R_{g_\Sigma}-\Delta_{g_\Sigma}\log\rho-\frac 1 2\left(\frac{n-1}{n}+\varepsilon\right)|\nabla_{\Sigma}\log\rho|_{g_{\Sigma}}^2-\frac{1}{4} Q \\
&\textit{Part 2.}\quad\!\!-\frac{\delta^2}{2}\Phi^2+|\nu(\Phi)|-\frac 1 4 Q\leq 0 \\
&\textit{Part 3.} -\!\left(\!\frac{n}{2(n-1)}-\frac{\delta^2}{2}\right)\!\Phi^2 \!+\!\frac{1}{n-1}\Phi\, \nu(\log\rho)\!-\!\left(\!\frac{1}{2n(n-1)}+\frac{\varepsilon}{2}\!\right)\!|\nu(\log\rho)|^2 \!\leq 0
\end{align*} 
Observe that Part 1 is the desired part, while the nonpositivity of Part 2 follows from \eqref{potential-main-ineq}
for $\delta\in (0,\delta_0)$. Next, choose a positive $\delta_{n,\varepsilon}<\delta_0$ depending on $n$ and $\varepsilon$ so that 
\begin{equation}
\left(\frac{1}{n-1}\right)^2 -4\left(\frac{n}{2(n-1)}-\frac{\delta_{n,\varepsilon}^2}{2}\right)\left(\frac{1}{2n(n-1)}+\frac{\varepsilon}{2}\right)<0.
\end{equation}
Then for each $\delta\in(0,\delta_{n,\varepsilon})$ the inequality of Part 3 is satisfied. It follows that
for such $\delta$ we have
\begin{equation}\begin{split}
&-Ric(\nu,\nu)-|A_\Sigma|^2+ \nabla^2\log\rho(\nu,\nu)-\nu(\Phi) \quad \\
&\leq \frac{1}{2} R_{g_\Sigma}-\Delta_{g_\Sigma}\log\rho- \frac{1}{2}\left(\frac{n-1}{n}+\varepsilon\right)|\nabla_\Sigma\log\rho|_{g_\Sigma}^2-\frac{1}{4} Q.
\end{split}
\end{equation} 
Inequality \eqref{inductive1} is now obtained by applying \eqref{mu-stability} and setting $Q_1=\frac{1}{4} Q$.
\end{proof}

\subsection{The second blow up}\label{secondbu}
The purpose of this subsection is to find a complete metric on $\Sigma^{n-1}\setminus\mathcal{S}_{\Sigma}$ by a conformal blow-up of the singular set, while preserving the ADM mass. The next result will place the blown-up $\mu$-bubble in the context of so called weak $n$-data sets, which are studied in the next section.


\begin{theorem}[Blow-up of GMT singularities]\label{blow-up-2rd-GMT} 
Assume the hypotheses and setting of Proposition \ref{weighted-stability-inequality}, with $Q\in C^{\infty}(M^n)\cap C^{0,\alpha}_{-n-q_0}(\E_{\Sigma})$ where $\alpha\in(0,1)$ and $q_0 >0$. Then there exists a positive constant $\varepsilon'_n $ such that for any $\varepsilon\in (0, \varepsilon'_{n})$, there exists a positive function $Q'\in C^{\infty}(\Sigma^{n-1}\setminus\mathcal{S}_{\Sigma})\cap C^{0,\alpha}_{-n-q_0}(\E_{\Sigma})$, a positive function $\rho'\in C^{\infty}(\Sigma^{n-1}\setminus \mathcal{S}_{\Sigma})$, and a complete harmonically asymptotically flat metric $g_{\Sigma}'$ on $\Sigma^{n-1}\setminus \mathcal{S}_{\Sigma}$ with arbitrary ends such that
\begin{equation}
    \rho'=1 \text{ on } \E_{\Sigma}, \quad\text{and } \quad  m(\E_{\Sigma},g_{\Sigma}')=0. 
\end{equation}
Furthermore, for any $c\in \mathbb{R}$ and $f\in C^{\infty}(\Sigma^{n-1}\setminus\mathcal{S}_{\Sigma})$ with the property that $f$ vanishes in a neighborhood of $\mathcal{S}_{\Sigma}$, and $f-c\in W^{1,2}_{-\frac{n-3}{2}}(\E_{\Sigma})$ when $n\geq 4$
or $\supp(f-c)$ is compact in $\overline{\E}_{\Sigma}$ when $n=3$, it holds that
\begin{align}\label{blow-up2}
 &  \int_{\Sigma^{n-1}} \!\!\!\left(\rho'|\nabla_{\Sigma} f|_{g_{\Sigma}'}^2+\frac{1}{2}\rho'\left(\!R_{g'_{\Sigma}}\!-\!2\Delta_{g'_\Sigma}\log\rho'\!-\!\left(\frac{n-1}{n}+\varepsilon\right)|\nabla_{\Sigma}\log\rho'|_{g_{\Sigma}'}^2\right)\!f^2 \! \right)dV_{g'_{\Sigma}}\nonumber\\
 \geq &\int_{\Sigma^{n-1}} \rho' Q'f^2 dV_{g'_{\Sigma}}.
\end{align}
\end{theorem}

\subsubsection{Local blow-up function}
Because the geometry of \(\Sigma^{n-1}\) may be poorly controlled near the singular
set \(\mathcal{S}_{\Sigma}\), we cannot directly construct a blow-up function on
\(\Sigma^{n-1}\) by using a Green's function with pole set \(\mathcal{S}_{\Sigma}\).
Instead, we use the ambient manifold \((M^n,g)\).  The ambient
distance function has controlled Hessian estimates, and this allows us to
construct a local blow-up function on \(\Sigma^{n-1}\) which replaces the Green's
function used in the first blow-up. We first establish a local blow-up function
following the construction of Brendle--Wang \cite[Section 3.7]{BrendleWang2026}.
In order for this to be meaningful we must have $n\geq 8$.



\begin{lemma}[Blow-up function near singularities]\label{local-blow-up} 
Let $U_1\subset M^n$ be a precompact open set containing $\mathcal{S}_{\Sigma}$, and assume that $n\geq 8$. Then there is a precompact open set $U_2 \subset M^n$ with 
\begin{equation}
\mathcal{S}_{\Sigma}\subset \overline{U}_2\subset U_1,
\end{equation}
and a nonnegative function $\psi_{\mathcal{S}}\in C^{\infty}(M^n \setminus\mathcal{S}_{\Sigma})$ with $\supp \,\psi_{\mathcal{S}}\subset U_1$ such that
\begin{equation}\label{3.71}
  \psi_{\mathcal{S}}(x) \geq Cd_g(x, \mathcal{S}_{\Sigma})^{-2} \quad \text{ for } x\in U_2 \setminus\mathcal{S}_{\Sigma}
\end{equation} 
and for some constant $C>0$, and
\begin{equation}\label{3.70}
   \Delta_{g_\Sigma}\psi_{\mathcal{S}} +(n-3)(n^{-1}-\varepsilon) g_\Sigma(\nabla_\Sigma\psi_{\mathcal{S}}, \nabla_\Sigma \log \rho)<0 \quad\text{ on } (\Sigma^{n-1} \setminus\mathcal{S}_{\Sigma})\cap U_2 .
\end{equation}
\end{lemma}

\begin{proof} 
\emph{Step 1: Construction of the one-point blow-up function.}
Choose a sufficiently small $t_*>0$ so that $\{x\in M^n : d_g(x,\mathcal{S}_{\Sigma})\leq \sqrt{t_*}\}$ is contained in $U_1$, and pick a smooth monotone nonnegative function $\zeta: \mathbb{R}_+\rightarrow \mathbb{R}_+$ with 
\begin{equation}
 \zeta(t)=t^{\frac{3-n}{2}}+t^{\frac{7-2n}{4}} \text{ for }t\in(0,t_*/2],  \quad \quad\quad  \zeta(t)=0 \text{ for }t\in[t_*,\infty).
\end{equation} 
For any point $p\in \mathcal{S}_\Sigma$ define 
\begin{equation}
    \eta_{p}(x)=d^2_g(x, p), \quad \quad\quad \psi_{p}(x)=\zeta(\eta_p(x)).
\end{equation}
Note that by choosing $t_*$ smaller if necessary to achieve $\sqrt{t_*}<\frac{1}{2}\mathrm{inj}_{(M^n,g)}(p)$ for all $p\in\mathcal{S}_{\Sigma}$, we can ensure that $\psi_p \in C^{\infty}(M^n \setminus\mathcal{S}_{\Sigma})$. Moreover,  $\supp \,\psi_p \subset U_1$.
Since $\overline{U}_1$ is compact, there exists a constant $C>0$ such that on the $(M^n,g)$-geodesic ball $B_{\!\!\sqrt{t_*/2}}(p)$ it holds that
\begin{equation}
|\nabla^2 \eta_p\!-\!2g|_g\leq\! C\eta_p,\quad
\psi_p=\eta_p^{\frac{3-n}{2}}\!+\eta_p^{\frac{7-2n}{4}},
\quad
 |\nabla\psi_p|_g\leq C(\eta_p^{\frac{2-n}{2}}\!+\eta_p^{\frac{5-2n}{4}}),
\end{equation}
and
\begin{align}
\nabla^2 \psi_p &=\left(\left(\tfrac{3-n}{2}\right) \eta_p^{\frac{1-n}{2}}+\left(\tfrac{7-2n}{4}\right) \eta_p^{\frac{3-2n}{4}}\right)\nabla^2 \eta_p \\
&\quad+\left(\left(\tfrac{3-n}{2}\right)\left(\tfrac{1-n}{2}\right)\eta_p^{-\frac{1+n}{2}}+\left(\tfrac{7-2n}{4}\right)\left(\tfrac{3-2n}{4}\right)\eta_p^{-\frac{1+2n}{4}}\right)d\eta_p \otimes d\eta_p.\nonumber
\end{align}
Combining this with \eqref{rho-bound} and \eqref{mean-curvature-bound}, on $(\Sigma^{n-1}\setminus\mathcal{S}_{\Sigma})\cap B_{\!\!\sqrt{t_*/2}}(p)$, produces
\begin{align}\label{point-pde-est-near-singularity}
     \Delta_{g_\Sigma} \psi_p+&(n-3)({n^{-1}}-\varepsilon)g_{\Sigma}(\nabla_\Sigma\psi_p, \nabla_\Sigma\log\rho) \quad \quad \nonumber\\
     &\leq \Delta_{g} \psi_p-\nabla^2\psi_p(\nu, \nu)-H_\Sigma \,\nu(\psi_p)+C_1 |\nabla\psi_p|_g\\
     &\leq \Delta_{g} \psi_p-\nabla^2\psi_p(\nu, \nu)+C_2 |\nabla\psi_p|_g \nonumber\\
     &\leq -\left(\tfrac{2n-7}{4}\right)\eta_p^{\frac{3-2n}{4}}+C_{3} (\eta_p^{\frac{2-n}{2}}+\eta_p^{\frac{3-n}{2}} + \eta^{\frac{5-2n}{4}}_p +\eta^{\frac{7-2n}{4}}_p)\leq -\left(\tfrac{2n-7}{8}\right)\eta_p^{\frac{3-2n}{4}}\nonumber
\end{align}
where the last two inequalities follow by choosing $t_*$ even smaller, if necessary; here and below uniform positive constants will be denoted by $C_i$, $i\in \mathbb{Z}_+$. Moreover, on the complimentary domain $(\Sigma^{n-1}\setminus\mathcal{S}_{\Sigma})\cap B_{\!\!\sqrt{t_*/2}}(p)^c$ we have
\begin{equation}\label{point-pde-est-away-sing}
\Delta_{g_\Sigma} \psi_p+(n-3)({n^{-1}}-\varepsilon)g_{\Sigma}(\nabla_\Sigma\psi_p, \nabla_\Sigma\log\rho) \leq C_4.
\end{equation} 

\medskip

\noindent\emph{Step 2: Construction of the desired blow-up function.}
A subset $\mathcal{T}\subset\mathcal{S}_{\Sigma}$ is called an \textit{$r$-separated subset} if $B_r(p)\cap B_r(p')=\emptyset$ for any two distinct points $p, p'\in\mathcal{T}$. Set $r_l=2^{-l}\sqrt{t_*}$ and take a maximal $r_l$-separated subset $\mathcal{T}_l\subset \mathcal{S}_{\Sigma}$, for each $l\in\mathbb{Z}_+$. The Minkowski dimension bound $\dim_{\mathcal M}\mathcal S_{\Sigma}\le n-8$, and compactness of the singular set, imply that $\mathcal{T}_l$ is a finite and satisfies the cardinality inequality
$|\mathcal{T}_l|\leq C_{5}r_l^{\mathring{a}-n}$,
for any $\mathring{a}<8$. By restricting further so that $\mathring{a}\in (5, 8)$, we also obtain
\begin{equation}
\sum_{l=1}^\infty r_l^{n-5}|\mathcal{T}_l|<\infty. 
\end{equation}
Define the blow-up function $\psi_{\mathcal{S}}:M^n \setminus\mathcal{S}_{\Sigma}\rightarrow\mathbb{R}_{+}$ by
\begin{equation}
	\psi_{\mathcal{S}}(x)=\sum_{l=1}^\infty\sum_{p\in \mathcal{T}_l}r_l^{n-5}\psi_{p}(x),
\end{equation}
and note that
\begin{align}
\begin{split}
\psi_{\mathcal{S}}(x)\leq C_6\sum\limits_{l=1}^\infty r_l^{n-5} |\mathcal{T}_l|d_g(x, \mathcal{S}_\Sigma)^{3-n}\leq C_{7}d_g(x,\mathcal{S}_\Sigma)^{3-n}. 
\end{split}
\end{align}
This shows that $\psi_{\mathcal{S}}$ is well-defined, and similar estimates show that it is smooth on $ M^n \setminus\mathcal{S}_{\Sigma}$. Moreover, $\supp \,\psi_{\mathcal{S}}\subset U_1$.
\medskip

\noindent\emph{Step 3: Construct $U_2$ and verify the properties of $\psi_{\mathcal{S}}$.} 
Let $t_{**}\in(0,t_*)$ be sufficiently small, to be chosen below, and set
\begin{equation*}
U_2:=\{~x\in M^n :d_g(x,\mathcal{S}_\Sigma)<t_{**}\}.     
\end{equation*} 
Take $l_*>1$ depending on $t_{**}$ such that for any $x\in U_2 \setminus \mathcal{S}_\Sigma$, there  
is an integer $l_x>l_*$ with the property 
\begin{equation}
r_{l_x+1}\leq d_g(x,\mathcal{S}_\Sigma)\leq r_{l_x}.
\end{equation} 
This ensures the existence of a point $p_{l_x}\in \mathcal{T}_{l_x}$ with $d_g(x,p_{l_x})\leq 2r_{l_x}$. It follows that
\begin{equation}
    \psi_{\mathcal{S}}(x)\geq r_{l_x}^{{n-5}}\psi_{p_{l_x}}(x)=r_{l_x}^{n-5}d_g(x, p_{l_x})^{3-n}\geq C_{8} r_{l_x}^{-2}\geq C_{9} d_g(x,\mathcal{S}_\Sigma)^{-2},
\end{equation}
which establishes \eqref{3.71}. 

It remains to verify \eqref{3.70}. Let $x\in(\Sigma^{n-1} \setminus\mathcal{S}_{\Sigma})\cap U_2$.
For convenience set
\begin{equation}
        L_\Sigma
        :=
        \Delta_{g_\Sigma}
        +
        (n-3)(n^{-1}-\varepsilon)
        g_{\Sigma}(\nabla_\Sigma(\cdot),\nabla_\Sigma\log\rho).
\end{equation}
By the choice of \(l_*\), the point \(x\) lies in
\(B_{\!\!\sqrt{t_*/2}}(p_{l_x})\). Hence \eqref{point-pde-est-near-singularity}
gives
\begin{equation}
        L_\Sigma\psi_{p_{l_x}}(x)
        \le
        -C_{10}\,r_{l_x}^{\frac{3-2n}{2}}
\end{equation}
for some constant $C_{10}>0$, while \eqref{point-pde-est-away-sing} implies that all other terms appearing in the sum $L_{\Sigma}\psi_{\mathcal{S}}$
have a uniform upper bound $L_\Sigma\psi_p(x)\le C_{11}$.
Therefore
\begin{equation}
        L_\Sigma\psi_{\mathcal S}(x)
        \le
        -C_{10}\,r_{l_x}^{-\frac{7}{2}}
        +
        C_{11}
        \sum_{\ell=1}^{\infty}
        \sum_{p\in\mathcal T_\ell} r_\ell^{n-5}  \\
        \le
        -C_{10}\,r_{l_x}^{-\frac{7}{2}}+C_{12}.
\end{equation}
Finally, by choosing \(t_{**}\) sufficiently
small, we can ensure that \(r_{l_x}\) is sufficiently small for every
\(x\in U_2\setminus\mathcal S_\Sigma\), to achieve $-C_{10}\,r_{l_x}^{-\frac{7}{2}}+C_{12}<0$,
which yields the desired result.
\end{proof}

\subsubsection{Gobal blow-up function}

\begin{lemma}\label{blow-up-fun-GMT} 
Consider the hypotheses and setting of Proposition \ref{weighted-stability-inequality}, with $n\geq 8$. Then there exists a function
$w\in C^{\infty}(M^n \setminus\mathcal{S}_{\Sigma})$ with $w\geq 1$, such that $w=1$ on $\mathcal{E}_{\Sigma}$, 
for some constant $C>0$ it holds that
\begin{equation}\label{twoside-behav}
w(x)\geq C d_g(x, \mathcal{S}_{\Sigma})^{-2} \quad \text{ for } x\in U_2 \setminus\mathcal{S}_{\Sigma}
\end{equation}
where $U_2 \supset \mathcal{S}_{\Sigma}$ is as in Lemma \ref{local-blow-up}, and
\begin{equation}\label{Q-1-term}
Q_1-\frac{n+1}{n-3}\left(\frac{\Delta_{g_\Sigma}w}{w}+(n-3)(n^{-1}-\varepsilon)g_{\Sigma}(\nabla_{\Sigma}\log \rho, \nabla_\Sigma\log w)\right)>0
\end{equation}
on $\Sigma^{n-1}\setminus \mathcal{S}_\Sigma$.
\end{lemma}

\begin{proof} 
Let $\tau>0$ be a constant to be determined, and set 
\begin{equation}
w=1+\tau \psi_{\mathcal{S}}.
\end{equation} 
where $\tau$ is a positive constant to be determined. Then immediately $w\geq 1$, and also \eqref{3.71} implies \eqref{twoside-behav}. Furthermore, since $\supp\, \psi_{\mathcal{S}}\subset U_1$ we find that \eqref{Q-1-term} holds
on $\Sigma^{n-1}\cap (M^n \setminus U_1)$ due to the positivity of $Q_1$. On the other hand, \eqref{3.70} implies \eqref{Q-1-term} on $(\Sigma^{n-1}\setminus\mathcal{S}_{\Sigma})\cap U_2$. Moreover, since $U_1 \setminus U_2$ is precompact and $Q_1$ is positive there, it follows that we can find $\tau>0$ to ensure that \eqref{Q-1-term} holds on $\Sigma^{n-1}\cap (U_1 \setminus U_2)$ as well.
\end{proof}

\subsubsection{Proof of Theorem \ref{blow-up-2rd-GMT}}
\begin{plainproof}
Consider a complete Riemannian manifold $(M^n,g)$ with a harmonically asymptotically
flat end $\mathcal{E}$ having negative mass $m(\E,g)<0$, and possibly with additional arbitrary ends.
Let \(\rho\in C^{\infty}(M^n)\) be a positive weight function with $\rho=1$ on $\E$, and let \(\Phi\) and $\Omega_{\delta}$ be the barrier potential and domain constructed in Lemma \ref{potential-existence} with respect
to a positive $Q\in C^{\infty}(M^n)\cap C^{0,\alpha}_{-n-q_0}(\E_{\Sigma})$, for some $\delta\in(0,\delta_0)$. Let $\Sigma^{n-1}\subset\Omega_{\delta}$ be the $\mu$-bubble associated with $(\rho,\Phi)$ given by Proposition \ref{existence-bubble}, with singular set denoted by $\mathcal{S}_{\Sigma}$. For $n\geq 8$ define
\begin{equation}
    g'_{\Sigma}=w^{\frac{n+1}{n-3}}g_\Sigma,\quad\quad\quad \rho'= w^{-\frac{n+1}{2}}\rho,
\end{equation}
where the positive function $w\in C^{\infty}(M^n \setminus \mathcal{S}_{\Sigma})$ is provided by Lemma \ref{blow-up-fun-GMT}, and for $3\leq n\leq 7$ define $g'_{\Sigma}=g_{\Sigma}$ and $\rho'=\rho$. Clearly $\rho'\in C^\infty(\Sigma^{n-1}\setminus\mathcal{S}_{\Sigma})$ and $\rho'=1$ on $\mathcal{E}_{\Sigma}$. Moreover, $g'_{\Sigma}$ is harmonically asymptotically flat with zero mass, $m(\mathcal{E}_{\Sigma},g'_{\Sigma})=0$.

We next show that the metric $g'_{\Sigma}$ is complete, and assume that $n\geq 8$ otherwise there is nothing to prove. Since  $\Sigma^{n-1}\setminus \mathcal{S}_{\Sigma}$ has ends, namely $\mathcal{E}_{\Sigma}$ and $(\Sigma^{n-1}\setminus\mathcal{S}_{\Sigma})\cap U_2$, and $(\mathcal{E}_{\Sigma},g'_{\Sigma})$ is harmonically asymptotically flat, it is sufficient to study the behavior of a $g'_{\Sigma}$-unit-speed curve $\gamma: [0, L)\rightarrow (\Sigma^{n-1}\setminus\mathcal{S}_{\Sigma})\cap U_2$ with $\lim_{t\rightarrow L}\gamma(t)\in \mathcal{S}_{\Sigma}$. Indeed, observe that \eqref{twoside-behav} implies
\begin{equation}
    L_{g'_{\Sigma}}(\gamma)=\int^{L}_0|\gamma'(t)|_{g_\Sigma}w(\gamma(t))^{\frac{n+1}{2(n-3)}} dt\geq c\int_0^L (L-t)^{-\frac{n+1}{n-3}} dt=\infty
\end{equation}
where $c>0$ is a constant, from which the desired conclusion follows.

It remains to verify the stability inequality \eqref{blow-up2}. 
Recall the following standard formulas for a conformal change
\begin{align}
\begin{split}
		&w^{\frac{n+1}{n-3}}R_{g'_\Sigma}=
        R_{g_\Sigma}-\frac{(n-2)(n+1)}{n-3}\frac{\Delta_{g_\Sigma} w}{w}-\frac{(n+1)(n-2)}{4}{|\nabla_\Sigma \log w|_{g_\Sigma}^2},\\
		&w^{\frac{n+1}{n-3}}\Delta_{g'_\Sigma}\log\rho' =\Delta_{g_\Sigma} \log\rho'+\frac{n+1}{2} g_\Sigma(\nabla_\Sigma \log w, \nabla_\Sigma \log\rho').
\end{split}
\end{align}
Therefore by setting
\begin{equation}
Q'=\frac{1}{2}w^{-\frac{n+1}{n-3}}Q_1,
\end{equation}
a calculation combined with \eqref{Q-1-term} shows that  
\begin{align}\label{a0fhq-09hj-}
    &R_{g'_\Sigma}-2\Delta_{g'_\Sigma}\log\rho'-\left(\frac{n-1}{n}+\varepsilon\right)|\nabla_\Sigma \log\rho'|_{g'_\Sigma}^2-2Q' \nonumber\\
    &=w^{-\frac{n+1}{n-3}}\left(R_{g_{\Sigma}}-2\Delta_{g_\Sigma} \log \rho -\left(\frac{n-1}{n}+\varepsilon\right)|\nabla_\Sigma \log\rho|_{g_\Sigma}^2\right)\\
    &-w^{-\frac{n+1}{n-3}}\left(Q_1+\frac{n+1}{n-3}\left(\frac{\Delta_{g_\Sigma}w}{w}+(n-3)(n^{-1}-\varepsilon) g_{\Sigma} (\nabla_{\Sigma}\log \rho, \nabla_\Sigma\log w)\right)\right)\nonumber\\
    &+\left(\frac{n+1}{4n}-\frac{(n+1)^2}{4}\varepsilon\right)|\nabla_\Sigma \log w|_{g_\Sigma}^2 w^{-\frac{n+1}{n-3}}\nonumber\\    
    &\geq w^{-\frac{n+1}{n-3}}\left(R_{g_{\Sigma}}-2\Delta_{g_\Sigma} \log \rho -\left(\frac{n-1}{n}+\varepsilon\right)|\nabla_\Sigma \log\rho|_{g_\Sigma}^2\right)-2w^{-\frac{n+1}{n-3}} Q_1,\nonumber
\end{align}
where we have chosen $\varepsilon'_n >0$ such that 
$\frac{n+1}{4n}-\frac{(n+1)^2}{4}\varepsilon'_n \geq 0$, and have taken $\varepsilon\in (0,\varepsilon'_n)$.
The desired stability inequality now follows by integrating \eqref{a0fhq-09hj-} against test functions, and applying
\eqref{inductive1}.
\end{plainproof}

\section{Weak \texorpdfstring{$n$}{n}-Data Sets and the Generalized PMT}
\label{sec4}

This section isolates the analytic structure needed for the dimension-reduction
argument. Brendle--Wang \cite[Definition 1.3]{BrendleWang2026} introduced \(n\)-data sets 
as the inductive objects in
their dimension-descent proof of the positive mass theorem, and our proof of
the generalized positive mass theorem follows their overall strategy. The
definition used here is a weakened version adapted to the codimension greater than $3-\frac{2}{n}$
singularities in our main theorem. In particular, the weighted stability
inequality below contains a gradient term involving
\(|\nabla\log\rho|^2\) with a coefficient of smaller absolute value than that
appearing in \cite{BrendleWang2026}.

\begin{definition}
Let $(M^n, g, \E)$ be a complete Riemannian manifold of dimension $n\geq 3$ with arbitrary ends, and having a designated
harmonically asymptotically flat end $\E$. 
A \textit{weak $n$-data set} consists of these objects together with a triple $(\rho,Q,\varepsilon)$, where  
$\rho$, $Q$ are smooth functions on $M^n$ and $\varepsilon$ is a constant, all satisfying the following properties.
\begin{enumerate}
\item The functions $\rho$ and $Q$ are globally positive and $\varepsilon\in (0, \frac{1}{n(n+1)})$. 

\item $Q\in C^{0, \alpha}_{-n-q_0}(\E)$ and  $\rho-(1+\beta r^{2-n})\in C^{2, \alpha}_{2-n-q_0}(\E)$ for some $\alpha\in(0,1)$, $q_0>0$, and $\beta\in \mathbb{R}$.
\item For any $c\in \mathbb{R}$ and $f\in C^{\infty}(M^n)$ with the property that $(\supp \,f) \setminus\E$ is precompact, and $f-c\in W^{1,2}_{-\frac{n-2}{2}}(\E)$, it holds that
\begin{align}\label{weak-data-stability}
 &  \int_{M^n} \!\!\!\left(\rho|\nabla f|_{g}^2+\frac{1}{2}\rho\left(\!R_{g}\!-\!2\Delta_{g}\log\rho\!-\!\left(\frac{n}{n+1}+\varepsilon\right)|\nabla\log\rho|_{g}^2\right)\!f^2 \! \right)dV_{g}\nonumber\\
 \geq &\int_{M^n} \rho Qf^2 dV_{g}.
\end{align}
\end{enumerate}
\end{definition}

The next theorem is the main result of this section. It may be considered as an extension of the positive mass theorem, and generalizes \cite[Theorem 1.5]{BrendleWang2026}.

\begin{theorem}\label{general-PMT} 
Let $(M^n, g, \E,\rho,Q,\varepsilon)$ be a weak $n$-data set, then
\begin{equation}\label{aofqhjt-9qy}
m(\E,g)+\left(\tfrac{n-2}{n-1}\right)\beta > 0.
\end{equation}
\end{theorem}

The quantity $m(\E,g)+\left(\tfrac{n-2}{n-1}\right)\beta$
will be referred to as the \textit{generalized mass} associated with the weak \(n\)-data set.
There is a loose analogy with the positive mass theorem involving charge \cite[Theorem 11.9]{BartnikChrusciel2005DiracBoundary}. In the present setting, the coefficient $\beta$ is the monopole coefficient of the weight $\rho$, and it enters the generalized mass in the same spirit that a charge contributes an additional asymptotic invariant in the charged positive mass theorem, with $\rho$ playing the role of potential for a static electric field. The strict positivity in \eqref{aofqhjt-9qy} is a consequence of the global positivity of \(Q\). This
strict positivity is essential for our proof of Theorem \ref{A}, which is given at the end of this section. 
Indeed, Theorem \ref{blow-up-2rd-GMT} produces a weak \(n\)-data set with zero
generalized mass, and the strict positivity obtained here is precisely that which provides the desired
contradiction.

\subsection{Construction of the auxiliary weight factor}
In this subsection we construct a positive auxiliary factor \(v\) which
modifies the weight \(\rho\).  The purpose of this step is to replace
\(\rho\) by the improved weight $\hat\rho:=\rho v$
so that the weak stability inequality yields a pointwise scalar-curvature
inequality for \(\hat\rho\).  This improved weight will provide the
positivity needed for the \(\mu\)-bubble construction in the next subsection.

Let $(M^n, g, \E,\rho,Q,\varepsilon)$ be a weak $n$-data set, and define the smooth function 
\begin{equation}\label{def-K}
    \mathcal{Q}:=\frac{1}{2}\left(R_g-2\Delta_g \log \rho -\left(\frac{n}{n+1}+\varepsilon\right)|\nabla\log \rho|_g^2-Q\right).
\end{equation}
The decay assumptions on $\rho$, $Q$, and the harmonically asymptotically flat end imply that
\begin{equation}\label{decay-K-rho}
     \nabla \log \rho\in C^{1, \alpha}_{1-n}(\E), \qquad \mathcal{Q}\in C^{0,\alpha}_{-n-q_1}(\E),
\end{equation}
where $q_1=\min\{q_0,n-2,\mathring{q}+2-n\}>0$.
Let $\Omega\subset M^n$ be a connected open set that contains $\E$, such that $\Omega\setminus \E$ is precompact. Then from \eqref{decay-K-rho} and the positivity of $\rho$, there is a constant $C_0>0$ for which
\begin{equation}\label{bound-rho-K}
    C_0^{-1}\leq \rho\leq C_0 \quad\text{and}\quad |\mathcal{Q}|\leq C_0 \quad\text{ on }\Omega.
\end{equation}
In this setting, consider the following Dirichlet problem for $v\in C^{\infty}(\Omega)$ with prescribed asymptotics at infinity 
\begin{equation}\label{conform-weak-stability}
    \operatorname{div}_g(\rho\nabla v)-\rho\mathcal{Q}v=0\text{ }\text{ in }\Omega, \text{ }\text{ }\quad v=0 \text{ }\text{ on }\partial\Omega,  \text{ }\text{ } \quad  v(x)\rightarrow 1 \text{ }\text{ as } x\rightarrow \infty.
\end{equation} 
To study this problem we introduce the bilinear form $\mathcal{B}_{\Omega}:W^{1,2}_0(\Omega)\times W^{1,2}_0(\Omega)\rightarrow \mathbb{R}$ given by
\begin{equation}\label{weighted--energy}
 \mathcal{B}_{\Omega}(u, w)=\int_{\Omega} \left(\rho\langle\nabla u, \nabla w\rangle_g+\rho \mathcal{Q}uw\right)dV_g.
\end{equation} 
Although the potential term \(\mathcal{Q}\) need not be nonnegative, the weak
\(n\)-data inequality implies that \(\mathcal B_\Omega\) is coercive in the
Sobolev sense. 

\begin{lemma}[Coercive Sobolev inequality]
\label{sobolev-weight-operator}
There exists a constant \(C_\Omega>0\), depending on \(\Omega\) and the weak
\(n\)-data, such that for every \(u\in W^{1,2}_0(\Omega)\),
\begin{equation}
        \left(
        \int_\Omega |u|^{\frac{2n}{n-2}}\,dV_g
        \right)^{\frac{n-2}{n}}
        \leq
        C_\Omega
        \int_\Omega
        \left(
        \rho|\nabla u|_g^2+\rho\mathcal{Q} u^2
        \right)dV_g .
\end{equation}
\end{lemma}

\begin{proof}
Suppose by way of contradiction that there is a sequence of $\{u_i\}^\infty_{i=1}\subset C^\infty_c(\Omega)$ with 
\begin{equation}\label{contrary-poincare}
\mathcal{B}_\Omega(u_i, u_i)\rightarrow 0, \qquad \quad \int_\Omega |u_i|^{\frac{2n}{n-2}}dV_g=1.
\end{equation}
Then according to \eqref{weak-data-stability} we have
\begin{equation}\label{weight-L2-convergence}
   \frac{1}{2}\int_\Omega \rho Qu_i^2 dV_g\leq  \mathcal{B}_\Omega(u_i, u_i)\rightarrow 0.
\end{equation}
Recall that the Sobolev inequality holds on asymptotically flat manifolds $(\Omega, g|_{\Omega})$ with boundary,
see \cite[Lemma 3.1]{SchoenYau1979PMT} for $n=3$ and note that the
same proof holds in all dimensions.
Thus, there is a constant $C_* >0$ such that 
\begin{equation}\label{stand-poincare-type}
  \left(\int_\Omega |u_i|^{\frac{2n}{n-2}}dV_g\right)^{\frac{n-2}{n}}  
  \leq C_* \int_\Omega |\nabla u_i|_g^2 dV_g.
\end{equation}

Next, observe that $\rho \mathcal{Q}\in L^{\frac{n}{2}}(\Omega)$ by \eqref{decay-K-rho}. Hence, there is a large constant $r_0>0$ such that 
\begin{equation}\label{small-term-infty}
\left(\int_{\E_{r_0}}|\rho \mathcal{Q}|^{\frac{n}{2}}dV_g \right)^\frac{2}{n}\leq\frac{1}{2C_*C_{0}},
\end{equation}
where in the Cartesian coordinates of the end, $\E_{r_0}:=\{x\in\E : |x|\geq r_0\}$. We will decompose $\Omega$ into two parts, namely $\E_{r_0}$ and $\Omega\setminus \E_{r_0}$. 
Since $\Omega\setminus \E_{r_0}$ is precompact and $\rho Q$ is positive, there is a constant $C_{1}>0$ such that $\rho Q\geq C_{1}$ on this domain. It follows from \eqref{weight-L2-convergence} that
\begin{equation}\label{cong-compact-set}
    \frac{C_1}{2} \int_{\Omega\setminus \E_{r_0}} u^2_i dV_g\leq \frac{1}{2}\int_{\Omega}\rho Q u_i^2 dV_g\leq \mathcal{B}_\Omega(u_i, u_i)\rightarrow 0,
\end{equation}    
and with the help of \eqref{bound-rho-K} this implies
\begin{equation}\label{fj-9jq09yhh}
\int_{\Omega\setminus\E_{r_0}}|\rho \mathcal{Q}|u^2_i dV_g\rightarrow 0.
\end{equation}
By utilizing again \eqref{bound-rho-K}, as well as \eqref{stand-poincare-type} and \eqref{small-term-infty} we obtain
\begin{align}
    \mathcal{B}_\Omega(u_i, u_i)&\geq \int_{\Omega} \rho|\nabla u_i|_g^2 dV_g+\int_{\E_{r_0}} \rho \mathcal{Q}
    u^2_i dV_g-\int_{\Omega\setminus \E_{r_0}}|\rho \mathcal{Q}|u^2_i dV_g\\
    &\geq \!\frac{1}{C_0}\!\int_\Omega\! |\nabla u_i|_g^2 dV_g-\|\rho \mathcal{Q}\|_{L^\frac{n}{2}(\E_{r_0})}\left(\int_\Omega \!|u_i|^{\frac{2n}{n-2}}dV_g \right)^{\frac{n-2}{n}}\!\!\!\!-\!\int_{\Omega\setminus \E_{r_0}}\!\!\!\!\!\!|\rho \mathcal{Q}|u^2_i dV_g \nonumber\\
    &\geq \frac{1}{2C_*C_{0}}\left(\int_\Omega |u_i|^{\frac{2n}{n-2}}dV_g \right)^{\frac{n-2}{n}}-\int_{\Omega\setminus \E_{r_0}}|\rho \mathcal{Q}|u^2_i dV_g .\nonumber
\end{align}
Combing this with \eqref{fj-9jq09yhh} produces
\begin{equation}
     \int_\Omega |u_i|^{\frac{2n}{n-2}}dV_g\rightarrow 0,
\end{equation} 
which yields a contradiction with \eqref{contrary-poincare}.
\end{proof}

\begin{proposition}\label{solving-conformal-factor}
Let $(M^n,g,\E,\rho,Q,\varepsilon)$ be a weak $n$-data set.
Then there exists a solution $v\in C^{\infty}(\overline{\Omega})$ of \eqref{conform-weak-stability}, which is positive in $\Omega$, such that 
\begin{equation}\label{decay-rate-conf-weig-solu}
v-\left(1+\frac{\mathcal{C}}{r^{n-2}}\right)\in C^{2, \alpha}_{2-n-q'_1}(\E),
\end{equation} 
where $q'_1=\min\{1, q_1\}$ and
\begin{equation}\label{expression-constant-infinity}
  (n-2)\omega_{n-1}  \mathcal{C}=-\mathcal{B}_\Omega(v, v)\leq- \frac{1}{2}\int_{M^n} \rho Q v^2 dV_g<0. 
\end{equation}
\end{proposition}

\begin{proof}
Let $\{\Omega_i\}_{i=1}^{\infty}$ be an exhaustion of $\Omega$ by precompact connected open sets whose closure contains $\partial\Omega$, and such that $\partial \Omega_i \setminus\partial\Omega \subset \mathcal{E}$ is a coordinate sphere. 
For convenience, set
\begin{equation}
    L:=\Delta_g +\langle\nabla \log\rho, \nabla (\cdot)\,\rangle_g-\mathcal{Q}.
\end{equation}
Then for each $i$, the kernel of $L$ in $W^{1,2}_0(\Omega_i)$ is trivial. To see this, observe that if a function $w\in\ker{L}\cap W^{1,2}_0(\Omega_i)$ then with the help of \eqref{weak-data-stability} we have
\begin{equation}
    0=-\int_{\Omega_i}\rho w L(w)dV_g=\mathcal{B}_{\Omega_i}(w, w)\geq \frac{1}{2}\int_{\Omega_i} \rho Qw^2 dV_g .
\end{equation}
Hence $w=0$, since both $\rho$ and $Q$ are positive.

Choose a function $v_0\in C^{\infty}(\overline{\Omega})$ such that 
\begin{equation}
    v_0=1 \ \text{ in }  \E \quad \quad \text{and} \quad \quad v_0=0 \ \text{ on } \partial \Omega.
\end{equation}
It follows by elliptic theory that there is then a unique solution $w_i \in C^{\infty}(\overline{\Omega}_i)$ to the Dirichlet problem
\begin{equation}\label{dir-problem-L}
    L(w_i)=-L(v_0) \,\text{ in } \Omega_i, \quad \quad\quad  w_i=0 \,\text{ on } \partial \Omega_i. 
\end{equation} 
Multiplying this equation by $\rho w_i$ and integrating by parts produces
\begin{equation}
 \mathcal{B}_{\Omega_i} (w_i, w_i)=\int_{\Omega_i}\rho w_i L(v_0) dV_g\leq \|\rho L(v_0)\|_{L^{\frac{2n}{n+2}}(\Omega)}\|w_i \|_{L^{\frac{2n}{n-2}}(\Omega_i)}.
\end{equation}
Lemma \ref{sobolev-weight-operator} then implies that 
\begin{equation}\label{voanfiaq9hh}
    \|w_i \|_{L^{\frac{2n}{n-2}}(\Omega_i)}\leq C_\Omega \|\rho L(v_0) \|_{L^{\frac{2n}{n+2}}}(\Omega).
\end{equation}
Note that since $v_0=1$ in $\E$ we find that $\rho L(v_0)=-\rho\mathcal{Q}\in C^{0, \alpha}_{-n-q_1}(\E)$ by \eqref{decay-K-rho}, which implies that the right-hand side of \eqref{voanfiaq9hh} is finite.

Using the same convergence argument as in the proof of Theorem \ref{positive-solu} shows that after passing to a subsequence, $\{w_i\}$ converges to $w\in C^{\infty}(\overline{\Omega})$ with $w|_{\partial \Omega}=0$. The decay of the coefficients of $L$ as recorded in \eqref{decay-K-rho}, together with the basic asymptotics analysis of \cite[Corollary A.38]{Lee} implies that $w-\frac{\mathcal{C}}{r^{n-2}}\in C^{2,\alpha}_{2-n-q_1'}(\E)$, for some constant $\mathcal{C}$ where $q'_1=\min\{1, q_1\}$. By setting $v=w+v_0$, we find that this function satisfies equations \eqref{conform-weak-stability} and \eqref{decay-rate-conf-weig-solu}. Moreover, the monopole estimate \eqref{expression-constant-infinity} follows from the stability inequality \eqref{weak-data-stability}, and the positivity of $v$ in $\Omega$ arises from a similar argument as in the proof of Theorem \ref{positive-solu}.
\end{proof}

\begin{corollary}\label{corpositive}
Let $v\in C^{\infty}(\overline{\Omega})$ be the function given by Proposition \ref{solving-conformal-factor}, and set $\hat{\rho}=\rho v$.  If $\varepsilon\in(0,\frac{2n+1}{n(n+1)})$ then it holds that  
\begin{equation}\label{pre-inequality-scalar}
    R_g-2\Delta_g \log\hat\rho -\left(\frac{n-1}{n}+\varepsilon\right)|\nabla\log\hat\rho|_g^2 \geq Q\quad\text{ on }\Omega.
\end{equation}
    
\end{corollary}
\begin{proof}  
Using \eqref{def-K} and \eqref{conform-weak-stability}, and substituting $\log \hat{\rho}=\log \rho+\log v$ we obtain that  
\begin{align}
R_g-2\Delta \log\hat\rho& =2g(\nabla \log \hat\rho, \nabla \log v )+\left(\frac{n}{n+1}+\varepsilon\right)|\nabla\log\hat\rho-\nabla \log v|_g^2+Q \nonumber\\
&=Q+\left(\frac{n}{n+1}+\varepsilon\right)|\nabla \log \hat \rho|_g^2 \\
&\quad+2\left(\frac{1}{n+1}-\varepsilon\right)g(\nabla \log \hat\rho, \nabla \log v)+\left(\frac{n}{n+1}+\varepsilon\right)|\nabla \log v|_g^2 \nonumber\\
&=Q+\left(\frac{n-1}{n}+\varepsilon\right)|\nabla \log\hat\rho|_g^2 +\frac{1}{n(n+1)}|\nabla \log \hat \rho|_g^2 \nonumber\\
&\quad+2\left(\frac{1}{n+1}-\varepsilon\right)g(\nabla \log \hat\rho, \nabla \log v)+\left(\frac{n}{n+1}+\varepsilon\right)|\nabla \log v|_g^2 . \nonumber
\end{align}
The last three terms form a quadratic expression in
$\nabla\log\hat\rho$ and $\nabla\log v$. By the allowable range for
$\varepsilon$, its discriminant is negative
\begin{equation}
\left(\frac1{n+1}-\varepsilon\right)^2-
\frac{1}{n(n+1)}
\left(\frac n{n+1}+\varepsilon\right)<0 .
\end{equation}
The desired inequality \eqref{pre-inequality-scalar} now follows.
\end{proof}

\subsection{Construction of \texorpdfstring{$\mu$}{mu}-bubble}
For each $\delta\in (0, \delta_0)$, Lemma \ref{potential-existence} provides a connected open set $\Omega_\delta$ and a potential function $\hat{\Phi}\in C^{\infty}(\Omega_{\delta})$ with the following properties: 
\begin{equation}
\E\subset \Omega_\delta ,\qquad\quad \Omega_\delta\setminus\E\Subset M^n ,
\end{equation}
and $\hat{\Phi}$ satisfies \eqref{potential-ineq} and \eqref{potential-main-ineq}
on $\Omega_\delta$. By applying Proposition \ref{solving-conformal-factor} on $\Omega_\delta$, we obtain a positive function $v\in C^{\infty}(\Omega_{\delta})$ satisfying \eqref{conform-weak-stability}, together with the asymptotic properties \eqref{decay-rate-conf-weig-solu} and \eqref{expression-constant-infinity}. Define
$\hat{\rho}=\rho v$.

\begin{proposition}[Existence of $\mu$-bubble]\label{existinductivebuble} 
Under the setting and hypotheses of Theorem \ref{general-PMT}, if 
\begin{equation}
m(\E,g)+\left(\tfrac{n-2}{n-1}\right)\beta \leq 0,
\end{equation} 
then there exists a $\mu$-bubble $\Sigma^{n-1}\subset\Omega_{\delta}$ associated with $(\hat{\rho},\hat{\Phi})$, which is smooth outside a closed singular set $\mathcal{S}_{\Sigma}$ residing within a compact subset of $M^n$, such that if it is nonempty
\begin{equation}
\dim_{\mathcal M}\mathcal S_{\Sigma}\le n-8 .
\end{equation}
Moreover, $(\Sigma^{n-1},g_{\Sigma})$ satisfies
\begin{equation}
H_{\Sigma}=\hat \Phi-\nu(\log\hat{\rho})
\end{equation}
away from the singular set where $\nu$ is the unit outer normal, and it
has a single harmonically asymptotically flat end $\mathcal{E}_\Sigma$, with the property that if $n\geq 4$ then the mass vanishes, $m(\E_{\Sigma},g_{\Sigma})=0$.
\end{proposition}

\begin{proof}
It is enough to verify the barrier conditions along the boundary of the
truncated regions; the rest of the construction is identical to the proof of
Proposition \ref{existence-bubble}. 
First, observe that using the expansions of $\rho$ and $v$
shows that
\begin{equation}
        \hat{\rho}-\left(1+\hat{\beta} r^{2-n}\right)
        \in C^{2,\alpha}_{2-n-q_1'}(\mathcal E), \quad\quad\quad \hat{\beta}=\beta+\mathcal{C}.
\end{equation}
Therefore, the assumed nonpositivity of the generalized mass, and the fact that $\mathcal{C}<0$ by \eqref{expression-constant-infinity}, gives the strict inequality
\begin{equation}
     \mathfrak m:=   m(\mathcal E,g)+\frac{n-2}{n-1}\hat\beta
        =
        m(\mathcal E,g)+\frac{n-2}{n-1}\beta
        +
        \frac{n-2}{n-1}\mathcal C
        <0 .
\end{equation}

We now verify the boundary barriers. Since \(\hat\Phi=0\) in the prescribed
asymptotically flat end, the \((\hat\rho,\hat\Phi)\)-barrier quantity on a
hypersurface with unit normal \(\nu\) is
\begin{equation}
        H+\nu(\log\hat\rho)-\hat\Phi
        =
        H+\nu(\log\hat\rho) .
\end{equation}
In the harmonically asymptotically flat coordinates, the horizontal slices
\(\{x^n=t\}\) satisfy, with respect to the normal
\(\nu=\eta\,\partial_{x^n}+o(1)\), where \(\eta=\pm1\),
\begin{equation}
        H+\nu(\log\hat\rho)
        =
        -(n-1)\mathfrak m \,\eta t\, r^{-n}
        +O(r^{-n-q_1'}) .
\end{equation}
On the top face \(\mathcal P_r^\sigma=\{x^n=\sigma\}\), the outward normal
of \(\mathcal Z_{r,\sigma}\) is asymptotic to \(+\partial_{x^n}\), so
\(\eta t=\sigma>0\). On the bottom face
\(\mathcal P_r^{-\sigma}=\{x^n=-\sigma\}\), the outward normal is asymptotic
to \(-\partial_{x^n}\), and again \(\eta t=\sigma>0\). Since
\(\mathfrak m<0\), it follows that, after choosing
\(\sigma\) and then \(r\) sufficiently large,
\begin{equation}
        H
        +
        \nu(\log\hat\rho)
        -
        \hat\Phi
        >0
        \quad\text{ on } \mathcal P_r^{\pm\sigma}\cup \mathcal{D}_{r,\sigma},
\end{equation}
where the slow $\frac{1}{r}$-decay of the mean curvature is responsible
for this inequality on the side cylinder $\mathcal{D}_{r,\sigma}$.
\end{proof}

The relevant stability inequality may be established in the same manner as Proposition \ref{weighted-stability-inequality}, and is recorded here. Note that Corollary \ref{corpositive} guarantees that the scalar curvature quantity positivity hypothesis \eqref{qequation}, needed for the proof of Proposition \ref{weighted-stability-inequality}, is satisfied in the current setting.

\begin{proposition}[Stability inequality]
\label{inductive-weigh-strong-stability}
Consider the hypotheses and setting of Proposition \ref{existinductivebuble}, and let $\Sigma^{n-1}\subset\Omega_{\delta}$ be the constructed
$\mu$-bubble associated with $(\hat \rho,\hat \Phi)$. For any 
$c\in \mathbb{R}$ let $f\in C^{\infty}(\Sigma^{n-1}\setminus\mathcal{S}_{\Sigma})$ be a test function such that $f$ vanishes in a neighborhood of $\mathcal{S}_{\Sigma}$, and $f-c\in W^{1,2}_{-\frac{n-3}{2}}(\E_{\Sigma})$ when $n\geq 4$
or $\supp(f-~c)$ is compact in $\overline{\E}_{\Sigma}$ when $n=3$. Then there exists a constant $\delta_{n,\varepsilon} <\delta_0$ depending only on $n$ and $\varepsilon$, and a positive function $\hat{Q}\in C^{\infty}(\Sigma^{n-1}\setminus \mathcal{S}_{\Sigma})\cap C^{0,\alpha}_{-n+1-q_1'}(\E_{\Sigma})$, such that for any $\delta\in (0, \delta_{n,\varepsilon})$ it holds that
\begin{align}\label{inductive123}
 &  \int_{\Sigma^{n-1}} \!\!\left(\hat{\rho}|\nabla_\Sigma f|_{g_{\Sigma}}^2 \!+\!\frac{1}{2}\hat{\rho}\!\left(\! R_{g_{\Sigma}}\!-2\Delta_{g_\Sigma}\log\hat \rho-\left(\!\frac{n-1}{n}+\varepsilon\right)|\nabla_{\Sigma}\log\hat \rho|_{g_{\Sigma}}^2\right)f^2 \right)dV_{g_{\Sigma}}\nonumber\\
 \geq &\int_{\Sigma^{n-1}} \hat \rho \, \hat Qf^2 dV_{g_{\Sigma}}.
\end{align}
\end{proposition}

Combined with the blow-up procedure for GMT singularities in Section 3.3, the preceding construction yields a complete harmonically asymptotically Euclidean manifold with arbitrary ends carrying a weak \((n-1)\)-data set. 

\begin{corollary}\label{inductive-step}
Let $(M^n, g, \E,\rho,Q,\varepsilon)$ be a weak $n$-data set with $n\geq 4$. If the generalized mass satisfies
\begin{equation}
m(\E,g)+\left(\tfrac{n-2}{n-1}\right)\beta \leq 0,
\end{equation}
then there exists a weak $(n-1)$-data set $(M^{n-1},g',\E',\rho',Q',\varepsilon')$ with vanishing mass
\begin{equation}
m(\E',g')=0.
\end{equation}
\end{corollary}

\begin{proof} 
Let $(\Sigma^{n-1},g_{\Sigma},\E_{\Sigma},\hat{\rho},\hat{Q},\varepsilon)$ be as in Proposition \ref{inductive-weigh-strong-stability}. Note that here $\hat{\rho}$ is restricted to $\Sigma^{n-1}$, and satisfies $\hat{\rho}-1 \in C^{2,\alpha}_{3-n-q_1'}(\E_\Sigma)$. According to Proposition \ref{existinductivebuble} the mass is zero $m(\E_{\Sigma},g_{\Sigma})=0$, and for $\delta\in(0,\delta_{n,\varepsilon})$ it admits the stability inequality \eqref{inductive123}. If $n<8$ then the singular set is empty $\mathcal{S}_{\Sigma}=\emptyset$, and thus this forms the desired weak $(n-1)$-data set. 

Consider now the case when $n\geq 8$. The GMT singular set may be blown-up using the construction of Section \ref{secondbu}.
In particular let $w\in C^{\infty}(\Sigma^{n-1}\setminus\mathcal{S}_{\Sigma})$, with $w=1$ on $\E_\Sigma$, be the function
constructed in Lemma \ref{blow-up-fun-GMT}, and set
\begin{equation}
g'=w^{\frac{n+1}{n-3}}g_\Sigma, \quad \quad\quad \rho'=w^{-\frac{n+1}{2}}\hat{\rho},\quad\quad\quad 
Q'=\frac{1}{2}w^{-\frac{n+1}{n-3}}\hat{Q}.
\end{equation}
By Theorem \ref{blow-up-2rd-GMT}, the blown-up manifold $(\Sigma^{n-1}\setminus\mathcal{S}_\Sigma,g')$ is complete and harmonically asymptotically flat with vanishing mass $m(\E_\Sigma,g')=0$. Moreover, it satisfies the stability inequality 
\eqref{blow-up2} if $\varepsilon\in(0,\varepsilon'_n)$, where we may take $\varepsilon'_n =\frac{1}{n(n+1)}$. Note also
that $Q'\in C^{0, \alpha}_{1-n-q'_1}(\E_\Sigma)$ and $\rho'-1\in C^{2, \alpha}_{3-n-q'_1}(\E_\Sigma)$. Therefore,
setting $M^{n-1}=\Sigma^{n-1}\setminus\mathcal{S}_{\Sigma}$, $\E'=\E_\Sigma$, and $\varepsilon'=\varepsilon$ yields
the desired weak $(n-1)$-data set.
\end{proof}

\subsection{Proof of Theorem \ref{general-PMT}}
We will first establish Theorem \ref{general-PMT} in the 3-dimensional case. For general $n$, we will argue by contradiction and use Corollary \ref{inductive-step} inductively to reduce the problem to dimension three.

\begin{lemma}\label{3-dim-general-PMT} Theorem \ref{general-PMT} holds for $n=3$. 
\end{lemma}

\begin{proof} 
Suppose to the contrary that $m(\E,g)+\left(\tfrac{n-2}{n-1}\right)\beta \leq 0$. As in the proof of Corollary \ref{inductive-step}, we obtain a weak $2$-data set $(\Sigma^2,g_\Sigma,E_\Sigma,\hat{\rho},\hat{Q},\varepsilon)$. In particular, for any $f\in C^1_c(\Sigma^2)$ the following stability inequality is satisfied
\begin{align}
 &  \int_{\Sigma^{n-1}} \!\!\left(\hat{\rho}|\nabla_\Sigma f|_{g_{\Sigma}}^2 \!+\!\frac{1}{2}\hat{\rho}\!\left(\! R_{g_{\Sigma}}\!-2\Delta_{g_\Sigma}\log\hat \rho-\left(\!\frac{2}{3}+\varepsilon\right)|\nabla_{\Sigma}\log\hat \rho|_{g_{\Sigma}}^2\right)f^2 \right)dV_{g_{\Sigma}}\nonumber\\
 \geq &\int_{\Sigma^{n-1}} \hat \rho \, \hat Qf^2 dV_{g_{\Sigma}}.
\end{align}
We now rearrange this inequality so that it is adapted to the logarithmic cut-off argument. Thus, write $f=\hat{\rho}^{-1/2}\psi$ where $\psi\in C^{1}_c(\Sigma^2)$, integrate by parts, and apply Young's inequality to the mixed term to obtain
\begin{equation}
   \int_{\Sigma^2} \left(4|\nabla_{\Sigma}\psi|_{g_\Sigma}^2+K_{g_\Sigma}\psi^2\right)dV_{g_\Sigma}\geq \int_{\Sigma^2}  \hat{Q}\psi^2 dV_{g_\Sigma},
\end{equation}
where $K_{g_\Sigma}$ denotes Gaussian curvature.

Following standard arguments, let $x$ denote Cartesian coordinates on the asymptotically flat end $\E_\Sigma$, and for $r_0 >1$ define the logarithmic cut-off function
\begin{equation}
    \psi_{r}(x)=\begin{cases}
        1 \qquad\qquad\quad   \text{if } |x|\leq r \\
        2-\frac{\log|x|}{\log r}  \quad \text{ if } r <|x|\leq r^2\\
        0 \qquad\qquad\quad \text{if } |x|\geq r^2 
    \end{cases}.
\end{equation}
Smoothing this function at the radii $r$ and $r^2$ and using the same notation, we note that 
since the end has quadratic area growth
\begin{equation}
    \int_{\Sigma^2} |\nabla_\Sigma \psi_{r}|_{g_\Sigma}^2 dV_{g_\Sigma}\rightarrow 0 \quad \quad\text{as } r\rightarrow\infty. 
\end{equation} 
Combining this with the positivity of $\hat{Q}$, and the decay of the metric, it follows that 
\begin{equation}
\int_{\Sigma^2}K_{g_\Sigma} dV_{g_\Sigma}>0.
\end{equation}
On the other hand, let $D_r\subset\Sigma^2$ be the compact connected domain bounded by the coordinate circle $\{|x|=r\}$. Since the end $\E_\Sigma$ is asymptotically flat, the geodesic curvature of the boundary circle satisfies
\begin{equation}
\kappa_{g_\Sigma}=r^{-1}+o(r^{-1}),\quad\qquad \int_{\partial D_R}\kappa_{g_{\Sigma}} ds=2\pi+o(1).
\end{equation}
Therefore, since the Euler characteristic admits the bound $\chi(D_r)\leq 1$, Gauss-Bonnet implies
\begin{equation}
\lim_{r\rightarrow\infty}\int_{D_r}K_{g_\Sigma}dV_{g_\Sigma}=\lim_{r\rightarrow\infty}\left(2\pi\chi(D_r)-\int_{\partial D_r}\kappa_{g_\Sigma}ds\right)\leq 0,
\end{equation}
yielding a contradiction.
\end{proof}


\begin{proof}[Proof of Theorem \ref{general-PMT} for $n>3$] Suppose the contrary that the generalized mass satisfies 
\begin{equation}
  m(\E,g)+\left(\tfrac{n-2}{n-1}\right)\beta \leq 0.
\end{equation} 
Corollary \ref{inductive-step} provides a weak $(n-1)$-data set $(M^{n-1},g',\E',\rho',Q',\varepsilon')$ with vanishing mass,
and strong decay for the weight function
\begin{equation}
m(\E',g')=0,\quad\quad\quad \rho'-1\in C^{2, \alpha}_{3-n-q'_1}(\E').
\end{equation}
In particular, the generalized mass of this new data set is zero.
By repeating this dimensional reduction finitely many times, we arrive at a weak $3$-data set with zero generalized mass, which contradicts Lemma \ref{3-dim-general-PMT}. 
\end{proof}

\subsection{Proof of the inequality in Theorem \ref{A}} 
\begin{plainproof}
Proceeding by contradiction, assume that the mass of some asymptotically flat end is negative, $m(\E,g)<0$. By 
Theorem \ref{density}, we may further assume that $(M^n,g)$ has strictly positive scalar curvature $R_g>0$ on $M^n \setminus\mathcal{S}$, and that all ends including $\E$ are harmonically asymptotically flat. By blowing-up the
singular set as in Proposition \ref{blow-up1}, we obtain a complete manifold $(M^n\setminus\mathcal{S},g')$ such that
$g'$ agrees with $g$ on $\E$. Moreover, for each $\varepsilon\in (0, \varepsilon_n)$, there exists a positive function $\rho\in C^{\infty}(M^n \setminus\mathcal{S})$ such that $\rho= 1$ on $\E$ and
\begin{equation}
 R_{g'}-2\Delta_{g'}\log\rho-\left(\frac{n-1}{n}+\varepsilon\right)|\nabla\log\rho|_{g'}^2>0\quad \text{ on }M^n\setminus\mathcal{S}.
\end{equation}
Now apply Propositions \ref{existence-bubble} and \ref{weighted-stability-inequality} to obtain a $\mu$-bubble
$\Sigma^{n-1}\subset M^n\setminus\mathcal{S}$ associated with $(\rho,\Phi)$, that satisfies the stability inequality
\eqref{inductive1}; here the barrier potential $\Phi$ is constructed in Lemma \ref{potential-existence}. Furthermore,
$(\Sigma^{n-1},g_\Sigma)$ has a single asymptotically flat end $\E_\Sigma$ of zero mass, $m(\E_\Sigma,g_\Sigma)=0$.
If $n=3$, we immediately obtain a contradiction between the stability inequality and the Gauss-Bonnet theorem, as in the proof of Lemma \ref{3-dim-general-PMT}. If $n>3$, then
by blowing-up any potential GMT singularities of the $\mu$-bubble, Theorem \ref{blow-up-2rd-GMT} gives rise to an $(n-1)$-data
set $(\Sigma^{n-1}\setminus\mathcal{S}_\Sigma,g'_\Sigma,\E_\Sigma, \rho', Q',\varepsilon)$ with 
\begin{equation}
m(\E_\Sigma,g'_\Sigma)=0,\quad\quad\quad \rho'=1 \text{ on }\E_\Sigma.
\end{equation}
In particular, the generalized mass of this data set vanishes, which yields a contradiction with Theorem \ref{general-PMT}.
\end{plainproof}

\section{Rigidity Statement of the Singular PMT}
\label{sec5}


The purpose of this section is to prove the rigidity statement in the equality
case of Theorem \ref{A}. In the smooth setting, the vanishing of the ADM mass
is usually treated by combining Ricci-flatness with harmonic coordinates and
then applying the classical splitting or Bartnik-type rigidity argument.
In the present singular setting, however, the lack of uniform Hessian control
for harmonic functions near the singular set prevents a direct application of
these methods. We overcome this difficulty by first proving Ricci-flatness on
the regular set \(M^n\setminus\mathcal S\), then constructing global harmonic
coordinates with H\"{o}lder control near \(\mathcal S\), and finally deriving
a weighted Hessian estimate using cut-offs adapted to the Minkowski dimension
bound on \(\mathcal S\). 

Throughout this section, it will be assumed for convenience that $M^n$ has a single end. The case of additional ends may be treated with minor modifications. 

\subsection{Ricci-flatness and harmonic coordinates}
\begin{lemma}\label{Ricci-flatness} 
Let $(M^n, g,\E)$ be a complete asymptotically flat $L^\infty$-manifold. Suppose that the singular set $\mathcal{S}$ is compact and satisfies the Minkowski dimension upper bound, $\mathrm{dim}_{\mathcal{M}}\mathcal{S} < n-3+\frac{2}{n}$. If the scalar curvature of $g$ is nonnegative on $M^n\setminus \mathcal{S}$, and the mass vanishes $m(\E,g)=0$, then the Ricci curvature vanishes on $M^n \setminus\mathcal{S}$.
\end{lemma}

\begin{proof} 
We first establish a weak version of Lemma \ref{Ricci-flatness} for scalar curvature. Assume by way of contradiction that there is a point $p\in M^n\setminus \mathcal{S}$ with $R_g(p)>0$.
Then there is a small geodesic ball $B_{2r}(p)\Subset M^n \setminus\mathcal{S}$ on which the scalar curvature is strictly positive. Choose a cut-off function $\varphi\in C_c^{\infty}(B_{2r}(p))$ with $\varphi=1$ on $B_{r}(p)$, and $0\leq \varphi\leq1$. Consider the following equation with prescribed asymptotics on the end $\E$:
\begin{equation}
    \Delta_g u-\frac{n-2}{4(n-1)}\varphi R_g u=0 \text{ on }M^n, \quad\quad\quad u\rightarrow 1 \text{ as } r\rightarrow \infty. 
\end{equation}
Theorem \ref{positive-solu} provides a positive solution $u\in C^{\infty}(M^n\setminus \mathcal{S})\cap W^{1,2}(M^n,g)\cap L^{\infty}(M^n)$ such that
\begin{equation}
u-\left(1+\frac{\mathcal{C}}{r^{n-2}}\right)\in C^{2, \alpha}_{2-n-q'}(\mathcal{E}) \quad\quad\mathcal{C}=-\frac{1}{4(n-1)\omega_{n-1}}\int_{M^n}\varphi R_g u \,dV_g <0.
\end{equation}
Moreover, $u$ is uniformly bounded below by a positive constant. 
Then the conformal change $\hat g=u^{\frac{4}{n-2}}g$, yields a complete asymptotically flat $L^\infty$-manifold $(M^n, \hat g, \E)$ satisfying
\begin{equation}
R_{\hat g}=u^{-\frac{4}{n-2}}(1-\varphi)R_g\geq 0 \text{ on }M^n \setminus\mathcal{S}, 
\qquad m(\E,\hat g)=m(\E, g)+2\mathcal{C}<0,
\end{equation}
giving a contradiction to the inequality statement of Theorem \ref{A}. Note that here we use the codimension thrshold of $3-\frac{2}{n}$, in order to apply the positive mass inequality. We conclude that 
the scalar curvature vanishes, $R_g=0$, away from the singular set. 


We next show that \(\operatorname{Ric}_g=0\) on \(M^n\setminus\mathcal S\), following the classic strategy of Schoen--Yau \cite{SchoenYau1979PMT}. Suppose, to the contrary, that \(\operatorname{Ric}_g(p)\neq0\) for some
\(p\in M^n\setminus\mathcal S\). Choose a geodesic ball
\(B_{2r}(p)\Subset M^n\setminus\mathcal S\), and a cut-off function
\(\varphi\in C_c^\infty(B_{2r}(p))\) as above. For
\(t>0\) sufficiently small, define
\begin{equation}
        g_t:=g-t\varphi\,\operatorname{Ric}_g .
\end{equation}
Then \(g_t=g\) outside \(B_{2r}(p)\). In particular, \(g_t\) has the same
asymptotics and the same ADM mass as \(g\):
\begin{equation}
        m(\mathcal E,g_t)=m(\mathcal E,g)=0 .
\end{equation}
Recall the linearization of the scalar curvature 
\begin{equation}
        \left.\frac{d}{dt}\right|_{t=0}R_{g_t}
        =
        -D R_g(\varphi\operatorname{Ric}_g),
\end{equation}
where
\begin{equation}
        D R_g(h)
        =
        -\Delta_g(\operatorname{tr}_g h)
        +\operatorname{div}_g\operatorname{div}_g h
        -\langle \operatorname{Ric}_g,h\rangle_g .
\end{equation}
Using that \(R_g=0\) and integrating by parts produces
\begin{equation}\label{tscalar}
        \int_{M^n} R_{g_t}\,dV_{g_t}
        =
        t\int_{M^n} \varphi |\operatorname{Ric}_g|^2\,dV_g
        +O(t^2)>0
\end{equation}
for all sufficiently small \(t>0\).

Now solve
\begin{equation}
        \Delta_{g_t}u_t
        -
        \frac{n-2}{4(n-1)}R_{g_t}u_t=0 \text{ on } M^n,
        \qquad\quad
        u_t\rightarrow 1
        \text{ as }r\rightarrow \infty.
\end{equation}
Note that for \(t>0\) sufficiently small, the negative part of \(R_{g_t}\) is small,
and thus Theorem \ref{positive-solu} applies to give a positive solution $u_t\in C^{\infty}(M^n\setminus \mathcal{S})\cap W^{1,2}(M^n,g_t)\cap L^{\infty}(M^n)$ such that
\begin{equation}
u_t-\left(1+\frac{\mathcal{C}_t}{r^{n-2}}\right)\in C^{2, \alpha}_{2-n-q'}(\mathcal{E}) \quad\quad\mathcal{C}=-\frac{1}{4(n-1)\omega_{n-1}}\int_{M^n} R_{g_t} u_t \,dV_{g_t}.
\end{equation}
Since \(u_t\to1\) as \(t\to0\), the positivity of the total scalar curvature \eqref{tscalar}
implies that $\mathcal C_t<0$ for all sufficiently small \(t>0\).

Next, define $\hat g_t:=u_t^{\frac{4}{n-2}}g_t$ and observe that these metrics are scalar flat $R_{\hat g_t}=0$, on $M^n\setminus\mathcal S$. Furthermore, \((M^n,\hat g_t,\mathcal E)\) is again a complete asymptotically flat
\(L^\infty\)-manifold with mass
\begin{equation}
        m(\mathcal E,\hat g_t)
        =
        m(\mathcal E,g_t)+2\mathcal C_t
        =
        2\mathcal C_t<0 .
\end{equation}
This contradicts the inequality statement of Theorem \ref{A}. Hence,
\(\operatorname{Ric}_g=0\) away from the singular set.
\end{proof}

On the asymptotically flat end \(\mathcal E\), there are Cartesian coordinates
\(x\!=\!(x^1 \!,\ldots,x^n)\) such that
\begin{equation}
        g_{ij}(x)-\delta_{ij}=O_2(r^{-q}), \quad\quad\quad  \Delta_g x^i=O_1(r^{-1-q}),
        \qquad i,j=1,\ldots,n ,      
\end{equation}
where $q>\frac{n-2}{2}$. Thus each coordinate function \(x^i\) is asymptotically harmonic, and the error may be
solved away to produce global harmonic functions $y^i$, referred to as \textit{harmonic coordinates}, which play
an important role in the rigidity argument.

\begin{lemma}\label{lemma:harmonicCoordinates}
Let $(M^n, g,\E)$ be a complete asymptotically flat $L^\infty$-manifold. Suppose that the singular set $\mathcal{S}$ is compact and satisfies the Minkowski dimension upper bound, 
$\dim_{\mathcal{M}} \mathcal{S}<n-2$. If the order of decay of this end satisfies $q\in (\frac{n-2}{2}, n-2)$, then there exist functions $y^i\in C^{\infty}(M^n\setminus \mathcal{S})$ such that 
\begin{equation}\label{exist-harmonic-func}
    \Delta_g y^i=0 \text{ on } M^n, \qquad y^i-x^i\in C^{2, \alpha}_{1-q}(\E),\qquad i=1,\dots,n,
\end{equation} 
for $\alpha\in (0, 1)$. Moreover, $y^i\in C_{loc}^{0,\gamma}(M^n)$ for some $\gamma\in (0,1)$, and there exists a constant $C>0$ such that 
\begin{equation}
    |\nabla y^i(p)|_g\leq C \max\{ d_{g_0}(p, \mathcal{S})^{\gamma-1},1 \},\quad\quad p\in M^n \setminus\mathcal{S}.
\end{equation}
\end{lemma}

\begin{proof} 
Choose a smooth cut-off function \(\zeta\) which is equal to \(1\) on the
asymptotic end outside a compact set and which vanishes on a neighborhood of
the singular set. Set $\tilde{x}^i:=\zeta x^i$.
Then \(\tilde{x}^i=x^i\) near infinity and, by the asymptotic flatness of the end,
\begin{equation}
        \Delta_g \tilde{x}^i=O_1(r^{-q-1})
        \quad\text{on } \mathcal E .
\end{equation}
We solve
\begin{equation}
        \Delta_g w^i=-\Delta_g \tilde{x}^i \quad\text{on }M^n,
\end{equation}
with \(w^i=o(r)\) on the asymptotic end. Since
\(\dim_{\mathcal M}\mathcal S<n-2\), the set \(\mathcal S\) has zero
\(W^{1,2}\)-capacity; hence no boundary condition is imposed along
\(\mathcal S\), and the weak equation across \(\mathcal S\) is equivalent to
the equation on \(M^n\setminus\mathcal S\). The weighted elliptic theory on
asymptotically flat manifolds \cite[Theorem~A.40]{Lee}, implies that $w^i\in C^{2,\alpha}_{1-q}(\mathcal E)$.
Define $y^i:=\tilde{x}^i+w^i$, then this function is weakly harmonic on $M^n$, is smooth on $M^n \setminus\mathcal{C}$,
and satisfies $y^i-x^i=w^i\in C^{2,\alpha}_{1-q}(\mathcal E)$.

It remains to show the H\"{o}lder regularity of $y^i$ and the gradient estimate near $\mathcal{S}$. Since $(M^n,g)$ is an $L^\infty$-manifold, the operator $\Delta_g$ may be written as a divergence-form elliptic operator with $L^\infty$ coefficients that are uniformly elliptic.  Standard De Giorgi--Nash--Moser theory \cite[Theorem 8.9]{Han} applies, even across the singular set, to yield $\gamma\in(0,1)$ such that $y^i \in C^{0,\gamma}(M^n)$.

We now prove the desired gradient estimate. Because the metric $g$ is smooth outside $\mathcal{S}$, standard gradient estimates for harmonic functions give a uniform bound for $|\nabla y^i|_g$ outside of $N_{r_0}(\mathcal{S})$, the closed $r_0$-neighborhood of $\mathcal{S}$ with respect to the background metric $g_0$, where $r_0>0$ is fixed. Take $p\in N_{r_0}(\mathcal{S})\setminus \mathcal{S}$, and set $r_p:=d_{g_0}(p,\mathcal{S})>0$. Then $B_{r_p /2}(p)\subset M^n \setminus\mathcal{S}$. Since the Ricci curvature vanishes in this ball, we can apply the 
Cheng--Yau gradient estimate \cite[Corollary 3.2, Chapter I]{SchoenYau1994Lectures} to obtain 
\begin{equation}\label{fahjg09haq09hh}
    |\nabla y^i(p)|_g\leq \sup_{B_{r_p /4}}|\nabla(y^i-\inf_{B_{r_p /2}(p)} y^i)|_g\leq cr_p^{-1} \sup_{B_{r_p /2}(p)}|y^i-\inf_{B_{r_p /2}(p)} y^i|.  
\end{equation} 
Furthermore, by the H\"{o}lder estimate above
\begin{equation}\label{-09}
\sup\limits_{p', p''\in B_{r_p /2}(p)}|y^i(p')-y^i(p'')|\leq c_0 d_{g_0}(p',p'')^\gamma\leq c_1 r_p^{\gamma}. 
\end{equation}
Combining \eqref{fahjg09haq09hh} and \eqref{-09} yields the desired result. We note that the final constant obtained is uniform,
since $N_{r_0}(\mathcal{S})$ is compact.
\end{proof}




Define functions $g_{ij}=g(\partial_{y^i}, \partial_{y^j})$ and $g^{ij}=g(dy^i, dy^j)$, then by Bochner's formula
\begin{equation}
    \frac{1}{2}\Delta_g g^{ii}=\frac{1}{2}\Delta_g |\nabla y^i|^2_g=|\nabla^2 y^i|_g^2+\operatorname{Ric}_g(\nabla y^i, \nabla y^i)=|\nabla^2 y^i|_g^2\in C^{0,\alpha}_{-2-2q}(\E),
\end{equation} 
where $q\in (\frac{n-2}{2}, n-2)$. Therefore, the basic asymptotic analysis of \cite[Corollary A.38]{Lee} shows that 
\begin{equation}\label{5.20}
g^{ij}-\left(\delta_{ij}-\frac{c_{ij}}{r^{n-2}}\right)\in C^{2, \alpha}_{2-n-q_0}(\E)
\end{equation}
for some set of constants $c_{ij}=c_{ji}$, where $q_0>0$. By performing a rotation of $y$-coordinates if necessary, it may be assumed without loss of generality that $c_{ij}=c_i \delta_{ij}$ for $i,j=1,\ldots, n$. The metric satisfies the following asymptotics 
\begin{equation}\label{metric-expansion-harmonic}
    g_{ij}=\left(1+\frac{c_i}{r^{n-2}}\right)\delta_{ij}+O_2(r^{2-n-q_0}).
\end{equation}

\begin{corollary}\label{mass-under-harmonic-coord} 
Under the hypotheses and setting of Lemma \ref{lemma:harmonicCoordinates}, it holds that
\begin{equation}
    m(\E,g)=\frac{(n-2)}{2n}\sum^n_{i=1}c_i.
\end{equation}
\end{corollary}

\begin{proof}
In the harmonic coordinates \((y^1,\ldots,y^n)\), we have the following expansions
\begin{equation}
     \sum_{i=1}^{n}   \partial_i g_{ij}
        =
        (2-n)c_j r^{1-n}\nu^j
        +
        O(r^{1-n-q_0}),
\end{equation}
and
\begin{equation}
\sum_{i=1}^{n}  \partial_j g_{ii}
        =
        (2-n)\left(\sum_{i=1}^n c_i\right)r^{1-n}\nu^j
        +
        O(r^{1-n-q_0}),
\end{equation}
where $\nu^j=\frac{y^j}{|y|}$. Therefore, the ADM mass flux density satisfies
\begin{align}
        \sum_{i,j=1}^n
        \bigl(\partial_i g_{ij}-\partial_j g_{ii}\bigr)\nu^j
        =
        (n-2)
        \left(
        \sum_{i=1}^n c_i
        -
        \sum_{i=1}^n c_i(\nu^i)^2
        \right)r^{1-n}  
         +O(r^{1-n-q_0}).
\end{align}
Integrating over the coordinate sphere \(S_r\) and using
\begin{equation}
        \int_{S_r}(\nu^i)^2\,dA_\delta
        =
        \frac{1}{n}\omega_{n-1}r^{n-1},
\end{equation}
we obtain
\begin{equation}
        \int_{S_r}
        \sum_{i,j=1}^n
        \bigl(\partial_i g_{ij}-\partial_j g_{ii}\bigr)\nu^j\,dA_\delta
        =\frac{(n-2)(n-1)}{n}
        \omega_{n-1}
        \sum_{i=1}^n c_i
        +o(1),
\end{equation}
from which the desired result follows.
\end{proof}

\subsection{Proof of rigidity in Theorem \ref{A}} 
\begin{plainproof}
Let $(M^n, g, \E)$ be a complete asymptotically flat $L^\infty$-manifold with a single end $\E$. 
Assume that the singular set $\mathcal{S}$ is compact and satisfies the Minkowski dimension upper bound, $\dim_\mathcal{M}\mathcal{S}\leq n-3+(n-1)^{-1}$, and that
\begin{equation}
    R_g\geq 0 \text{ on } M^n\setminus \mathcal{S}, \quad\qquad m(\E,g)=0. 
\end{equation}
Note that since $(n-1)^{-1}<2n^{-1}$, the hypotheses for the inequality portion of Theorem \ref{A} are satisfied.
According to Lemma \ref{Ricci-flatness}, we have $\operatorname{Ric}_g=0$ on $M^n\setminus \mathcal{S}$. From Lemma \ref{lemma:harmonicCoordinates}, there are global harmonic functions $\{y^i\}_{i=1}^{n}$ on $M^n$, and on $\E$ these functions form a coordinate system such that
\begin{equation}
     g_{ij}-\left(1+\frac{c_i}{r^{n-2}}\right)\delta_{ij}\in C^{2,\alpha}_{2-n-q_0}(\E).
\end{equation}

The presence of the singular set $\mathcal S$ prevents us from obtaining
the standard  Hessian estimates for the harmonic coordinates
$y^i$ near $\mathcal S$. Consequently, the classical rigidity argument
does not apply directly in the present setting. We overcome this difficulty
by deriving new estimates which control the relevant gradient and Hessian
terms. To this end, we introduce the following quantities
\begin{equation}
    u^i=(|\nabla y^i|_g^2+1)^b,\quad\quad\quad
    b=\frac1 2-\frac{1}{2(n-1)}+\frac{n\epsilon}{2(n-1)},
\end{equation}
where $\epsilon>0$ is an arbitrarily small parameter. A direct computation yields that on $M^n\setminus \mathcal{S}$,
\begin{align}\label{hessian-laplacian}
 \Delta_g u^i
 =&b (|\nabla y^i|_g^2+1)^{b-1}\Delta_g|\nabla y^i|^2+4b(b-1)(|\nabla y_i|_g^2+1)^{b-2}|\nabla y^i|_g^2|\nabla|\nabla y^i|_g|_g^2 \nonumber\\
 \geq& 2b(|\nabla y^i|_g^2+1)^{b-1}\left(|\nabla^2 y^i|_g^2+2(b-1)|\nabla|\nabla y^i|_g|_g^2\right)\nonumber\\
 \geq &2b(|\nabla y^i|_g^2+1)^{b-1}\left(1+2(b-1)\frac{n-1}{n}\right)|\nabla^2 y^i|_g^2 \\
 =&2b\epsilon(|\nabla y^i|_g^2+1)^{b-1}|\nabla^2 y^i|_g^2\nonumber
\end{align}
where we have used the Bochner's formula and the refined Kato inequality 
\begin{equation}
|\nabla^2 y^i|_g^2\geq \frac{n}{n-1}|\nabla|\nabla y^i|_g|_g^2.
\end{equation}

Choose the same cut-off function $\chi_r \in C^{\infty}_c(M^n)$ used in the proof of Proposition \ref{prop2.3}, which satisfies
$0\leq \chi_r\leq 1$ and additionally
\begin{equation}\label{cut-off-near-sing}
\chi_r=1 \text{ on } N_r(\mathcal{S}), \quad\quad \supp\, \chi_r\subset N_{2r}(\mathcal{S})
\quad\quad \int_{M^n}|\nabla\chi_r|_g^2\leq C_\delta {r}^{1-\delta}, 
\end{equation} 
where $\delta\in (\frac{1}{n-1},1)$.  Set $r\leq r_0/4$ and $\eta_r=1-\chi_r$, where $r_0$ is as in the proof of Lemma \ref{lemma:harmonicCoordinates}. According to Lemma \ref{lemma:harmonicCoordinates} we have
\begin{equation}\label{harmonic-gradient-near-sing}
    |\nabla y^i|\leq C r^{\gamma-1} \quad\text{ on }N_{2r}(\mathcal{S})\setminus N_{r}(\mathcal{S}).
\end{equation} 
Multiplying both sides of \eqref{hessian-laplacian} by $\eta_r^2$, integrating by parts, and applying Young's inequality produces
\begin{align}
 &\int_{M^n}2\epsilon b\eta_r^2(|\nabla y_i|_g^2+1)^{b-1}|\nabla^2 y^i|_g^2 dV_g \nonumber\\
 &=\lim_{R\rightarrow \infty}\int_{S_R}\nu(u^i)dA_g-2\int_{N_{2r}(\mathcal{S})\setminus N_r(\mathcal{S})}\!\eta_r g(\nabla\eta_r,\nabla u^i)dV_g\\
    &\leq (n-2)c_i \omega_{n-1}+4b\int_{N_{2r}(\mathcal{S})\setminus N_r(\mathcal{S})} \!\!\eta_r|\nabla\eta_r|_g(|\nabla y_i|_g^2+1)^{b-1}|\nabla y^i|_g|\nabla|\nabla y^i|_g|_g dV_g\nonumber\\
    &\leq (n-2)c_i \omega_{n-1}+\epsilon b\int_{M^n} \eta_r^2(|\nabla y_i|_g^2+1)^{b-1}|\nabla^2 y^i|_g^2 dV_g \nonumber\\
    &\quad+\frac{4b}{\epsilon}\int_{N_{2r}(\mathcal{S})\setminus N_r(\mathcal{S})}\!|\nabla \eta_r|_g^2(|\nabla y^i|_g^2+1)^b dV_g, \nonumber
\end{align} 
where $\nu$ is the unit outer normal to the coordinate sphere $S_R$, and we have used \eqref{5.20} in the third line.
Next employ \eqref{cut-off-near-sing} and \eqref{harmonic-gradient-near-sing} to obtain
\begin{align}
    \epsilon b\!\int_{M^n}\!\! \eta^2_r(|\nabla y_i|_g^2+1)^{b-1}|\nabla^2 y^i|_g^2 dV_g &\leq  (n-2)c_i \omega_{n-1}+C_\epsilon  r^{2b(\gamma-1)}\int_{M^n}\! |\nabla\chi_r|_g^2 dV_g \nonumber\\
    &\leq (n-2)c_i \omega_{n-1}+C_\epsilon C_\delta r^{2b(\gamma-1)+1-\delta}. 
\end{align}
Now choose $\epsilon$ sufficiently small, and $\delta$ sufficiently close to $\frac{1}{n-1}$ to achieve
\begin{equation}
2b(\gamma-1)+1-\delta>0.
\end{equation}
By sending $r\rightarrow 0$ it follows that 
\begin{equation}
    \epsilon b\sum_{i=1}^n\int_{M^n}(|\nabla y_i|_g^2+1)^{b-1}|\nabla^2 y^i|_g^2 dV_g\leq (n-2)\omega_{n-1}\sum^{n}_{i=1}c_i=0,
\end{equation}
where we have used Corollary \ref{mass-under-harmonic-coord} and $m(\E,g)=0$.  Thus, 
\begin{equation}\label{h1}
|\nabla^2 y^i|_g=0\quad \text{ on }M^n\setminus \mathcal{S}, \quad i=1,\dots,n.
\end{equation}

We now complete the rigidity argument. Since
\begin{equation}
        \dim_{\mathcal M}\mathcal S
        \le n-3+\frac{1}{n-1}<n-1,
\end{equation}
we have \(\dim_{\mathcal H}\mathcal S<n-1\).
By Szpilrajn's theorem \cite[Chapter VII, Section 4]{HurewiczWallman1941DimensionTheory}, the topological dimension (an integer) of a metric space is bounded above by its Hausdorff dimension, so that $\dim_{\mathrm{top}}\mathcal S\le n-2$.
According to the classical dimension-theoretic separation theorem, a closed subset of
an \(n\)-manifold of topological dimension at most \(n-2\) does not separate
the manifold; see \cite[Theorem~1.8.13]{Engelking1978DimensionTheory}.
Hence \(M^n\setminus\mathcal S\) is connected. Moreover, in light of \eqref{h1},
the vector fields \(\nabla y^i\) are
parallel on \(M^n\setminus\mathcal S\). It follows that the functions $g^{ij}:=g(\nabla y^i,\nabla y^j)$
are constant on \(M^n\setminus\mathcal S\); note that connectedness is used here. Since \(g^{ij}\to\delta^{ij}\) on
the asymptotically flat end, we obtain
\begin{equation}
        g(\nabla y^i,\nabla y^j)=\delta^{ij}
        \quad\text{on }M^n\setminus\mathcal S .
\end{equation}
Hence, the map
\begin{equation}
        \Psi=(y^1,\ldots,y^n): M^n \rightarrow\mathbb{R}^n,
\end{equation}
yields
\begin{equation}
        g=\sum_{i=1}^n dy^i\otimes dy^i
        =\Psi^* \delta
        \quad\text{on }M^n\setminus\mathcal S .
\end{equation}
In particular, $\Psi|_{M^n\setminus\mathcal S}$
is a smooth local isometry onto its image.

We now show that $\Psi$ is surjective. Since \(\Psi(x)=x+o(r)\) on the unique asymptotically flat end, \(\Psi\) is
proper. Moreover, on each sufficiently large coordinate sphere \(S_R\subset
\mathcal E\), the map
\begin{equation}
        x\mapsto \frac{\Psi(x)}{|\Psi(x)|}
\end{equation}
is homotopic to the standard radial map \(x\mapsto x/|x|\). Hence the
induced boundary map has mod \(2\) degree one. By the relative degree
formula, the proper continuous map $\Psi:M^n\to\mathbb R^n$
has mod \(2\) topological degree one. This does not require orientability of
\(M\) or \(C^1\)-regularity of \(\Psi\) along \(\mathcal S\). Thus $\Psi$ is onto.

We next prove that \(\Psi\) is injective. Set
\begin{equation}
        A:=\mathbb R^n\setminus\Psi(\mathcal S),
        \qquad
        \mathcal U:=\Psi^{-1}(A),
\end{equation}
and note that $\mathcal U\subset M^n\setminus\mathcal S$.
Thus, the restricted map \(\Psi:\mathcal U\to A\) is a smooth local isometry.
We first observe that \(A\) is connected. Indeed, after the vanishing of the
Hessians, the functions \(y^i\) have uniformly bounded gradients with
respect to a fixed smooth background metric. Hence \(\Psi\) is 
Lipschitz on \(M^n\). Therefore
\begin{equation}
        \dim_{\mathcal H}\Psi(\mathcal S)
        \le
        \dim_{\mathcal H}\mathcal S
        <n-1 .
\end{equation}
By the dimension-theoretic separation theorem, \(\Psi(\mathcal S)\) does not
disconnect \(\mathbb R^n\). Thus $A=\mathbb R^n\setminus\Psi(\mathcal S)$
is connected.

We claim that the restricted map is a global isometry. To see this, recall that $\Psi:\mathcal U\to A$ is a proper local homeomorphism. Moreover, since \(\Psi\) is surjective, this restriction map is onto. Therefore it 
is a covering map. On the asymptotically flat end, the functions \(y^i\)
agree with the asymptotic coordinates up to lower-order terms. Hence, after
passing sufficiently far out in the end, \(\Psi\) is one-to-one and its image
contains the complement of a large compact set in \(\mathbb R^n\). Thus the
covering has one sheet near infinity. Since \(A\) is connected, the number of
sheets of the covering is constant. Consequently, $\Psi:\mathcal U\to A$
is a one-sheeted covering, and yields a global isometry.

It remains to rule out the possibility that two distinct points of $M^n$ map to a single point of
\(\Psi(\mathcal S)\). Let
\begin{equation}
        B:=M^n\setminus\mathcal U=\Psi^{-1}(\Psi(\mathcal S)).
\end{equation}
Note that \(B\) is closed since \(\Psi(\mathcal S)\) is compact, as it is the image of the compact set
\(\mathcal S\) and \(\Psi\) is continuous.  Moreover, $\dim_{\mathcal H} B<n-1$.
Indeed, \(B\cap\mathcal S\subset\mathcal S\), while on
\(M^n\setminus\mathcal S\) the map \(\Psi\) is a local diffeomorphism; hence,
locally on the regular set, the preimage of \(\Psi(\mathcal S)\) has the
same Hausdorff dimension as \(\Psi(\mathcal S)\), which is strictly less than
\(n-1\). Since \(g\) is uniformly equivalent to a smooth background metric, removing
a closed set of Hausdorff dimension strictly less than \(n-1\) does not
change the metric completion. Thus the metric completion of
\((\mathcal U,g)\) is \((M^n,d_g)\). Similarly, the metric completion of
\((A,\delta)\) is \((\mathbb R^n,d_\delta)\). The isometry $\Psi:\mathcal U\to A$
therefore extends uniquely to an isometry between the metric completions,
\begin{equation}\label{isometry}
        \Psi:(M^n,d_g)\to(\mathbb R^n,d_\delta).
\end{equation}
This extension agrees with the original continuous map \(\Psi\). This, of course, implies global injectivity,
since if \(p,q\in M^n\) satisfy \(\Psi(p)=\Psi(q)\), then $d_g(p,q)=d_\delta(\Psi(p),\Psi(q))=0$.
 
We now have that \(\Psi:M^n \rightarrow\mathbb{R}^n\) is a continuous bijection; in fact, as noted above $\Psi$ is Lipschitz with respect to the background metric $g_0$, and the same is true of its inverse. Moreover, the extended map
\eqref{isometry} is an isometry, so its inverse is also an isometry and is therefore continuous.
Hence \(\Psi\) is a homeomorphism, and
\begin{equation}
        \Psi|_{M^n\setminus\mathcal S}:
        M^n\setminus\mathcal S
        \longrightarrow
        \mathbb R^n\setminus\Psi(\mathcal S)
\end{equation}
is a smooth isometry.

The case of multiple ends may be treated with minor modifications.
We briefly indicate the steps here. If \(\mathcal E_0\) is the zero-mass end, then the harmonic
functions \(y^i\) may be chosen so that
\begin{equation}
        y^i-x^i\to0 \quad\text{on }\mathcal E_0,
        \qquad
        y^i\to0 \quad\text{on every other end}.
\end{equation}
Once it is established that \(|\nabla^2 y^i|_g=0\) on \(M^n\setminus\mathcal S\), the quantities
\(g(\nabla y^i,\nabla y^j)\) are constant on the connected set
\(M^n \setminus\mathcal S\). Their limits are \(\delta^{ij}\) on
\(\mathcal E_0\), but would be \(0\) on any other end, yielding a contradiction to the presence
of additional ends. Thus the multiple-end case reduces to the single-end case considered above.
\end{plainproof}

\begin{example}[Failure of smooth removability under the \(L^\infty\) hypothesis]
\label{example:flat-Linfty-nonsmooth}
The metric-space rigidity conclusion cannot in general be strengthened to
smooth removability of the singular set under only an \(L^\infty\) assumption
on the metric. To show this we provide an example.

Let $M^n=\mathbb R^n$ and $\mathcal S=\{0\}$.
Choose a smooth diffeomorphism
\begin{equation}
        \phi:S^{n-1}\to S^{n-1},
\end{equation}
which is isotopic to the identity but is not the restriction of a linear
orthogonal transformation. Let \(\phi_s\), \(s\in[0,1]\), be a smooth isotopy
from the identity to \(\phi\). Choose a smooth nonnegative cut-off function
\(\chi:[0,\infty)\to[0,1]\) such that \(\chi=1\) for \(r\le1\) and
\(\chi=0\) for \(r\ge2\). Define a map
\begin{equation}
        \Psi:\mathbb R^n\to\mathbb R^n
\end{equation}
by \(\Psi(0)=0\) and, for \(x=(r,\theta)\), \(r>0\), \(\theta\in S^{n-1}\),
\begin{equation}
        \Psi(r,\theta)=r\,\phi_{\chi(r)}(\theta).
\end{equation}
Then \(\Psi\) is a bi-Lipschitz homeomorphism of \(\mathbb R^n\), is a smooth
diffeomorphism on \(\mathbb R^n\setminus\{0\}\), and agrees with the identity
map outside a compact set. However, \(\Psi\) is not \(C^1\) at the origin:
if \(D\Psi(0)\) existed, then
\begin{equation}
        \lim_{r\to0}\frac{\Psi(r,\theta)}{r}=\phi(\theta)
\end{equation}
would have to be the restriction to \(S^{n-1}\) of a linear map, contrary to
the choice of \(\phi\).

Now define
\begin{equation}
        g:=\Psi^*\delta
        \quad\text{on }\mathbb R^n\setminus\{0\}.
\end{equation}
Then \(g\) is smooth and flat on \(M^n\setminus\mathcal S\), is uniformly
equivalent to the Euclidean metric, and hence defines an \(L^\infty\)-metric
on \(M^n\). Moreover, since \(\Psi=\mathrm{Id}\) outside a compact set, the
metric \(g\) is asymptotically flat with zero ADM mass, and
\begin{equation}
        R_g=0
        \quad\text{on }M^n\setminus\mathcal S.
\end{equation}
The singular set satisfies $\dim_{\mathcal M}\mathcal S=0$.
Furthermore,
\begin{equation}
        \Psi:
        (\mathbb R^n\setminus\{0\},g)
        \longrightarrow
        (\mathbb R^n\setminus\{0\},\delta)
\end{equation}
is a smooth isometry, and \(\Psi\) extends to a homeomorphism of the metric
completions.

Nevertheless, \(g\) does not extend smoothly across the origin in the
original smooth structure. Indeed, near the origin
\begin{equation}
        g=dr^2+r^2\phi^*g_{S^{n-1}} .
\end{equation}
If \(g\) extended smoothly, as a Riemannian metric at
the origin in the given smooth structure, the limiting inner product at
the origin would determine the standard round metric on infinitesimal
spheres. This would force
\begin{equation}
        \phi^*g_{S^{n-1}}=g_{S^{n-1}},
\end{equation}
so that \(\phi\) is an isometry of the round sphere, contradicting the choice
of \(\phi\). Thus the singularity is removable in the metric-space sense, but
not necessarily in the smooth sense.
\end{example}

\bibliographystyle{plain}
\bibliography{ref}

@unpublished{Wang-Wang-Xie,
	author = {Jian Wang and Jinmin Wang and Zhizhang Xie},
	eprint = {2606.21325},
	note = {arXiv:2606.21325},
	title = {${L}^\infty$-metrics on tori and Schoen's conjecture},
	url = {https://arxiv.org/pdf/2606.21325.pdf},
	year = {2026}}

@unpublished{Bi-Zhu,
	author = {Yuchen Bi and Jintian Zhu},
	eprint = {2606.20528},
	note = {arXiv:2606.20528},
	title = {Positive Scalar Curvature Obstructions via Singular Dimension Descent},
	url = {https://arxiv.org/pdf/2606.20528.pdf},
	year = {2026}}

@article{chodosh2024generalized,
  title={Generalized soap bubbles and the topology of manifolds with positive scalar curvature},
  author={Chodosh, Otis and Li, Chao},
  journal={Annals of Mathematics},
  volume={199},
  number={2},
  pages={707--740},
  year={2024},
  publisher={Department of Mathematics, Princeton University Princeton, New Jersey, USA}
}

@article {bartnik1986,
    AUTHOR = {Bartnik, Robert},
     TITLE = {The mass of an asymptotically flat manifold},
   JOURNAL = {Comm. Pure Appl. Math.},
  FJOURNAL = {Communications on Pure and Applied Mathematics},
    VOLUME = {39},
      YEAR = {1986},
    NUMBER = {5},
     PAGES = {661--693},
      ISSN = {0010-3640,1097-0312},
   MRCLASS = {58G30 (46E35 53C80 58D30 83C30)},
MRREVIEWER = {M.\ A. H. MacCallum},
       DOI = {10.1002/cpa.3160390505},
       URL = {https://doi.org/10.1002/cpa.3160390505},
}

@article {Chrusciel,
    AUTHOR = {Chru\'sciel, Piotr},
    TITLE = {Boundary conditions at spatial infinity from a {H}amiltonian
              point of view},
 BOOKTITLE = {Topological properties and global structure of space-time
              ({E}rice, 1985)},
    JOURNAL = {NATO Adv. Sci. Inst. Ser. B: Phys.,},
    VOLUME = {138},
    YEAR = {1986},
    PAGES = {49--59},   
}

@book{Lee,
  title={Geometric {R}elativity},
  author={Lee, Dan},
  volume={201},
  year={2021},
  publisher={American Mathematical Society}
}

@article{F,
	Author = {D. {Fischer-Colbrie}},
	Journal = {{Invent. Math.}},
	Pages = {121--132},
	Title = {{On complete minimal surfaces with finite Morse index in three manifolds.}},
	Volume = {82},
	Year = {1985}}

@article{Gromov1,
	Author = {Gromov, Misha},
	Journal = {Geometric and Functional Analysis},
	Number = {3},
	Pages = {645-726},
	Title = {Metric inequalities with scalar curvature},
	Volume = {28},
	Year = {2018}}

@book{Han,
	Author = {Han, Qing and Lin, Fanghua},
	Publisher = {American Mathematical Soc.},
	Title = {Elliptic partial differential equations},
	Volume = {1},
	Year = {2011}}

@book{GT,
	Author = {Gilbarg, David and Trudinger, Neil},
	Publisher = {Springer},
	Title = {Elliptic partial differential equations of second order},
	Year = {2015}}

@book{S,
	Author = {Stallings, John},
	Publisher = {Yale University Press},
	Title = {Group theory and three-dimensional manifolds},
	Year = {1972}}

@inproceedings{And,
	Author = {Anderson, Michael},
	Booktitle = {Annales scientifiques de l'{\'E}cole Normale Sup{\'e}rieure},
	Organization = {Elsevier},
	Pages = {89-105},
	Title = {Curvature estimates for minimal surfaces in $3 $-manifolds},
	Volume = {18},
	Year = {1985}}

@article{Donaldson-Sun,
 author = {Donaldson, Simon and Sun, Song},
 title = {Gromov-{Hausdorff} limits of {K{\"a}hler} manifolds and algebraic geometry},
 fjournal = {Acta Mathematica},
 journal = {Acta Math.},
 issn = {0001-5962},
 volume = {213},
 number = {1},
 pages = {63--106},
 year = {2014},
 language = {English},
 doi = {10.1007/s11511-014-0116-3},
 zbMATH = {6381225},
 Zbl = {1318.53037}
}

@unpublished{DaiWangWangWei2024RCD,
	author = {Xianzhe Dai and Changliang Wang and Lihe Wang and Guofang Wei},
	eprint = {2412.09185},
	note = {arXiv:2412.09185},
	title = {Singular metrics with nonnegative scalar curvature and RCD},
	year = {2024}}

@unpublished{BiHaoHeShiZhu26,
	author = {Yuchen Bi and Tianze Hao and Shihang He and Yuguang Shi and Jintian Zhu},
	eprint = {2603.02769},
	note = {arXiv:2603.02769},
	title = {A proof for the Riemannian positive mass theorem up to dimension 19},
	url = {https://arxiv.org/pdf/2603.02769.pdf},
	year = {2026}}

@unpublished{WangXie2024SphereSubsetsLinfty,
	author = {Jinmin Wang and Zhizhang Xie},
	eprint = {2407.21312},
	note = {arXiv:2407.21312},
	title = {Scalar curvature rigidity of spheres with subsets removed and {$L^\infty$} metrics},
	url = {https://arxiv.org/pdf/2407.21312.pdf},
	year = {2024}}

@article {MR657523,
    AUTHOR = {Gr\"uter, Michael and Widman, Kjell-Ove},
     TITLE = {The {G}reen function for uniformly elliptic equations},
   JOURNAL = {Manuscripta Math.},
  FJOURNAL = {Manuscripta Mathematica},
    VOLUME = {37},
      YEAR = {1982},
    NUMBER = {3},
     PAGES = {303--342},
      ISSN = {0025-2611,1432-1785},
   MRCLASS = {35J25 (31B35)},
  MRNUMBER = {657523},
MRREVIEWER = {V.\ M.\ Isakov},
       DOI = {10.1007/BF01166225},
       URL = {https://doi.org/10.1007/BF01166225},
}

@article {MR161019,
    AUTHOR = {Littman, W. and Stampacchia, G. and Weinberger, H. F.},
     TITLE = {Regular points for elliptic equations with discontinuous
              coefficients},
   JOURNAL = {Ann. Scuola Norm. Sup. Pisa Cl. Sci. (3)},
  FJOURNAL = {Annali della Scuola Normale Superiore di Pisa. Classe di
              Scienze. Serie III},
    VOLUME = {17},
      YEAR = {1963},
     PAGES = {43--77},
      ISSN = {0391-173X},
   MRCLASS = {35.42},
  MRNUMBER = {161019},
MRREVIEWER = {C.\ B.\ Morrey, Jr.},
}

@article{SchoenYau1979PMT,
  author  = {Schoen, Richard and Yau, Shing-Tung},
  title   = {On the proof of the positive mass conjecture in general relativity},
  journal = {Communications in Mathematical Physics},
  volume  = {65},
  number  = {1},
  year    = {1979},
  pages   = {45--76}
}

@article{SchoenYau1981PMTII,
  author  = {Schoen, Richard and Yau, Shing-Tung},
  title   = {Proof of the positive mass theorem. {II}},
  journal = {Communications in Mathematical Physics},
  volume  = {79},
  number  = {2},
  year    = {1981},
  pages   = {231--260}
}

@article{Witten1981,
  author  = {Witten, Edward},
  title   = {A new proof of the positive energy theorem},
  journal = {Communications in Mathematical Physics},
  volume  = {80},
  number  = {3},
  year    = {1981},
  pages   = {381--402}
}

@incollection{SchoenYau2019MinimalSingularities,
  author    = {Schoen, Richard and Yau, Shing-Tung},
  title     = {Positive scalar curvature and minimal hypersurface singularities},
  booktitle = {Surveys in Differential Geometry 2019. Differential Geometry, Calabi--Yau Theory, and General Relativity. Part 2},
  series    = {Surveys in Differential Geometry},
  volume    = {24},
  publisher = {International Press},
  address   = {Boston, MA},
  year      = {2022},
  pages     = {441--480},
  doi       = {10.4310/SDG.2019.v24.n1.a10}
}

@article{Miao2002,
  author  = {Miao, Pengzi},
  title   = {Positive mass theorem on manifolds admitting corners along a hypersurface},
  journal = {Advances in Theoretical and Mathematical Physics},
  volume  = {6},
  number  = {6},
  year    = {2002},
  pages   = {1163--1182},
  doi     = {10.4310/ATMP.2002.v6.n6.a4}
}

@article{McFeronSzekelyhidi2012,
  author  = {McFeron, Donovan and Sz{\'e}kelyhidi, G{\'a}bor},
  title   = {On the positive mass theorem for manifolds with corners},
  journal = {Communications in Mathematical Physics},
  volume  = {313},
  number  = {2},
  year    = {2012},
  pages   = {425--443},
  doi     = {10.1007/s00220-012-1498-8}
}

@article{Lee2013,
  author  = {Lee, Dan},
  title   = {A positive mass theorem for Lipschitz metrics with small singular sets},
  journal = {Proceedings of the American Mathematical Society},
  volume  = {141},
  number  = {11},
  year    = {2013},
  pages   = {3997--4004},
  doi     = {10.1090/S0002-9939-2013-11672-2}
}

@article{LeeLeFloch2015,
  author  = {Lee, Dan and LeFloch, Philippe},
  title   = {The positive mass theorem for manifolds with distributional curvature},
  journal = {Communications in Mathematical Physics},
  volume  = {339},
  number  = {1},
  year    = {2015},
  pages   = {99--120},
  doi     = {10.1007/s00220-015-2414-9}
}

@article{HirschMiao2020,
  author  = {Hirsch, Sven and Miao, Pengzi},
  title   = {A positive mass theorem for manifolds with boundary},
  journal = {Pacific Journal of Mathematics},
  volume  = {306},
  number  = {1},
  year    = {2020},
  pages   = {185--201},
  doi     = {10.2140/pjm.2020.306.185}
}

@article{MantoulidisMiaoTam2020,
  author  = {Mantoulidis, Christos and Miao, Pengzi and Tam, Luen-Fai},
  title   = {Capacity, quasi-local mass, and singular fill-ins},
  journal = {Journal f{\"u}r die reine und angewandte Mathematik},
  volume  = {768},
  year    = {2020},
  pages   = {55--92},
  doi     = {10.1515/crelle-2019-0040}
}

@article{LeeLesourdUnger2023,
  author  = {Lee, Dan and Lesourd, Martin and Unger, Ryan},
  title   = {Density and positive mass theorems for incomplete manifolds},
  journal = {Calculus of Variations and Partial Differential Equations},
  volume  = {62},
  year    = {2023},
  number  = {194},
  doi     = {10.1007/s00526-023-02516-4}
}

@article{LiMantoulidis2019,
  author  = {Li, Chao and Mantoulidis, Christos},
  title   = {Positive scalar curvature with skeleton singularities},
  journal = {Mathematische Annalen},
  volume  = {374},
  number  = {1--2},
  year    = {2019},
  pages   = {99--131},
  doi     = {10.1007/s00208-018-1753-1}
}

@article{Kazaras2024,
  author  = {Kazaras, Demetre},
  title   = {Desingularizing positive scalar curvature 4-manifolds},
  journal = {Mathematische Annalen},
  volume  = {390},
  year    = {2024},
  pages   = {4951--4972},
  doi     = {10.1007/s00208-024-02829-5}
}

@unpublished{CecchiniFrenckZeidler2024,
	author = {Simone Cecchini and Georg Frenck and Rudolf Zeidler},
	eprint = {2407.20163},
	note = {arXiv:2407.20163},
	title = {Positive scalar curvature with point singularities},
	url = {https://arxiv.org/pdf/2407.20163.pdf},
	year = {2024}}

@article{DaiSunWang2025,
  author  = {Dai, Xianzhe and Sun, Yukai and Wang, Changliang},
  title   = {The positive mass theorem for asymptotically flat manifolds with isolated conical singularities},
  journal = {Science China Mathematics},
  volume  = {68},
  year    = {2025},
  pages   = {1671--1686},
  doi     = {10.1007/s11425-024-2325-6}
}

@article{LesourdUngerYau2024,
  author  = {Lesourd, Martin and Unger, Ryan and Yau, Shing-Tung},
  title   = {The positive mass theorem with arbitrary ends},
  journal = {Journal of Differential Geometry},
  volume  = {128},
  number  = {1},
  year    = {2024},
  pages   = {257--293},
  doi     = {10.4310/jdg/1721075263}
}

@unpublished{ChodoshMantoulidisSchulzeWang2025,
	author = {Otis Chodosh and Christos Mantoulidis and Felix Schulze and Zhihan Wang},
	eprint = {2506.12852},
	note = {arXiv:2506.12852},
	title = {Generic regularity for minimizing hypersurfaces in dimension 11},
	url = {https://arxiv.org/pdf/2506.12852.pdf},
	year = {2025}}

@unpublished{BrendleWang2026,
	author = {Simon Brendle and Yipeng Wang},
	eprint = {2604.08473},
	note = {arXiv:2604.08473},
	title = {A dimension descent scheme for the positive mass theorem in arbitrary dimension},
	url = {https://arxiv.org/pdf/2604.08473.pdf},
	year = {2026}}

@unpublished{Lohkamp2016HigherDimensionalPMTII,
	author = {Joachim Lohkamp},
	eprint = {1612.07505},
	note = {arXiv:1612.07505},
	title = {The Higher Dimensional Positive Mass Theorem {II}},
	url = {https://arxiv.org/pdf/1612.07505.pdf},
	year = {2016}}

@unpublished{Lohkamp2006,
	author = {Joachim Lohkamp},
	eprint = {math/0608795},
	note = {arXiv:math/0608795},
	title = {The Higher Dimensional Positive Mass Theorem {I}},
	url = {https://arxiv.org/pdf/math/0608795.pdf},
	year = {2006}}

@unpublished{ChodoshMantoulidisSchulze2023,
	author = {Otis Chodosh and Christos Mantoulidis and Felix Schulze},
	eprint = {2302.02253},
	note = {arXiv:2302.02253},
	title = {Generic regularity for minimizing hypersurfaces in dimensions 9 and 10},
	url = {https://arxiv.org/pdf/2302.02253.pdf},
	year = {2023}}

@unpublished{BurkhardtGuim2024Smoothing,
	author = {Paula Burkhardt-Guim},
	eprint = {2406.04564},
	note = {arXiv:2406.04564},
	title = {Smoothing {$L^\infty$} {R}iemannian metrics with nonnegative scalar curvature outside of a singular set},
	url = {https://arxiv.org/pdf/2406.04564.pdf},
	year = {2024}}

@article{ShiTam2002,
  author  = {Shi, Yuguang and Tam, Luen-Fai},
  title   = {Positive mass theorem and the boundary behaviors of compact manifolds with nonnegative scalar curvature},
  journal = {Journal of Differential Geometry},
  volume  = {62},
  number  = {1},
  year    = {2002},
  pages   = {79--125},
  doi     = {10.4310/jdg/1090425530}
}

@article{Smale1993,
  author  = {Smale, Nathan},
  title   = {Generic regularity of homologically area minimizing hypersurfaces in eight-dimensional manifolds},
  journal = {Communications in Analysis and Geometry},
  volume  = {1},
  number  = {2},
  year    = {1993},
  pages   = {217--228},
  doi     = {10.4310/CAG.1993.v1.n2.a2}
}

@unpublished{BrendleWang2026Spacetime,
	author = {Simon Brendle and Yipeng Wang},
	eprint = {2604.18561},
	note = {arXiv:2604.18561},
	title = {On the spacetime positive energy theorem in arbitrary dimension},
	url = {https://arxiv.org/pdf/2604.18561.pdf},
	year = {2026}}

@unpublished{HirschKhuriLesourdZhang2026,
	author = {Sven Hirsch and Marcus Khuri and Martin Lesourd and Yiyue Zhang},
	eprint = {2604.24746},
	note = {arXiv:2604.24746},
	title = {The Hyperboloidal and Spacetime Positive Mass Theorem in All Dimensions},
	url = {https://arxiv.org/pdf/2604.24746.pdf},
	year = {2026}}

@unpublished{Tsang2026InitialDataArbitraryEnds,
	author = {Tin-Yau Tsang},
	eprint = {2604.26978},
	note = {arXiv:2604.26978},
	title = {Positive mass theorem for initial data sets with arbitrary ends},
	url = {https://arxiv.org/pdf/2604.26978.pdf},
	year = {2026}}

@article{Bray2001,
  author  = {Bray, Hubert},
  title   = {Proof of the {Riemannian} {Penrose} inequality using the positive mass theorem},
  journal = {Journal of Differential Geometry},
  volume  = {59},
  number  = {2},
  year    = {2001},
  pages   = {177--267},
  doi     = {10.4310/jdg/1090349428}
}

@article{ChengLeeTam2022,
  author  = {Cheng, Man-Chuen and Lee, Man-Chun and Tam, Luen-Fai},
  title   = {Singular metrics with negative scalar curvature},
  journal = {International Journal of Mathematics},
  volume  = {33},
  number  = {7},
  year    = {2022},
  pages   = {2250047},
  doi     = {10.1142/S0129167X22500471}
}

@unpublished{DaiSunWang2024ConicalPSC,
	author = {Xianzhe Dai and Yukai Sun and Changliang Wang},
	eprint = {2412.02941},
	note = {arXiv:2412.02941},
	title = {Positive scalar curvature and isolated conical singularity},
	url = {https://arxiv.org/pdf/2412.02941.pdf},
	year = {2024}}

@article{ShiWangWeiZhu2021FillIn,
  author        = {Shi, Yuguang and Wang, Wenlong and Wei, Guodong and Zhu, Jintian},
  title         = {On the fill-in of nonnegative scalar curvature metrics},
  journal       = {Mathematische Annalen},
  volume        = {379},
  number        = {1--2},
  pages         = {235--270},
  year          = {2021},
  doi           = {10.1007/s00208-020-02087-1},
  eprint        = {1907.12173},
  archivePrefix = {arXiv},
  primaryClass  = {math.DG}
}

@article{Zhu2021Width,
  author        = {Zhu, Jintian},
  title         = {Width estimate and doubly warped product},
  journal       = {Transactions of the American Mathematical Society},
  volume        = {374},
  number        = {2},
  pages         = {1497--1511},
  year          = {2021},
  doi           = {10.1090/tran/8263},
  eprint        = {2003.01315},
  archivePrefix = {arXiv},
  primaryClass  = {math.DG}
}

@article{NaberValtorta2019Varifolds,
  author  = {Naber, Aaron and Valtorta, Daniele},
  title   = {The singular structure and regularity of stationary and minimizing varifolds},
  journal = {Annals of Mathematics},
  volume  = {190},
  number  = {2},
  pages   = {413--565},
  year    = {2019},
  eprint  = {1505.03428},
  archivePrefix = {arXiv},
  primaryClass = {math.DG}
}

@article{Eichmair2009PlateauMOTS,
  author        = {Eichmair, Michael},
  title         = {The Plateau problem for marginally outer trapped surfaces},
  journal       = {Journal of Differential Geometry},
  volume        = {83},
  number        = {3},
  pages         = {551--583},
  year          = {2009},
  doi           = {10.4310/jdg/1264601035},
  eprint        = {0711.4139},
  archivePrefix = {arXiv},
  primaryClass  = {math.DG}
}

@article{CheegerNaber2013QuantitativeStratification,
  author  = {Cheeger, Jeff and Naber, Aaron},
  title   = {Quantitative stratification and the regularity of harmonic maps and minimal currents},
  journal = {Communications on Pure and Applied Mathematics},
  volume  = {66},
  number  = {6},
  pages   = {965--990},
  year    = {2013},
  doi     = {10.1002/cpa.21446}
}

@article{BartnikChrusciel2005DiracBoundary,
  author  = {Bartnik, Robert and Chru{\'s}ciel, Piotr },
  title   = {Boundary value problems for {Dirac}-type equations},
  journal = {Journal f{\"u}r die reine und angewandte Mathematik},
  volume  = {579},
  pages   = {13--73},
  year    = {2005},
  doi     = {10.1515/crll.2005.2005.579.13},
  eprint  = {math/0307278},
  archivePrefix = {arXiv}
}

@unpublished{MazurowskiYao2026ContinuousMetrics,
	author = {Liam Mazurowski and Xuan Yao},
	eprint = {2606.19123},
	note = {arXiv:2606.19123},
	title = {A Positive Mass Theorem for Continuous Metrics},
	url = {https://arxiv.org/pdf/2606.19123.pdf},
	year = {2026}}

@book{SchoenYau1994Lectures,
  author    = {Schoen, Richard and Yau, Shing-Tung},
  title     = {Lectures on Differential Geometry},
  publisher = {International Press},
  address   = {Boston},
  year      = {1994}
}

@book{HurewiczWallman1941DimensionTheory,
  author    = {Hurewicz, Witold and Wallman, Henry},
  title     = {Dimension Theory},
  series    = {Princeton Mathematical Series},
  volume    = {4},
  publisher = {Princeton University Press},
  address   = {Princeton},
  year      = {1941}
}

@book{Engelking1978DimensionTheory,
  author    = {Engelking, Ryszard},
  title     = {Dimension Theory},
  publisher = {North-Holland},
  address   = {Amsterdam},
  year      = {1978}
}

@article {LeeTam,
    AUTHOR = {Lee, Man-Chun and Tam, Luen-Fai},
     TITLE = {Continuous metrics and a conjecture of {S}choen},
   JOURNAL = {Trans. Amer. Math. Soc.},
  FJOURNAL = {Transactions of the American Mathematical Society},
    VOLUME = {378},
      YEAR = {2025},
    NUMBER = {3},
     PAGES = {1531--1550},
      ISSN = {0002-9947,1088-6850},
   MRCLASS = {53C18 (53C20)},
  MRNUMBER = {4866343},
MRREVIEWER = {Huaiyu\ Zhang},
       DOI = {10.1090/tran/9379},
       URL = {https://doi-org.proxy.library.stonybrook.edu/10.1090/tran/9379},
}

@article {JSZ,
    AUTHOR = {Jiang, Wenshuai and Sheng, Weimin and Zhang, Huaiyu},
     TITLE = {Removable singularity of positive mass theorem with continuous
              metrics},
   JOURNAL = {Math. Z.},
  FJOURNAL = {Mathematische Zeitschrift},
    VOLUME = {302},
      YEAR = {2022},
    NUMBER = {2},
     PAGES = {839--874},
      ISSN = {0025-5874,1432-1823},
   MRCLASS = {53C20},
  MRNUMBER = {4480213},
MRREVIEWER = {David\ J.\ Wraith},
       DOI = {10.1007/s00209-022-03081-w},
       URL = {https://doi-org.proxy.library.stonybrook.edu/10.1007/s00209-022-03081-w},
}

\end{document}